\documentclass[12pt]{article}
\usepackage{amsmath,amssymb,amsthm,mathtools,mathrsfs,latexsym}
\usepackage{graphicx}
\usepackage{xcolor}
\usepackage{tikz}
\usepackage[linkcolor=blue,colorlinks=true,urlcolor=red,bookmarksopen=true]{hyperref}
\usepackage{newclude}
\usepackage{imakeidx}
\makeindex
\usepackage{indentfirst}
\usepackage{bm}

\allowdisplaybreaks

\title{On the large-time behavior of strong solutions to the generalized compressible Navier-Stokes-Korteweg system in 2D and 3D for arbitrarily large initial data}
\date{}
\author{
	\bf\large Xiangdi Huang$^{a}$, Weili Meng$^a$\thanks{E-mail addresses: xdhuang@amss.ac.cn (X. Huang); mengweili@amss.ac.cn (W. Meng).}\\
	\small a. State Key Laboratory of Mathematical Sciences, Academy of Mathematics and Systems Science,\\
	\small Chinese Academy of Sciences, Beijing 100190, China;\\
}

\newcommand{\divg}{{\rm  div}}
\newcommand{\norm}[1]{\left\Vert#1\right\Vert}
\let\div\relax
\DeclareMathOperator*{\div}{div}

\usepackage{hyperref}
\hypersetup{hypertex=true,
	colorlinks=true,
	linkcolor=blue,
	anchorcolor=blue,
	citecolor=blue}

\usepackage{color}
\newtheorem{thm}{Theorem}[section]
\newtheorem{corl}{Corollary}[section]
\newtheorem{lema}{Lemma}[section]
\newtheorem{prop}{Proposition}[section]
\newtheorem{defi}{Definition}[section]
\newtheorem{rmk}{Remark}[section]

\allowdisplaybreaks

\begin{document}
	\maketitle %显示标题
	\begin{abstract}
		In this paper, we establish the global existence and large-time behavior of strong solutions for the two- and three-dimensional periodic compressible Navier--Stokes--Korteweg system with arbitrarily large initial data $(\rho_0,u_0)\in H^3\times H^2$. The viscosity coefficients satisfy the BD relation $\mu(\rho)=\nu\rho^\alpha$ and $\lambda(\rho)=2\nu(\alpha-1)\rho^\alpha$, while the capillarity coefficient is given by $\kappa(\rho)=\varepsilon^2\alpha^2\rho^{2\alpha-3}$. For the case $\alpha<1$, we first enlarge the admissible parameter range for the global existence of strong solutions established in Gu--Huang--Meng--Zhou [arXiv:2603.11762 (2026)] by exploiting the doubly parabolic structure of the density--effective velocity system. We then develop a time-discretization strategy to establish uniform integrability estimates for the effective velocity, yielding a uniform upper bound for the density. Furthermore, we introduce a novel bootstrap argument to successively improve these integrability estimates, which leads to a uniform positive lower bound for the density. Finally, we derive global-in-time higher-order estimates and prove the large-time behavior
		\[
		\left\|\rho(t)-\frac{1}{|\mathbb{T}^N|}\int_{\mathbb T^N}\rho_0 dx\right\|_{H^3}
		+\|\nabla u(t)\|_{H^1}
		\longrightarrow0,
		\qquad
		t\to\infty,
		\]
		without imposing any smallness assumption on the initial data.
		For the critical case $\alpha=1$, we improve the admissible parameter range established in Huang--Meng--Zhang [arXiv:2602.00455 (2026)] and establish a uniform upper bound for the density.\\[4mm]
	\end{abstract}
	
	\tableofcontents 
	
	\section{Introduction}
	In this paper, we study the following compressible Navier--Stokes--Korteweg system:
	\begin{equation}
		\label{Equ1}
		\left\{
		\begin{array}{l}
			\rho_t+\div(\rho u)=0,\\
			(\rho u)_t+\divg(\rho u\otimes u)+\nabla P=\div(2\mu(\rho) \mathbb{D}u)+\nabla(\lambda(\rho)\divg u)+\div\mathbb{K}.
		\end{array}
		\right.
	\end{equation}
	Here, $t$ denotes the time variable and $x$ the spatial variable. The unknowns $\rho$ and $u=(u_1,\dots,u_N)^{\top}$ $(N=2,3)$ stand for the fluid density and velocity, respectively, while $P=\rho^\gamma$ $(\gamma\ge 1)$ is the pressure. Moreover, $\mathbb{D}u$ denotes the deformation tensor defined by
	\begin{align*}
		\mathbb{D} u=\frac{\nabla u+(\nabla u)^{\top}}{2}.
	\end{align*}
	The density-dependent viscosity coefficients $\mu(\rho)$ and $\lambda(\rho)$ are assumed to satisfy the physical conditions
	\begin{align*}
		\mu(\rho)\ge 0, \quad 2\mu(\rho)+N\lambda(\rho)\ge 0.
	\end{align*}
	The capillarity tensor $\mathbb{K}$ is defined by
	\begin{align*}
		\mathbb{K}=\Big(\rho \operatorname{div}(\kappa(\rho)\nabla\rho)-\frac{\rho\kappa'(\rho)-\kappa(\rho)}{2}|\nabla\rho|^2\Big)\mathbb{I}
		-\kappa(\rho)\nabla\rho\otimes\nabla\rho,
	\end{align*}
	where $\mathbb{I}$ denotes the identity matrix. Accordingly, its divergence can be expressed as
	\begin{align*}
		\operatorname{div}\mathbb{K}=\nabla\left(\rho\kappa(\rho)\Delta \rho+\frac{\kappa(\rho)+\rho\kappa'(\rho)}{2}|\nabla\rho|^2\right)-\operatorname{div}(\kappa(\rho)\nabla\rho\otimes\nabla\rho).
	\end{align*}

    The mathematical formulation of the compressible Navier--Stokes--Korteweg system can be traced back to Korteweg's work \cite{Korteweg}, in which capillarity effects were modeled through a stress tensor involving density gradients. Later, Dunn and Serrin \cite{Dunn-Serrin} generalized this formulation within the framework of capillary fluid theory.

	In the sequel, we prescribe the viscosity coefficients $\mu(\rho)$ and $\lambda(\rho)$, as well as the capillarity coefficient $\kappa(\rho)$, by
	\begin{align}\label{vis coff}
		\mu(\rho)=\nu\rho^\alpha,\quad\lambda(\rho)=2\nu(\alpha-1)\rho^\alpha,\quad\kappa(\rho)=\varepsilon^2\alpha^2\rho^{2\alpha-3},\quad  \nu\ge\varepsilon>0,
	\end{align}
	where the positive constants $\nu$ and $\varepsilon$ characterize the strengths of viscosity and capillarity, respectively. The system \eqref{Equ1}--\eqref{vis coff} is considered on the periodic domain $\mathbb{T}^N=\mathbb{R}^N/\mathbb{Z}^N$ and is complemented by the initial conditions
	\begin{align}\label{ini data}
		\rho(x,0)=\rho_0(x),\quad u(x,0)=u_0(x),\quad x\in\mathbb{T}^N,
	\end{align}
	where, without loss of generality, we assume that both $\rho_0$ and $u_0$ are $1$-periodic in each spatial variable.
	
	When $\kappa(\rho)=0$, the compressible Navier--Stokes--Korteweg system reduces to the compressible Navier--Stokes system with density-dependent viscosity. The regularity, global existence, and large-time behavior of solutions under different viscosity laws have attracted considerable attention over the past decades. A particularly important class consists of viscosity coefficients satisfying the Bresch--Desjardins (BD) relation
	\begin{align*}
		\lambda(\rho)=2\rho\mu'(\rho)-2\mu(\rho).
	\end{align*}
	This relation leads to a fundamental mathematical entropy estimate, established in \cite{Bresch-Desjardins-Lin, Bresch-Desjardins}, which provides additional regularity for the density and has played a crucial role in the development of the weak solution theory. We refer to Mellet and Vasseur \cite{Mellet-Vasseur} for the stability of weak solution sequences, Guo, Jiu and Xin \cite{Guo-Jiu-Xin} for the existence of finite-energy spherically symmetric weak solutions, Vasseur and Yu \cite{Vasseur-Yu} for the viscous Saint--Venant system, Li and Xin \cite{Li-Xin} for power-law viscosity coefficients, Bresch, Vasseur and Yu \cite{Bresch-Vasseur-Yu} for general physically symmetric viscous stress tensors, and Huang, Meng and Zhang \cite{Huang-Meng-Zhang-111} for the existence of spherically symmetric weak solutions with higher regularity and the disappearance of vacuum states.
	
	Despite these developments, whether arbitrarily large smooth initial data generate global smooth solutions in higher dimensions for compressible Navier--Stokes systems with BD viscosity remains a longstanding open problem. Only recently has substantial progress been made for power-law viscosity coefficients satisfying
    \begin{align*}
    \mu(\rho)=\nu\rho^\alpha, \quad \lambda(\rho)=2\nu(\alpha-1)\rho^\alpha, \quad \nu>0,
    \end{align*}
    under spherical symmetry. Zhang \cite{Zhang} first established the global existence of spherically symmetric classical solutions away from vacuum in bounded domains in higher dimensions for $\alpha<1$, under additional restrictions on the parameters.  Subsequently, Huang, Meng and Zhang \cite{Huang-Meng-Zhang-111} improved the admissible parameter range in \cite{Zhang} and established uniform density bounds together with the large-time behavior of classical solutions. Their result also covers the viscous Saint--Venant model $\alpha=1$.  Gu and Huang \cite{Gu-Huang} further extended the result for $\alpha=1$ in \cite{Huang-Meng-Zhang-111} to the three-dimensional case and generalized the framework to non-isentropic flows involving entropy transport.

When the viscosity and capillarity coefficients are given by
\begin{align*}
	\mu(\rho)=\nu\rho,\qquad
	\lambda(\rho)=0,\qquad
	\kappa(\rho)=\frac{\varepsilon^2}{\rho},
\end{align*}
where $\nu,\varepsilon>0$ are constants, the compressible Navier--Stokes--Korteweg system reduces to the compressible quantum Navier--Stokes system. In the one-dimensional case, Jüngel \cite{Jüngel} established the global existence of smooth solutions away from vacuum under the condition $\varepsilon=\nu$. Subsequently, Chen and Zhao \cite{Chen-Zhao} extended this result to the case $\varepsilon<\nu$ and further proved the large-time behavior of solutions. In higher dimensions, Huang, Meng and Zhang \cite{Huang-Meng-Zhang-2026} established the global existence of strong solutions away from vacuum to the two- and three-dimensional quantum Navier--Stokes system on the periodic domain under the assumption $\varepsilon=\nu$. Subsequently, Gu, Huang and Lei \cite{Huang-Gu-Lei} extended this result to the Cauchy problem under the same assumption $\varepsilon=\nu$. It should be pointed out that, in higher dimensions, the above global existence results are only available in the intermediary regime $\varepsilon=\nu$. The main reason is that the derivation of the positive lower bound for the density relies crucially on an $L^\infty$ estimate for the effective velocity, which can be closed only when the effective momentum equation is governed solely by the standard Laplace-type diffusion $\operatorname{div}(\rho\nabla v)$, namely, when $\varepsilon=\nu$.

For the general compressible Navier--Stokes--Korteweg system, various existence results for strong and smooth solutions have also been established. In the one-dimensional setting, Chen, Chai, Dong and Zhao \cite{Chen-Chai-Dong-Zhao} established the global existence of classical solutions away from vacuum for the Cauchy problem with large initial data under certain power-law assumptions on the viscosity and capillarity coefficients. Subsequently, Chen \cite{Chen} obtained uniform upper and lower bounds for the density and further established the large-time behavior of solutions. Burtea and Haspot \cite{Burtea-Haspot-40} also established the global existence of strong solutions under certain restrictions on the viscosity and capillarity coefficients by introducing an effective velocity transformation. In higher dimensions, Hattori and Li \cite{Hattori-Li, Hattori-Li2} established the local well-posedness and global existence of strong solutions with small initial data for the Cauchy problem, assuming that the initial density is bounded away from vacuum. This local existence theory was subsequently extended to initial--boundary value problems by Kotschote \cite{Kotschote}. Danchin and Desjardins \cite{Danchin-Desjardins}, as well as Haspot \cite{Haspot}, established the global existence of strong solutions with small initial data in the framework of Besov spaces. For studies on weak solutions to general compressible Navier--Stokes--Korteweg systems, including the quantum Navier--Stokes system, we refer to \cite{Antonelli-Bresch-Spirito, Antonelli-Spirito,Burtea-Haspot, Dong, Germain-LeFloch, Haspot2,  Jiang, Jüngel2, Lacroix-Violet-Vasseur, Lu-Zhang-Zhong, Tsyganov} and the references therein.

However, the global existence of strong/smooth solutions with large initial data for the multidimensional compressible Navier--Stokes--Korteweg system remains largely open. A recent breakthrough was made by Gu, Huang, Meng and Zhou \cite{Gu-Huang-Meng-Zhou-2026}, who established the global existence of strong solutions away from vacuum on the periodic domain under the assumptions \eqref{vis coff}, $\alpha<1$, and suitable restrictions on the parameters. Subsequently, Huang, Lei and Zhou \cite{Huang-Lei-Zhou} extended this result to the Cauchy problem under the same assumptions. In contrast to the corresponding results for the quantum Navier--Stokes system \cite{Huang-Meng-Zhang-2026, Huang-Gu-Lei}, these results cover not only the intermediary regime $\varepsilon=\nu$, but also the parabolic regime $\varepsilon<\nu$. This difference stems from the fact that, in the non-critical case $\alpha<1$, the derivation of the positive lower bound for the density relies only on an $L^p$ estimate for the effective velocity. As a consequence, the effective momentum equation is allowed to contain non-Laplace-type diffusion terms, provided that they can be absorbed by the standard Laplace-type diffusion.

Although \cite{Huang-Meng-Zhang-2026} and  \cite{Gu-Huang-Meng-Zhou-2026} established the global existence of strong solutions with arbitrarily large initial data for the multidimensional periodic quantum Navier--Stokes system and the generalized compressible Navier--Stokes--Korteweg system, respectively, it remains open whether these solutions admit uniform-in-time upper and lower bounds for the density and whether they converge to the equilibrium state as time tends to infinity. The present paper provides an affirmative answer to these questions. Moreover, we further enlarge the admissible range of the adiabatic exponent $\gamma$ in both results.

Before presenting our main results, we first specify some notation and conventions that will be used throughout the paper. We write
\begin{align*}
\int fdx=\int_{\mathbb{T}^N} fdx.
\end{align*}
For $1\le s\le\infty$ and $k\in\mathbb{N}^+$, the Lebesgue and Sobolev spaces are denoted by
\begin{align*}
L^s=L^s(\mathbb{T}^N),\quad W^{k,s}=W^{k,s}(\mathbb{T}^N),\quad H^k=W^{k,2}.
\end{align*}
Throughout the proofs, the generic constant $C$ may change from one occurrence to another. Its dependence will be clear from the context and may involve, in particular, suitable norms of the initial data and the parameters of the system appearing in the corresponding propositions or lemmas.

	We begin by defining the global strong solution to the initial-value problem \eqref{Equ1}--\eqref{ini data}.
	\begin{defi}
		A pair $(\rho,u)$ is called a global strong solution to the initial-value problem \eqref{Equ1}--\eqref{ini data} if, for any $0<T<\infty$ and $(x,t)\in \mathbb{T}^N\times[0,T]$, 
		\begin{align}
			\left\{
			\begin{array}{l}
				(C(T))^{-1}\leq \rho(x,t)\leq C(T),\\
				\rho\in C([0,T];H^3)\cap L^2(0,T;H^4),\ \rho_t\in C([0,T];H^1)\cap L^2(0,T;H^2),\\
				u\in C([0,T];H^2)\cap L^2(0,T;H^3),\ u_t\in L^\infty(0,T;L^2)\cap L^2(0,T;H^1),
			\end{array}
			\right.
		\end{align}
		where $C(T)$ is a positive constant depending on $N,\alpha,\gamma,\nu,\varepsilon,\rho_0,u_0$, and $T$.
	\end{defi}
	
	The first theorem establishes the global existence and asymptotic behavior of strong solutions for the case $\alpha<1$.
	\begin{thm}\label{Thm 1.1}
		Let $N\in \{2,3\}$ and $\beta=\sqrt{1-\frac{\varepsilon^2}{\nu^2}}$. Assume that $(\alpha,\gamma,\beta)$ satisfies
		\begin{align}
			&N=2,\quad\alpha\in \Big(\frac{\sqrt{5}-1}{2},1\Big), \quad\gamma\in[1,\infty),\quad \beta\in [0,\beta_{2}^+(\alpha));\label{2d gamma}\\
			&N=3,\quad \alpha\in\left(
			\frac{9\sqrt3-4\sqrt2}{9\sqrt3-2\sqrt2},\,1
			\right),\quad\gamma\in \Big[1,\frac{111\alpha^2-122\alpha+33}{21\alpha-13}\Big),\quad \beta\in[0,\beta_{3}^+(\alpha)), \label{3d gamma}
		\end{align}
		where $\beta_{N}^+(x)\in(0,1)$ denotes the unique positive root of the following quadratic equation:
		\begin{align*}
			\frac{2x}{1-x}=\frac{4(1-\beta_{N}^+(x))(Nx-N+1)}{(\beta_{N}^+(x)+\sqrt{N}(1-x)(1-\beta_{N}^+(x)))^2}.
		\end{align*}
		Moreover, assume that the initial data $(\rho_0,u_0)$ satisfy 
		\begin{align}\label{equ ini data}
			0<\underline{\rho_0}\le \rho_0\leq \overline{\rho_0},\quad\rho_0\in H^3,\quad u_0\in H^2,
		\end{align}
		where $\underline{\rho_0}$ and $\overline{\rho_0}$ are positive constants. Then the initial-value problem \eqref{Equ1}--\eqref{ini data} admits a unique global strong solution $(\rho,u)$ satisfying
		\begin{align*}
			C^{-1}\leq \rho(x,t)\leq C,\quad \forall (x,t)\in \mathbb{T}^N\times[0,\infty),
		\end{align*}
		where $C>0$ depends only on $N, \alpha,\gamma,\nu,\varepsilon,\rho_0$ and $u_0$. Furthermore, the following asymptotic behavior holds:
		\begin{align*}
			\lim\limits_{t\to\infty}\big(\|\rho(t)-\rho_s\|_{H^3}+\|\nabla u(t)\|_{H^1}\big)=0,
		\end{align*}
		where $\rho_s=\int \rho_0 dx$. 
	\end{thm}
	
	\begin{rmk}
		Gu--Huang--Meng--Zhou \cite{Gu-Huang-Meng-Zhou-2026} established the global existence of strong solutions to the initial-value problem \eqref{Equ1}--\eqref{ini data} for the compressible Navier--Stokes--Korteweg system under the parameter condition \eqref{2d gamma} in two dimensions and under the condition obtained by replacing the upper bound $\frac{111\alpha^2-122\alpha+33}{21\alpha-13}$ for $\gamma$ in \eqref{3d gamma} with $\frac{15\alpha-7}{3}$ in three dimensions. Since, for every $\alpha>\frac23$,
		\[
		\frac{111\alpha^2-122\alpha+33}{21\alpha-13}
		>
		\frac{15\alpha-7}{3}.
		\]
		Therefore, Theorem~\ref{Thm 1.1} enlarges the admissible range of $\gamma$ and improves the three-dimensional global existence result obtained in \cite{Gu-Huang-Meng-Zhou-2026}.
	\end{rmk}
    \begin{rmk}
        One of the main features of Theorem~\ref{Thm 1.1} is the establishment of uniform-in-time upper and lower bounds for the density, which implies that the density remains uniformly bounded away from both vacuum and concentration on $\mathbb{T}^N\times[0,\infty)$. These uniform bounds provide a key ingredient for the analysis of the large-time behavior of solutions. Moreover, by exploiting the fact that the original system \eqref{Equ1} can be reformulated, through the so-called effective velocity transformation, as a coupled doubly parabolic system, we further prove that the $H^3$-norm of $\rho-\rho_s$ and the $H^1$-norm of $\nabla u$ both converge to zero as time tends to infinity, without any smallness assumption on the initial data.
    \end{rmk}

	The next theorem establishes the global existence of strong solutions for the case $\alpha=1$.
	\begin{thm}\label{Thm 1.1'}
		Let $N\in \{2,3\}$ and $\beta=\sqrt{1-\frac{\varepsilon^2}{\nu^2}}$. Assume that $(\alpha, \gamma,\beta)$ satisfies
		\begin{align}
			&N=2,\quad\alpha=1, \quad\gamma\in[1,\infty),\quad \beta=0;\label{2d gamma'}\\
			&N=3,\quad\alpha=1, \quad\gamma\in \Big[1,\frac{11}{4}\Big),\quad \beta=0. \label{3d gamma'}
		\end{align}
		Moreover, assume that the initial data $(\rho_0,u_0)$ satisfy	\eqref{equ ini data}.  Then the initial-value problem \eqref{Equ1}–\eqref{ini data} admits a unique global strong solution $(\rho,u)$ satisfying
		\begin{align*}
			\rho(x,t)\leq C,\quad \forall (x,t)\in \mathbb{T}^N\times[0,\infty),
		\end{align*}
		where $C>0$ depends only on $N, \alpha,\gamma,\nu,\varepsilon,\rho_0$, and $u_0$. 
	\end{thm}
	\begin{rmk}
		Huang--Meng--Zhang \cite{Huang-Meng-Zhang-2026} established
		the global existence of strong solutions to the
		initial-value problem \eqref{Equ1}--\eqref{ini data} for
		the compressible Navier--Stokes--Korteweg system under the
		parameter condition \eqref{2d gamma'} in two dimensions and
		under the condition obtained by replacing the upper bound
		$\frac{11}{4}$ for $\gamma$ in \eqref{3d gamma'} with
		$\frac{7}{3}$ in three dimensions.
		Furthermore, Gu--Huang--Lei
		\cite{Huang-Gu-Lei} established the global existence
		of strong solutions to the corresponding Cauchy problem and
		improved the admissible range of $\gamma$ in three dimensions
		from $[1,\frac{7}{3})$ to $[1,\frac{8}{3})$.
		Since $\frac{11}{4}>\frac83,$ Theorem~\ref{Thm 1.1'} further enlarges the admissible range
		of $\gamma$ in three dimensions and therefore improves the
		previous global existence results for the periodic problem.
	\end{rmk}
	\begin{rmk}
		Theorem~\ref{Thm 1.1'} establishes a uniform-in-time upper
		bound for the density of global strong solutions.
		Nevertheless, establishing a uniform positive lower bound for the density remains a challenging problem. We leave this question for future work.
	\end{rmk}
\begin{rmk}
It would be interesting to apply the strategy developed in this paper to the corresponding Cauchy problem to establish uniform estimates and large-time behavior for the solutions obtained in \cite{Huang-Gu-Lei,Huang-Lei-Zhou}, as well as to enlarge the admissible range of the parameters. We leave these questions for future work.
\end{rmk}

The remainder of this paper is organized as follows. Section 2 presents the main ideas and strategies of our analysis. Section 3 is devoted to collecting the existing global existence results for strong solutions to the periodic compressible Navier--Stokes--Korteweg system. Sections 4--8 focus on establishing uniform-in-time lower-order estimates in the three-dimensional case. More precisely, Section 4 establishes the energy estimate (the $L^2$ estimate for the effective velocity), Section 5 derives the subcritical estimate (the $L^2$--$L^{3-}$ estimate for the effective velocity), Section 6 establishes the critical estimate (the $L^3$ estimate for the effective velocity), Section 7 develops the first supercritical estimate (the $L^{3+}$--$L^6$ estimate for the effective velocity) and obtains the uniform upper bound for the density in the case $\alpha\leq1$, and Section 8 establishes the second supercritical estimate (the $L^{6+}$ estimate for the effective velocity) and the uniform positive lower bound for the density in the case $\alpha<1$. Section 9 is devoted to deriving uniform-in-time higher-order estimates for the density and the effective velocity, as well as the large-time behavior of solutions in the three-dimensional case with $\alpha<1$. Finally, Section 10 completes the proof of the main theorems.

\section{Main Strategy}

In this section, we outline the main ideas of our analysis in the three-dimensional case from four perspectives. The first point is to exploit the doubly parabolic coupling structure of the density--effective velocity system, which enables us to enlarge the admissible range of the adiabatic exponent $\gamma$. The second point is to develop a time-discretization strategy that provides uniform higher-integrability estimates for the effective velocity and, in turn, yields a uniform-in-time upper bound for the density. The third point is to introduce a novel bootstrap procedure to further improve the integrability estimates of the effective velocity. This procedure allows us to derive a uniform positive lower bound for the density when $\alpha<1$, while also explaining why the same approach cannot be applied to obtain such a lower bound in the critical case $\alpha=1$. The final point is to further exploit the doubly parabolic coupling structure of the density--effective velocity system to establish global higher-order estimates and the large-time behavior of solutions based on the lower-order estimates, without imposing any smallness assumption on the initial data.

	\subsection{Improvement of the admissible range of $\gamma$}
	
	Under the structural assumptions \eqref{vis coff}, the standard energy estimate for the compressible NSK system yields the bound
	\[
	\|\nabla\rho^{\alpha-\frac12}\|_{L^\infty(0,T;L^2)},
	\]
	which originates from the capillarity term. In the one-dimensional case, the Sobolev embedding $H^1(\mathbb{T}^1)\hookrightarrow L^\infty(\mathbb{T}^1)$ immediately converts this estimate into pointwise control of the density, making it possible to derive upper and lower bounds at the energy level, as demonstrated in \cite{Chen-Zhao}. In higher dimensions, however, the embedding $H^1(\mathbb{T}^N)\hookrightarrow L^\infty(\mathbb{T}^N)$ is no longer valid. Consequently, obtaining  bounds for the density becomes one of the principal difficulties in the analysis of the multidimensional compressible NSK system.
	
One of the key ideas in \cite{Gu-Huang-Meng-Zhou-2026} for overcoming this difficulty is to introduce the effective velocity
\begin{align}\label{eff vel}
v:=u+c\alpha\rho^{\alpha-2}\nabla\rho,
\end{align}
where $c:=\nu+\sqrt{\nu^2-\varepsilon^2}$.
 Then the pair $(\rho,v)$ satisfies the following density--effective velocity system, whose derivation can be found in Lemma 8.2 of \cite{Gu-Huang-Meng-Zhou-2026}:
	\begin{equation}\label{Equ_0'}
		\left\{
		\begin{array}{l}
			\rho_t + \div(\rho v) - c\Delta\rho^\alpha = 0,\\[4pt]
			\begin{aligned}
				\rho v_t + \rho u\cdot\nabla v + \nabla P
				&= \nu\divg(\rho^\alpha\nabla v)
				+(\nu-c)\divg(\rho^\alpha(\nabla v)^\top)\\
				&\quad+(\alpha-1)(2\nu-c)\nabla(\rho^\alpha\divg v).
			\end{aligned}
		\end{array}
		\right.
	\end{equation}
	Based on this reformulation, they developed a modified Nash--Moser iteration scheme to establish bounds for the density. More precisely, they showed that if one can obtain the estimate
	\[
	\|\rho^{1/p}v\|_{L^\infty(0,T;L^p)}<\infty,\qquad p>N,
	\]
	then the nonlinear parabolic equation satisfied by the density can be iterated to yield an  upper bound for $\rho$. Therefore, the key ingredient in deriving the density upper bound is the density-weighted estimate of the effective velocity in the space $L^{N+}$.
	
	In the three-dimensional case, the density-weighted $L^{3+}$ estimate for the effective velocity is obtained by exploiting the parabolic structure of the effective velocity equation. The main difficulty lies in controlling the pressure term, which relies on the a priori estimate $\|\rho\|_{L^\infty(0,T;L^{6\alpha-3})}$. Consequently, this approach requires $\gamma\in\big[1,\frac{15\alpha-7}{3}\big)$.
	
	The improvement of the admissible range of $\gamma$ is based on the key observation that the gap between the energy estimate (the weighted $L^2$ estimate for the effective velocity) and the critical estimate (the weighted $L^3$ estimate) in \cite{Gu-Huang-Meng-Zhou-2026} can be bridged through a bootstrap argument. Indeed, when $\gamma\in\big[\frac{15\alpha-7}{3},\,\frac{111\alpha^2-122\alpha+33}{21\alpha-13}\big),$
	the energy estimate still provides the bound
	\[
	\|\rho\|_{L^\infty(0,T;L^{6\alpha-3})}<\infty.
	\]
	Using this estimate together with the parabolic equation satisfied by the effective velocity, namely \eqref{Equ_0'}$_2$, we first derive the weighted estimate
	\[
	\|\rho^{1/(s_1+2)}v\|_{L^\infty(0,T;L^{s_1+2})}<\infty,
	\]
	where $s_1\in(0,1]$. Substituting this estimate into the density equation \eqref{Equ_0'}$_1$ then yields the improved density estimate
	\[
	\|\rho\|_{L^\infty(0,T;L^{q_1})}<\infty,
	\]
	where $q_1>6\alpha-3$. Repeating this procedure, we obtain increasingly stronger integrability estimates for both the effective velocity and the density. After finitely many iterations, the integrability exponent of the effective velocity exceeds the critical exponent $3$, allowing the modified Nash--Moser iteration developed in \cite{Gu-Huang-Meng-Zhou-2026} to be applied.
	
	The crucial point that enables this bootstrap procedure to start is precisely the strict improvement $q_1>6\alpha-3$, which is guaranteed by the condition $\gamma<\frac{111\alpha^2-122\alpha+33}{21\alpha-13}.$ In contrast, when
	$
	\gamma=\frac{111\alpha^2-122\alpha+33}{21\alpha-13},
	$
	the first bootstrap step only reproduces the same density integrability, namely
	$
	q_1=6\alpha-3.
	$
	As a consequence, no gain in integrability is achieved, and the bootstrap procedure cannot proceed beyond its initial step.
	
	\subsection{Uniform upper bound for the density}
	To derive a uniform upper bound for the density, we adopt a time-discretization strategy inspired by \cite{Chen-Zhao}. The starting point is the basic energy estimate, which yields
	\[
	\int_0^\infty\Big(\norm{\nabla\rho^{\frac{3\alpha-2}{4}}}_{L^4}^4
	+\norm{\rho^{\frac{\alpha}{2}}\nabla v}_{L^2}^2\Big)\,dt\leq C.
	\]
	This estimate enables us to partition the time interval $[0,\infty)$ into a sequence of subintervals
	\[
	[0,t_0],\ [t_0,t_1],\ \ldots,\ [t_n,t_{n+1}],\ \ldots,
	\]
	whose lengths are uniformly bounded. Furthermore, the density and the effective velocity satisfy essentially the same quantitative bounds at the initial time of every subinterval. More precisely, there exists a sufficiently large constant $T_0$, depending only on the initial data and the physical parameters, such that for every $n\in\mathbb{N}$, one can find $t_n\in[nT_0,(n+1)T_0]$ satisfying
	\begin{align*}
		&\frac12\underline{\rho_0}\leq \rho(x,t_n)\leq \frac32\overline{\rho_0},
		\qquad x\in\mathbb{T}^3,\\
		&\int_{\mathbb{T}^3}|\nabla v(x,t_n)|^2\,dx\le1.
	\end{align*}
	
	Based on this time-discretization strategy, we establish the uniform estimate
	\[
	\|\rho^{1/p}v\|_{L^\infty(0,\infty;L^p)}<\infty,
	\qquad 2<p\le6.
	\]
	The proof relies on applying Grönwall's inequality on each time subinterval. The resulting estimate is uniform with respect to the subintervals for two reasons. First, the lengths of the subintervals are uniformly bounded, ensuring that the contribution of the forcing terms in Grönwall's inequality remains uniformly controlled. Second, the density and the effective velocity satisfy the same initial bounds at the left endpoint of every subinterval, so that the initial value in Grönwall's inequality can be estimated uniformly. Indeed,
	\begin{align*}
		\int \rho|v|^{p}(t_n)\,dx
		&\le \|\rho(t_n)\|_{L^\infty}\|v(t_n)\|_{L^p}^p\\
		&\le C\|\rho(t_n)\|_{L^\infty}\|v(t_n)\|_{H^1}^p\\
		&\le C\|\rho(t_n)\|_{L^\infty}
		\Big(
		\|\rho^{-1/2}(t_n)\|_{L^\infty}
		\|\rho^{1/2}v(t_n)\|_{L^2}
		+\|\nabla v(t_n)\|_{L^2}
		\Big)^p,
	\end{align*}
	which is bounded independently of $n$. It is worth emphasizing that this argument only provides density-weighted \(L^p\) estimates for the effective velocity with \(2<p\le6\), rather than for all \(p>2\). Nevertheless, the endpoint \(p=6\)  already exceeds the critical threshold $3$ required by the modified Nash--Moser iteration, and is therefore sufficient to derive the uniform upper bound for the density.
	
	Finally, we apply the Nash--Moser iteration developed in \cite{Gu-Huang-Meng-Zhou-2026} on each time subinterval. Since the effective velocity enjoys a uniform-in-time supercritical estimate, the lengths of the subintervals are uniformly bounded, and the density has the same quantitative bounds at the initial time of every subinterval, the iteration produces a density upper bound that is independent of the particular subinterval. Consequently, the resulting upper bound is uniform in time.
	
	\subsection{Uniform positive lower bound for the density}
	
	To establish a uniform positive lower bound for the density, we introduce the reciprocal variable $\tau:=\rho^{-1}$. A direct computation shows that $\tau$ satisfies
	\begin{align*}
		\partial_t \tau
		-c\alpha\divg(\tau^{1-\alpha}\nabla \tau)
		+2c\alpha \tau^{-\alpha}|\nabla \tau|^2
		+v\cdot\nabla \tau
		-\tau\divg v
		=0.
	\end{align*}
	It was shown in \cite{Gu-Huang-Meng-Zhou-2026} that if
	\[
	\|\rho^{1/p}v\|_{L^\infty(0,T;L^{p})}<\infty,
	\qquad
	p>\frac{2}{1-\alpha},
	\]
	then one can apply the modified Nash--Moser iteration in a similar manner to the above equation and thereby derive a positive lower bound for the density. Therefore, under our framework, the main task is to establish uniform-in-time higher integrability estimates for the effective velocity.
	
	A fundamental difference between the upper- and lower-bound arguments is that the latter requires substantially stronger integrability of the effective velocity. Indeed, since $\alpha\in(2/3,1]$, we have $\frac{2}{1-\alpha}>6.$ Hence, the uniform estimate established in the previous subsection,
	\[
	\|\rho^{1/6}v\|_{L^\infty(0,\infty;L^{6})}<\infty,
	\]
	although already sufficient for deriving the density upper bound, is no longer adequate for the lower-bound argument. The central challenge is therefore to further improve the uniform integrability of the effective velocity.
	
	We overcome this difficulty by developing a bootstrap argument. The bootstrap procedure is organized inductively. Assume that, for some $k\in\mathbb N$, the following $(k-1)$-th level estimates have already been established:
	\begin{align*}
		\sup_{0\le t<\infty}\int \rho |v|^{6\times3^{k-1}}dx
		&<\infty,
		&&\text{(uniform estimate)},\\
		\int_{\tau_1}^{\tau_2}\int
		\rho^\alpha|v|^{6\times3^{k-1}-2}
		|\nabla v|^2\,dxdt
		&=o(\tau_2-\tau_1),
		\quad
		\tau_2-\tau_1\to\infty,
		&&\text{(dissipation estimate).}
	\end{align*}
	The objective is to show that both estimates remain valid after replacing $k-1$ by $k$. The induction step consists of two parts. The first part is to upgrade the uniform estimate. Combining the $(k-1)$-th level dissipation estimate with the available estimates for the density derivatives yields
	\begin{align*}
		\int_{\tau_1}^{\tau_2}
		\int
		\Big(
		|\nabla\rho|^4
		+
		\rho^\alpha
		|v|^{6\times3^{k-1}-2}
		|\nabla v|^2
		\Big)
		dxdt
		=o(\tau_2-\tau_1),
		\quad
		\tau_2-\tau_1\to\infty.
	\end{align*}
	Consequently, there exists a sufficiently large constant $\widetilde T_0$ such that, for every $n\in\mathbb N$, one can choose
	$\widetilde t_n\in[n\widetilde T_0,(n+1)\widetilde T_0]$
	satisfying
	\begin{align*}
		&\frac12\underline{\rho_0}\leq \rho(x,\widetilde t_n)\leq \frac32\overline{\rho_0},
		\qquad x\in\mathbb T^3,\\
		&\int_{\mathbb T^3}|v|^{6\times3^{k-1}-2}|\nabla v(\widetilde t_n)|^2\,dx\le1.
	\end{align*}
	These uniformly controlled initial states, together with the Sobolev embedding theorem, provide a uniform bound for the initial value of the $L^{6\times3^k}$-energy functional in the Gr\"onwall argument, namely, $\int \rho(\widetilde t_n)|v(\widetilde t_n)|^{6\times3^k}\,dx$ is uniformly bounded with respect to $n$. Applying Gr\"onwall's inequality on each interval $[\widetilde t_n,\widetilde t_{n+1}]$ then yields the $k$-th level uniform estimate
	\[
	\sup_{0\le t<\infty}\int \rho|v|^{6\times3^k}dx<\infty .
	\]
	The second part is to upgrade the corresponding dissipation estimate. Testing the effective velocity equation with $|v|^{6\times3^k-2}v$ yields
	\[
	\frac{1}{6\times3^k}\frac{d}{dt}\int \rho|v|^{6\times3^k}dx
	+\delta_{6\times3^k-2}
	\int \rho^\alpha|v|^{6\times3^k-2}|\nabla v|^2dx
	\le
	C\int \rho^\gamma|v|^{6\times3^k-2}|\nabla v|dx.
	\]
	The pressure term on the right-hand side is controlled by combining the newly established $k$-th level uniform estimate, the $(k-1)$-th level dissipation estimate, and the uniform upper bound of the density. Since the $(k-1)$-th level dissipation estimate already exhibits sublinear growth in time, the same property carries over to the time integral of the pressure term. Consequently, the $k$-th level dissipation estimate also satisfies
	\[
	\int_{\tau_1}^{\tau_2}
	\int
	\rho^\alpha|v|^{6\times3^k-2}|\nabla v|^2\,dxdt
	=o(\tau_2-\tau_1),
	\qquad
	\tau_2-\tau_1\to\infty.
	\]
	The above argument provides the induction step of the bootstrap procedure. Since the basic energy estimate
	\[
	\sup_{0\le t<\infty}\int \rho|v|^2dx
	+\int_0^\infty\int\rho^\alpha|\nabla v|^2dxdt
	<\infty
	\]
	serves as the starting point, the bootstrap can be carried out recursively.
	
	For the case $\alpha<1$, repeating the above bootstrap argument upgrades both the uniform integrability and the corresponding dissipation estimates. The iteration terminates after finitely many steps, thereby reaching the integrability required by the modified Nash--Moser iteration to establish a uniform positive lower bound for the density.
	
	The critical case $\alpha=1$ requires a different argument. In this case, the modified Nash--Moser iteration can only be closed if the effective velocity admits an $L^\infty$ estimate. Although such an estimate is currently unavailable, \cite{Huang-Meng-Zhang-2026} established the following logarithmic control:
	\[
	\sup_{0\le t\le T}\|v\|_{L^\infty}
	\le
	C_T\sqrt{\log\bigl(e+\|\rho^{-1}\|_{L^\infty(0,T;L^\infty)}\bigr)},
	\]
	which is sufficient to derive a positive lower bound for the density on every finite time interval. Within our framework, the bootstrap argument developed above remains valid for $\alpha=1$, but it cannot be closed at the $L^\infty$ level of the effective velocity. The main obstacle is that we are unable to construct a time-discretization strategy that provides a uniform $L^\infty$ bound for the effective velocity at the initial time of each discrete interval. As a result, the constant $C_T$ in the logarithmic estimate inevitably depends on the time horizon $T$, and our approach does not yield a time-uniform positive lower bound for the density in the critical case $\alpha=1$.
	
	\subsection{Higher-order estimates and the large-time behavior $(\alpha<1)$}
	When $\alpha<1$, the uniform upper and lower bounds for the density established in the previous subsections, together with the uniform higher integrability estimate for the effective velocity, guarantee that, for some $q>3$,
	\begin{align*}
		\sup_{0\le t<\infty}
		\Big(
		\|\rho\|_{L^\infty}
		&+\|\rho^{-1}\|_{L^\infty}
		+\|v\|_{L^q}
		+\|\nabla\rho\|_{L^2}
		\Big)\\
		&+\int_0^\infty
		\Big(
		\|\nabla v\|_{L^2}^2
		+\|\nabla\rho\|_{L^2}^2
		+\|\nabla\rho\|_{L^4}^4
		+\|\nabla^2\rho\|_{L^2}^2
		\Big)\,dt
		<\infty.
	\end{align*}
	
	These estimates form the starting point of the higher-order analysis. A key feature of the density--effective velocity formulation is that it transforms the original system into a coupled parabolic system. Consequently, after differentiating the equations, all nonlinear terms appearing on the right-hand side of the resulting differential inequalities can be controlled by quantities that are already known to be integrable in time. This makes it possible to propagate higher-order regularity globally in time. By adapting the high-order energy method developed in \cite{Gu-Huang-Meng-Zhou-2026} to incorporate the uniform density bounds obtained in the present paper, we successively derive global first- and second-order energy estimates for both the density and the effective velocity.

	An additional consequence of this higher-order analysis is the large-time behavior of solutions. More precisely, the higher-order energy inequalities, together with the corresponding estimates, imply
	\[
	\lim_{t\to\infty}
	\big(
	\|\rho(t)-\rho_s\|_{H^3}
	+\|\nabla v(t)\|_{H^1}
	\big)
	=0.
	\]
	Since the effective velocity differs from the original velocity by a density-gradient correction, the above convergence immediately yields
	\[
	\lim_{t\to\infty}
	\big(
	\|\rho(t)-\rho_s\|_{H^3}
	+\|\nabla u(t)\|_{H^1}
	\big)
	=0.
	\]
	
	It is worth emphasizing that the above convergence is obtained without imposing any additional assumptions on the initial data, in particular, without any smallness condition.

	\section{Global Existence of Strong Solutions}
	In this section, we collect the global existence results for
	strong solutions to the initial-value problem of the
	compressible Navier--Stokes--Korteweg system established in
	\cite{Gu-Huang-Meng-Zhou-2026, Huang-Meng-Zhang-2026}.
	\begin{lema}\label{Lem glo sol}
		Assume that one of the following parameter conditions is satisfied:
		\eqref{2d gamma}, \eqref{2d gamma'},
		\begin{align}\label{3d gamma old}
			N=3,\quad \alpha\in\left(
			\frac{9\sqrt3-4\sqrt2}{9\sqrt3-2\sqrt2},\,1
			\right),\quad
			\gamma\in \Big[1,\frac{15\alpha-7}{3}\Big),\quad
			\beta\in[0,\beta_{3}^+(\alpha)),
		\end{align}
		or
		\begin{align}\label{3d gamma old'}
			N=3,\quad
			\alpha=1,\quad
			\gamma\in \Big[1,\frac{8}{3}\Big),\quad
			\beta=0.
		\end{align} 
		Moreover, assume that the initial data
		$(\rho_0,u_0)$ satisfy \eqref{equ ini data}. Then the
		initial-value problem \eqref{Equ1}--\eqref{ini data}
		admits a unique global strong solution $(\rho,u)$.
	\end{lema}
	\begin{rmk}
		The global existence result under the condition
		\eqref{3d gamma old'} follows from a combination of
		\cite{Huang-Meng-Zhang-2026} and
		\cite{Huang-Gu-Lei}. More precisely,
		\cite{Huang-Meng-Zhang-2026} established the global
		existence of strong solutions with the upper bound
		$\frac{8}{3}$ for $\gamma$ in \eqref{3d gamma old'}
		replaced by $\frac{7}{3}$. For
		$\gamma\in\big[\frac{7}{3},\frac{8}{3}\big)$,
		the argument developed in \cite{Huang-Gu-Lei} for
		the Cauchy problem can be adapted  to the
		periodic setting, which yields the corresponding global
		existence result.
	\end{rmk}
	
	For the sake of simplicity, throughout the remainder of this paper, we restrict our attention to the three-dimensional case. Accordingly, \textbf{unless otherwise specified, we shall always assume that the parameters satisfy either \eqref{3d gamma} or \eqref{3d gamma'}.} The corresponding arguments in the two-dimensional case are simpler and can be obtained in an analogous manner.
	
	For $\gamma\in\left[1,\frac{15\alpha-7}{3}\right)$, the global existence of strong solutions to the initial-value problem \eqref{Equ1}--\eqref{ini data} follows from Lemma \ref{Lem glo sol}. For $\gamma\in\left[\frac{15\alpha-7}{3},
	\frac{111\alpha^2-122\alpha+33}{21\alpha-13}\right)$, the global existence will follows from the a priori estimates established in the present paper. Indeed, once these estimates are available, the existence proof can be completed by following essentially the same argument as in \cite{Gu-Huang-Meng-Zhou-2026}, and is therefore omitted.
	
	In what follows, $(\rho,u)$ denotes the unique global strong solution to the initial-value problem \eqref{Equ1}--\eqref{ini data} on $\mathbb{T}^3\times[0,\infty)$ under either of the parameter conditions \eqref{3d gamma} or \eqref{3d gamma'}, and the effective velocity $v$ is defined by \eqref{eff vel}.

	\section{Energy Estimates: $L^2$ Estimates for $v$}
	We begin with the standard energy estimate. Define the initial energy by
	\begin{align*}
		E_0 = \int \left( \frac{1}{2}\rho_0|u_0|^2 +\pi_+(\rho_0) + \frac{2\varepsilon^2\alpha^2}{(2\alpha-1)^2}|\nabla\rho_0^{\alpha-\frac{1}{2}}|^2 \right) dx,
	\end{align*}
	where
	\begin{equation*}
		\pi_+(\rho_0)=\left\{
		\begin{array}{ll}
			\frac{1}{\gamma-1}\rho_0^\gamma,\quad &\text{ if }\gamma>1,\\
			\rho_0\log_+\rho_0, \quad &\text{ if }\gamma=1.
		\end{array}
		\right.
	\end{equation*}
	The following standard energy estimate was established in Proposition 4.1 of \cite{Gu-Huang-Meng-Zhou-2026} for the case $\alpha<1$, and in Proposition 2.1 of \cite{Huang-Meng-Zhang-2026} for the case $\alpha=1$.
	\begin{prop}
		There exists a constant $C>0$, depending only on $\alpha,\gamma,\nu,\varepsilon,$ and $E_0$, such that
		\begin{align}\label{2-1}
			\sup_{0\leq t<\infty}\int \left(\rho|u|^2+\rho^\gamma+|\nabla\rho^{\alpha-\frac{1}{2}}|^2\right)dx+\int_0^\infty\int  \rho^\alpha|\mathbb{D}u|^2dxdt\leq C.
		\end{align}
	\end{prop}
	
	We next derive the $L^2$ estimate for the effective velocity. The proof is given in Proposition 4.2 of \cite{Gu-Huang-Meng-Zhou-2026} for  the case $\alpha<1$, and in Proposition 2.2 of \cite{Huang-Meng-Zhang-2026} for  the case $\alpha=1$.
	\begin{prop}
		There exists a constant $C>0$, depending only on $\alpha,\gamma,\nu,\varepsilon,$ and $E_0$, such that
		\begin{align}\label{2-2}
			\sup_{0\leq t<\infty}\int \rho|v|^2dx+\int_0^\infty\int \rho^{\gamma+\alpha-3}|\nabla\rho|^2dxdt+\int_0^\infty\int \rho^\alpha|\nabla v|^2dxdt\leq C.
		\end{align}
	\end{prop}
	
	The dissipation estimates in \eqref{2-1} and \eqref{2-2} further imply higher-order dissipation estimates involving the density. For $\alpha<1$, the proof is given in Proposition 4.3 of \cite{Gu-Huang-Meng-Zhou-2026}. For $\alpha=1$, the proof can be obtained by a slight modification of the argument for the case $\alpha<1$; see also Corollaries 3.2 and 3.3 of \cite{Huang-Gu-Lei}.
	\begin{prop}
		There exists a constant $C>0$, depending only on $\alpha,\gamma,\nu,\varepsilon,$ and $E_0$, such that
		\begin{align}\label{nabla rho 4}
			\int_0^\infty\int\big|\nabla\rho^{\frac{3\alpha-2}{4}}\big|^4dxdt+\int_0^\infty\int|\nabla^2\rho^{\frac{3}{2}\alpha-1}|^2dxdt\leq C.
		\end{align}
	\end{prop}
   \begin{rmk}
	Since the viscosity coefficients satisfy the BD algebraic relation, the physical conditions $2\mu(\rho)+3\lambda(\rho)\ge 0$ are fulfilled only when $\alpha\geq \frac{2}{3}$. Moreover, the lower bound of $\alpha$ in \eqref{3d gamma} satisfies $\frac{9\sqrt{3}-4\sqrt{2}}{9\sqrt{3}-2\sqrt{2}}>\frac{2}{3}$. Therefore, the power of $\rho$ in \eqref{nabla rho 4} does not vanish.
\end{rmk}
	
	\section{Subcritical Estimates: $L^p$ Estimates for $v$ ($2<p<3$)}
	In this section, we fully exploit the doubly parabolic structure of system \eqref{Equ_0'} to simultaneously improve the integrability of the effective velocity and the density in the subcritical regime.
	
	Our analysis begins with a time-discretization argument, which plays a crucial role in deriving uniform-in-time estimates for the solution. Inspired by \cite{Chen-Zhao}, we decompose the time axis $[0,\infty)$ into a sequence of subintervals
	\[
	[0,t_0],\ [t_0,t_1],\ \ldots,\ [t_n,t_{n+1}],\ \ldots,
	\]
	whose lengths are uniformly bounded, and such that the density and the effective velocity satisfy comparable estimates at the initial time of each subinterval. The following lemma provides precisely such a decomposition.	\begin{lema}\label{Lem 3.1}
		There exists a positive constant $T_0$, depending only on
		$\gamma,\alpha,\nu,\varepsilon,E_0,\overline{\rho_0}$, and
		$\underline{\rho_0}$, such that, for every
		$n\in\mathbb{N}$, one can find
		$t_n\in[nT_0,(n+1)T_0]$ satisfying
		\begin{align}
			\frac{1}{2}\underline{\rho_0}\leq \rho(x,t_n)\leq \frac{3}{2}\overline{\rho_0},\quad\text{ for all }x\in \mathbb{T}^3 ,\label{prop 3.1 0}\\
			\int|\nabla v|^2(x,t_n)dx\leq \frac{\hat{C}}{T_0}\Big(\frac{2}{\underline{\rho_0}}\Big)^\alpha, \label{prop 3.1 1}
		\end{align}
		where $\hat{C}$ is a positive constant depending only on
		$\gamma,\alpha,\nu,\varepsilon,$ and $E_0$.
	\end{lema}
	\begin{proof}
		By \eqref{2-2} and \eqref{nabla rho 4}, there exists a positive constant $\hat{C}$, depending only on $\gamma,\alpha,\nu,\varepsilon,$ and $E_0$, such that
		\begin{align*}
			\int_0^\infty\Big(\norm{\nabla\rho^{\frac{3\alpha-2}{4}}}_{L^4}^4+\norm{\rho^{\frac{\alpha}{2}}\nabla v}_{L^2}^2\Big)dt\leq \hat{C}.
		\end{align*}
		Choose a positive constant $T_0$, whose value will be specified later in \eqref{T_0 def'}. In particular, for every $n\in\mathbb{N}$, one has
		\begin{align*}
			\int_{nT_0}^{(n+1)T_0}\Big(\norm{\nabla\rho^{\frac{3\alpha-2}{4}}}_{L^4}^4+\norm{\rho^{\frac{\alpha}{2}}\nabla v}_{L^2}^2\Big)dt\leq \hat{C},
		\end{align*}
		which, by the mean value theorem for integrals, implies the existence of a point
		$t_n\in[nT_0,(n+1)T_0]$ such that
		\begin{align}\label{prop 4.1 2}
			\norm{\nabla\rho^{\frac{3\alpha-2}{4}}(t_n)}_{L^4}^4+\norm{\rho^{\frac{\alpha}{2}}\nabla v(t_n)}_{L^2}^2\leq \frac{\hat{C}}{T_0}.
		\end{align}
		It follows from the Sobolev embedding theorem and H\"older's inequality that
		\begin{align}\label{100-8}
			\begin{split}
			\norm{\rho^{\frac{3\alpha-2}{4}}(t_n)}_{L^\infty}&\leq C_{e}\norm{\rho^{\frac{3\alpha-2}{4}}(t_n)}_{W^{1,4}}\\
			&\leq C_{e}\norm{\rho^{\frac{3\alpha-2}{4}}(t_n)}_{L^{\frac{4}{3\alpha-2}}}\norm{1_{\mathbb{T}^3}}_{L^{\frac{4}{3-3\alpha}}}+C_e\norm{\nabla\rho^{\frac{3\alpha-2}{4}}(t_n)}_{L^4}\\
			&\leq C_e\rho_s^{\frac{3\alpha-2}{4}}+C_e\Big(\frac{\hat{C}}{T_0}\Big)^{\frac{1}{4}},
				\end{split}
		\end{align}
		where $C_e>0$ denotes the Sobolev embedding constant. When $\alpha=1$, the exponent $\frac{4}{3-3\alpha}$ is interpreted as $+\infty$. Since $\alpha\in\left(\frac23,1\right]$, it follows from \eqref{prop 4.1 2} and \eqref{100-8}  that
		\begin{align*}
			\norm{\nabla\rho(t_n)}_{L^4}&\leq C_\alpha\norm{\rho^{\frac{6-3\alpha}{4}}(t_n)}_{L^\infty}\norm{\nabla\rho^{\frac{3\alpha-2}{4}}(t_n)}_{L^4}\\
			&\leq C_\alpha C_e^{\frac{6-3\alpha}{3\alpha-2}}\Big(\rho_s^{\frac{3\alpha-2}{4}}+\Big(\frac{\hat{C}}{T_0}\Big)^{\frac{1}{4}}\Big)^{\frac{6-3\alpha}{3\alpha-2}}\Big(\frac{\hat{C}}{T_0}\Big)^{\frac{1}{4}},
		\end{align*}
		where $C_\alpha>0$ is a constant depending only on $\alpha$. Applying the Sobolev embedding theorem once again, we obtain
		\begin{align}\label{3-1}
			\norm{\rho(t_n)-\rho_s}_{L^\infty}\leq \tilde{C}_e\norm{\nabla\rho(t_n)}_{L^4}\leq  \tilde{C}_e C_\alpha C_e^{\frac{6-3\alpha}{3\alpha-2}}\Big(\rho_s^{\frac{3\alpha-2}{4}}+\Big(\frac{\hat{C}}{T_0}\Big)^{\frac{1}{4}}\Big)^{\frac{6-3\alpha}{3\alpha-2}}\Big(\frac{\hat{C}}{T_0}\Big)^{\frac{1}{4}},
		\end{align}
		where $\tilde{C}_e>0$ denotes the Sobolev embedding constant. Here we have also used the conservation of mass $\int \rho(t_n)\,dx=\rho_s$. Choose $T_0$ sufficiently large so that
		\begin{align}\label{T_0 def'}
			\tilde{C}_e C_\alpha C_e^{\frac{6-3\alpha}{3\alpha-2}}\Big(\rho_s^{\frac{3\alpha-2}{4}}+\Big(\frac{\hat{C}}{T_0}\Big)^{\frac{1}{4}}\Big)^{\frac{6-3\alpha}{3\alpha-2}}\Big(\frac{\hat{C}}{T_0}\Big)^{\frac{1}{4}}\leq \frac{1}{2}\underline{\rho_0}.
		\end{align}
		Since $\underline{\rho_0}\le \rho_s\le \overline{\rho_0}$,
		it follows from \eqref{3-1} and \eqref{T_0 def'} that
		\begin{align*}
			\rho(x,t_n)\leq \frac{3}{2}\overline{\rho_0},\quad
			\rho(x,t_n)\geq \frac{1}{2}\underline{\rho_0},\quad\forall x\in \mathbb{T}^3,
		\end{align*}
		which proves \eqref{prop 3.1 0}.
		
		It remains to verify \eqref{prop 3.1 1}. This follows immediately from \eqref{prop 3.1 0} and \eqref{prop 4.1 2}. Therefore, the proof of Lemma \ref{Lem 3.1} is complete.
	\end{proof}
	
	The following proposition shows that the integrability of the density can be improved by exploiting the integrability of the effective velocity through the parabolic equation \eqref{Equ_0'}$_1$ satisfied by the density.
	\begin{prop}\label{Prop 3.2}
		Assume that
		\begin{align}\label{3-1.2}
			\sup_{0\leq t<\infty}\int \rho |v|^{s+2}dx\leq C_{v,s}.
		\end{align}
		\textbf{\rm{(i)}} If $s\in\left(\frac{5\alpha-3}{8\alpha-5},1\right)$, then there exists a constant $C>0$, depending only on
		$\gamma,\alpha,\nu,\varepsilon,E_0,\overline{\rho_0},\underline{\rho_0},s,$ and $C_{v,s}$, such that
		\begin{align*}
			\sup_{0\leq t<\infty}\int \rho^q dx\leq C,
		\end{align*}
		where $q=2-2\alpha+\frac{3\alpha-2}{1-s}$.\\
		\textbf{\rm{(ii)}} If $s=1$, then for every $q\in(1,\infty)$, there exists a constant $C>0$, depending only on
		$\gamma,\alpha,\nu,\varepsilon,E_0,\overline{\rho_0},\underline{\rho_0},C_{v,1},$ and $q$, such that
		\begin{align*}
			\sup_{0\leq t<\infty}\int \rho^q dx\leq C.
		\end{align*}
	\end{prop}
	\begin{rmk}
    For every
    $s\in\left(\frac{5\alpha-3}{8\alpha-5},1\right]$, the exponent satisfies
    \[
    2-2\alpha+\frac{3\alpha-2}{1-s}>6\alpha-3,
    \]
    where the left-hand side is interpreted as $+\infty$ when $s=1$.
    \end{rmk}
	\begin{proof}
		Assume that
        $s\in\Big(\frac{5\alpha-3}{8\alpha-5},1\Big]$.
        Since $\alpha\in(\frac23,1]$, we have
        $s>\frac23$.
		
		For any $r>0$, multiplying   \eqref{Equ_0'}$_1$ by $\rho^r$, integrating the resulting equation over $\mathbb{T}^3$, and applying integration by parts, we obtain
		\begin{align*}
			\frac{1}{r+1}\frac{d}{dt}\int \rho^{r+1}dx+c\alpha r\int \rho^{\alpha+r-2}|\nabla\rho|^2dx=r\int \rho^{r}v\cdot \nabla\rho dx.
		\end{align*}
		The right-hand side can be estimated by applying Young's inequality, Hölder's inequality, and \eqref{3-1.2} as follows:
		\begin{align*}
			\begin{split}
				&\quad r\int \rho^{r}v\cdot \nabla\rho dx\\
				&\leq \frac{c\alpha r}{2}\int \rho^{\alpha+r-2}|\nabla\rho|^2dx+C\int \rho^{-\alpha+r+2}|v|^2dx\\
				&\leq \frac{c\alpha r}{2}\int \rho^{\alpha+r-2}|\nabla\rho|^2dx+C\Big(\int \rho|v|^{s+2}dx\Big)^{\frac{2}{s+2}}\Big(\int \rho^{m_r\frac{s+2}{s}}dx\Big)^{\frac{s}{s+2}}\\
				&\leq \frac{c\alpha r}{2}\int \rho^{\alpha+r-2}|\nabla\rho|^2dx+C\Big(\int \rho^{m_r\frac{s+2}{s}}dx\Big)^{\frac{s}{s+2}},
			\end{split}
		\end{align*}
		where $m_r=-\alpha+r+2-\frac{2}{s+2}$. Consequently,
		\begin{align}\label{3-2}
			\begin{split}
				\frac{1}{r+1}\frac{d}{dt}\int \rho^{r+1}dx+\frac{c\alpha r}{2}\int \rho^{\alpha+r-2}|\nabla\rho|^2dx\leq C\Big(\int \rho^{m_r\frac{s+2}{s}}dx\Big)^{\frac{s}{s+2}}.
			\end{split}
		\end{align}
        
		We distinguish two cases.
		\medskip
		
		\noindent\textbf{Case I: $s\in\Big(\frac{5\alpha-3}{8\alpha-5},1\Big).$}
		Let $r=1-2\alpha+\frac{3\alpha-2}{1-s}.$
		By the Sobolev embedding theorem,
		\begin{align}\label{3-3}
			\begin{split}
				C\Big(\int \rho^{m_r\frac{s+2}{s}}dx\Big)^{\frac{s}{s+2}}
				=C\norm{\rho^{\frac{\alpha+r}{2}}}_{L^{\frac{2m_r}{\alpha+r}\frac{s+2}{s}}}^{\frac{2m_r}{\alpha+r}}
				\leq
				C\norm{\rho^{\frac{\alpha+r}{2}}}_{H^1}^{\frac{2m_r}{\alpha+r}},
			\end{split}
		\end{align}
		where we have used the fact that
		\begin{align*}
			1\leq \frac{2m_r}{\alpha+r}\frac{s+2}{s}\leq 6,
		\end{align*}
		which follows from $s\in\left(\frac23,1\right)$ and the definition of $r$. Indeed,
		\begin{align*}
			\frac{2m_r}{\alpha+r}\frac{s+2}{s}\ge \frac{2(-\alpha+r+\frac{5}{4})}{\alpha+r}\times3\ge 1
			&\Leftarrow 7\alpha\leq 5r+\frac{15}{2},\\
			\frac{2m_r}{\alpha+r}\frac{s+2}{s}\leq 6
			&\Leftrightarrow r\leq 1-2\alpha+\frac{3\alpha-2}{1-s}.
		\end{align*}
		
		\noindent\textbf{Case II: $s=1.$}
		Let $r>0$ be arbitrary. By the Sobolev embedding theorem, we obtain
		\begin{align}\label{3-4}
			\begin{split}
				C\Big(\int \rho^{3m_r}dx\Big)^{\frac{1}{3}}=C\norm{\rho^{\frac{\alpha+r}{2}}}_{L^{\frac{6m_r}{\alpha+r}}}^{\frac{2m_r}{\alpha+r}}\leq C\norm{\rho^{\frac{\alpha+r}{2}}}_{H^1}^{\frac{2m_r}{\alpha+r}},
			\end{split}
		\end{align}
		where we have used the fact that
		\begin{align*}
			1\leq \frac{6m_r}{\alpha+r}\leq 6,
		\end{align*}
		which follows from
		\begin{align*}
			\frac{6m_r}{\alpha+r}= \frac{6(-\alpha+r+\frac{4}{3})}{\alpha+r}\ge 1&\Leftarrow 7\alpha\leq 5r+8, \\
			\frac{6m_r}{\alpha+r}= \frac{6(-\alpha+r+\frac{4}{3})}{\alpha+r}\leq 6&\Leftarrow \alpha>\frac{2}{3}.
		\end{align*}
		
		Combining the two cases, we deduce from \eqref{3-2} together with
		\eqref{3-3} or \eqref{3-4} that
		\begin{align}\label{3-5}
			\begin{split}
				\frac{1}{r+1}\frac{d}{dt}\int \rho^{r+1}dx+\frac{c\alpha r}{2}\int \rho^{\alpha+r-2}|\nabla\rho|^2dx\leq C\norm{\rho^{\frac{\alpha+r}{2}}}_{H^1}^{\frac{2m_r}{\alpha+r}}. 
			\end{split}
		\end{align}
		Since $\frac{s+1}{s+2}\le \frac{2}{3}<\alpha$, we have
		$\frac{2m_r}{\alpha+r}<2$. Applying Young's inequality and Hölder's inequality, we obtain
		\begin{align*}
			C\norm{\rho^{\frac{\alpha+r}{2}}}_{H^1}^{\frac{2m_r}{\alpha+r}}&\leq \delta\norm{\rho^{\frac{\alpha+r}{2}}}_{H^1}^{2}+C_\delta\\
			&\leq \delta\norm{\nabla \rho^{\frac{\alpha+r}{2}}}_{L^2}^2+C_\delta\int \rho^{r+1}dx+C_\delta.
		\end{align*}
		Therefore, by choosing $\delta>0$ sufficiently small, we get
		\begin{align*}
			\frac{d}{dt}\int \rho^{r+1}dx\leq C_1\int \rho^{r+1}dx+C_1,
		\end{align*}
		where $C_1>0$ is a positive constant depending only on the quantities stated in Proposition \ref{Prop 3.2}. Set $t_{-1}=0$. For each
		$n\in\mathbb{N}\cup\{-1\}$, applying Grönwall's inequality on
		$[t_n,t_{n+1}]$ and using \eqref{prop 3.1 0}, we obtain
		\begin{align*}
			\sup_{t\in [t_n,t_{n+1}]}\int \rho^{r+1}(t)dx&\le e^{C_1(t_{n+1}-t_n)}\int \rho^{r+1}(t_n)dx+e^{C_1(t_{n+1}-t_n)}-1\\
			&\leq e^{2C_1T_0}\Big(\frac{3}{2}\overline{\rho_0}\Big)^{r+1}+e^{2C_1T_0}-1.
		\end{align*}
		Since the bound on the right-hand side of the above inequality is independent of $n$, we obtain
		\begin{align*}
			\sup_{0\leq t<\infty}\int \rho^{r+1}dx\leq C,
		\end{align*}
		which, together with the definitions of $r$ in the two cases above, completes the proof of Proposition \ref{Prop 3.2}.
	\end{proof}
	
	Next, we show that the integrability of the effective velocity can be improved by exploiting the higher integrability of the density through the parabolic system \eqref{Equ_0'}$_2$ satisfied by the effective velocity. To this end, we introduce the following notation. Let
	\begin{align*}
		n_{3,\alpha,\beta}= \left\{
		\begin{array}{lc}
			\frac{4(1-\beta)(3\alpha-2)}{(\beta+\sqrt{3}(1-\alpha)(1-\beta))^2}, & \text{ if }\alpha<1, \\
			\infty, &\text{ if }\alpha=1. 
		\end{array}
		\right.
	\end{align*}
	When $\alpha<1$, the parameter condition \eqref{3d gamma} implies that
	\begin{align}\label{3-6.5}
		n_{3,\alpha,\beta}>\frac{2\alpha}{1-\alpha}>\frac{2\alpha}{1-\alpha}|_{\alpha=\frac{9\sqrt{3}-4\sqrt{2}}{9\sqrt{3}-2\sqrt{2}}}>7.
	\end{align}
	\begin{prop}\label{Prop 3.4}
		Assume that, for some $q\geq1$,
		\begin{align}\label{3-5.5}
			\sup_{0\leq t<\infty}\int \rho^q dx\leq C_q.
		\end{align}
		Then, for any $s\in(0,4]$ satisfying
		\begin{align}
			(6\gamma-12\alpha+3-q)s+12\gamma-24\alpha+12-8q&\leq 0, \label{3-6}
		\end{align}
		there exists a constant $C>0$, depending only on
		$\gamma,\alpha,\nu,\varepsilon,E_0,\overline{\rho_0},
		\underline{\rho_0},q,C_q,s$, and
		$\norm{\rho_0^{1/(s+2)}v_0}_{L^{s+2}}$, such that
		\begin{align*}
			\sup_{0\leq t<\infty}\int \rho |v|^{s+2}dx\leq C.
		\end{align*}
	\end{prop}
	\begin{proof}
		For any $s$ satisfying the assumptions of the proposition,
		we multiply   \eqref{Equ_0'}$_2$ by $|v|^sv$, integrate the resulting
		identity over $\mathbb{T}^3$, and then apply integration by parts to obtain
		\begin{align}\label{3-8}
			\begin{split}
				&\quad\frac{1}{s+2}\frac{d}{dt}\int \rho|v|^{s+2}dx+\nu\int \rho^\alpha|v|^{s}|\nabla v|^2dx+\nu s\int\rho^\alpha|v|^{s}|\nabla|v||^2dx\\
				&\quad+(\nu-c)\int \rho^\alpha(\nabla v)^\top:\nabla(|v|^{s}v)dx+(\alpha-1)(2\nu-c)\int\rho^\alpha\divg v\divg(|v|^{s}v)dx\\
				&\leq C\int \rho^\gamma|v|^{s}|\nabla v|dx.
			\end{split}
		\end{align}
		Since $c\in[\nu,2\nu)$, we have
		\begin{align}\label{3-9}
			\begin{split}
				&\quad(\nu-c)\int \rho^\alpha(\nabla v)^\top:\nabla(|v|^{s} v)dx\\
				&\ge (\nu-c)\int \rho^\alpha|v|^{s}|\nabla v|^2dx+(\nu-c)s\int \rho^\alpha|v|^{s}|\nabla v||\nabla|v||dx.
			\end{split}
		\end{align}
		Moreover, since $c\in[\nu,2\nu)$ and $\alpha\in(2/3,1]$, it holds that
		\begin{align}\label{3-10}
			\begin{split}
				&\quad(\alpha-1)(2\nu-c)\int \rho^\alpha\divg v\divg(|v|^{s} v)dx\\
				&\ge (\alpha-1)(2\nu-c)\int \rho^\alpha|v|^{s}(\divg v)^2dx+(\alpha-1)(2\nu-c)s\int\rho^\alpha |v|^{s}|\divg v||\nabla|v||dx\\
				&\ge 3(\alpha-1)(2\nu-c)\int \rho^\alpha|v|^{s}|\nabla v|^2dx+\sqrt{3}(\alpha-1)(2\nu-c)s\int\rho^\alpha |v|^{s}|\nabla v||\nabla|v||dx,
			\end{split}
		\end{align}
		where, in the last inequality, we have used
		$|\divg v|\leq\sqrt{3}|\nabla v|$. Adding \eqref{3-9} and \eqref{3-10}, and applying Young's inequality, we obtain
		\begin{align}\label{3-11}
			\begin{split}
				&\quad(\nu-c)\int \rho^\alpha(\nabla v)^\top:\nabla(|v|^{s} v)dx+(\alpha-1)(2\nu-c)\int \rho^\alpha\divg v\divg(|v|^{s} v)dx\\
				&\ge \big(\nu-c+3(\alpha-1)(2\nu-c)\big)\int \rho^\alpha|v|^{s}|\nabla v|^2dx\\
				&\quad+\big(\nu-c+\sqrt{3}(\alpha-1)(2\nu-c)\big)s\int\rho^\alpha |v|^{s}|\nabla v||\nabla|v||dx\\
				&\ge \big(\nu-c+3(\alpha-1)(2\nu-c)\big)\int \rho^\alpha|v|^{s}|\nabla v|^2dx\\
				&\quad-\nu s\int\rho^\alpha|v|^{s}|\nabla|v||^2dx-\frac{(c-\nu+\sqrt{3}(1-\alpha)(2\nu-c))^2}{4\nu}s\int \rho^\alpha|v|^{s}|\nabla v|^2dx.
			\end{split}
		\end{align}
		Substituting \eqref{3-11} into \eqref{3-8}, we have
		\begin{align}\label{8-0}
			\begin{split}
				\frac{1}{s+2}\frac{d}{dt}\int \rho|v|^{s+2}dx+\delta_s\int \rho^\alpha|v|^{s}|\nabla v|^2dx\leq C\int \rho^\gamma|v|^{s}|\nabla v|dx,
			\end{split}
		\end{align}
		where
        \[
        \delta_s:=(2\nu-c)(3\alpha-2)
        -\frac{(c-\nu+\sqrt{3}(1-\alpha)(2\nu-c))^2}{4\nu}s>0.
        \]
        The positivity of $\delta_s$ follows from the fact that
        $s<n_{3,\alpha,\beta}$. Using Young's inequality and Hölder's inequality, we get
		\begin{align*}
			\begin{split}
				&\quad\frac{1}{s+2}\frac{d}{dt}\int \rho|v|^{s+2}dx+\delta_s\int \rho^\alpha|v|^{s}|\nabla v|^2dx\\
				&\leq \frac{\delta_s}{2}\int \rho^\alpha|v|^s|\nabla v|^2dx+C\int\rho^{2\gamma-\alpha}|v|^sdx\\
				&\leq \frac{\delta_s}{2}\int \rho^\alpha|v|^s|\nabla v|^2dx+C\Big(\int \rho|v|^{s+2}dx\Big)^{\frac{s}{s+2}}\Big(\int \rho^{\frac{(2\gamma-\alpha)(s+2)-s}{2}}dx\Big)^{\frac{2}{s+2}},
			\end{split}
		\end{align*}
		which implies that
		\begin{align}\label{10--1}
			\begin{split}
				\frac{d}{dt}\|\rho^{\frac{1}{s+2}}v\|_{L^{s+2}}^2&\leq C\left(\int \rho^{\frac{(2\gamma-\alpha)(s+2)-s}{2}}dx\right)^{\frac{2}{s+2}}.
			\end{split}
		\end{align}
		
		We next estimate the right-hand side of the above inequality.
		To this end, we claim that there exists a constant
		$C_2>0$, depending only on
		$\gamma,\alpha,\nu,\varepsilon,E_0,\underline{\rho_0},\overline{\rho_0}$,
		but independent of $n$, such that, for every
		$n\in\mathbb{N}\cup\{-1\}$,
		\begin{align}\label{10-0}
			\norm{\rho^{\frac{3}{2}\alpha-1}}_{L^2(t_n,t_{n+1}; H^2)}\leq C_2.
		\end{align}
		Indeed, since $\gamma>\alpha-\frac{1}{2}$, it follows from
		\eqref{2-1} that
		$\sup_{0\leq t<\infty}\norm{\nabla\rho^{\alpha-1/2}}_{L^2}$ and
		$\sup_{0\leq t<\infty}\norm{\rho^{\alpha-1/2}}_{L^1}$ are both
		bounded. Hence, by the Sobolev embedding theorem,
		$\sup_{0\leq t<\infty}\norm{\rho}_{L^{6\alpha-3}}$ is also bounded,
		where the bound depends only on
		$\gamma,\alpha,\nu,\varepsilon$, and $E_0$. Consequently,
		Young's inequality yields
		\begin{align}\label{10-1}
			\int_{t_n}^{t_{n+1}}\int  \rho^{3\alpha-2}dxdt\leq
			\int_{t_n}^{t_{n+1}}\int (\rho^{6\alpha-3}+1)dxdt
			\leq C(t_{n+1}-t_n)\leq C.
		\end{align}
		It remains to estimate
		$\norm{\nabla\rho^{\frac{3}{2}\alpha-1}}_{L^2(t_n,t_{n+1};L^2)}$.
		Applying Young's inequality together with
		\eqref{nabla rho 4} and \eqref{10-1}, we obtain
		\begin{align}\label{10-2}
			\begin{split}
				&\quad\int_{t_n}^{t_{n+1}}\int  |\nabla\rho^{\frac{3}{2}\alpha-1}|^2dxdt\\
				&=\frac{(3\alpha-2)^2}{4}	\int_{t_n}^{t_{n+1}}\int  \rho^{3\alpha-4}|\nabla\rho|^2 dxdt\\
				&\leq \frac{(3\alpha-2)^2}{4}	\int_{t_n}^{t_{n+1}}\int  \rho^{3\alpha-6}|\nabla\rho|^4 dxdt+\frac{(3\alpha-2)^2}{4}	\int_{t_n}^{t_{n+1}}\int  \rho^{3\alpha-2} dxdt\\
				&=\frac{64}{(3\alpha-2)^2}\int_{t_n}^{t_{n+1}}\int |\nabla\rho^{\frac{3\alpha-2}{4}}|^4dxdt+\frac{(3\alpha-2)^2}{4}	\int_{t_n}^{t_{n+1}}\int  \rho^{3\alpha-2} dxdt\\
				&\leq C.
			\end{split}
		\end{align}
		Combining \eqref{10-1}, \eqref{10-2}, and \eqref{nabla rho 4},
		we conclude that \eqref{10-0} holds. By \eqref{3-5.5}, \eqref{10-0}, and the interpolation inequality,
		we obtain
		\begin{align*}
			\begin{split}
				\|\rho^{\frac{3}{2}\alpha-1}\|_{L^{\frac{2}{s+2}\zeta}(t_n,t_{n+1};L^\zeta)}
				&\le C\|\rho^{\frac{3}{2}\alpha-1}\|_{L^\infty(t_n,t_{n+1}; L^{\frac{2q}{3\alpha-2}})}^{\theta}
				\|\rho^{\frac{3}{2}\alpha-1}\|_{L^{2}(t_n,t_{n+1};W^{1,6})}^{1-\theta}\\
				&\le C,
			\end{split}
		\end{align*}
		where
		\begin{align*}
			\theta=\frac{qs+8q}{(9\alpha-6+q)s+18\alpha-12+8q},\quad
			\zeta=\frac{(9\alpha-6+q)s+18\alpha-12+8q}{9\alpha-6}.
		\end{align*}
		Therefore, there exists a constant $C_3>0$,
		depending only on
		$\gamma,\alpha,\nu,\varepsilon,E_0,\underline{\rho_0},
		\overline{\rho_0},q,C_q$, and $s$, but independent of $n$,
		such that
		\begin{align}\label{10-3}
			\int_{t_n}^{t_{n+1}}\Big(\int \rho^{\zeta'}dx\Big)^{\frac{2}{s+2}}dt\leq C_3,
		\end{align}
		where
		\begin{align*}
			\zeta'=\frac{(9\alpha-6+q)s+18\alpha-12+8q}{6}.
		\end{align*}
		Combining \eqref{3-6} with Young's inequality, we deduce from
        \eqref{10--1} that
		\begin{align*}
			\begin{split}
				\frac{d}{dt}\|\rho^{\frac{1}{s+2}}v\|_{L^{s+2}}^2&\leq C\left(\int \rho^{\frac{(2\gamma-\alpha)(s+2)-s}{2}}dx\right)^{\frac{2}{s+2}}\\
				&=C\left(\int \rho^{\zeta'}\rho^{\frac{(6\gamma-12\alpha+3-q)s+12\gamma-24\alpha+12-8q}{6}}dx\right)^{\frac{2}{s+2}}\\
				&\leq C\left(\int (\rho^{\zeta'}+1)dx\right)^{\frac{2}{s+2}}\\
				&\leq C\left(\int \rho^{\zeta'} dx\right)^{\frac{2}{s+2}}+C.
			\end{split}
		\end{align*}
		Integrating the above inequality over $[t_n,t]$ and taking the supremum over
		$t\in[t_n,t_{n+1}]$, we obtain
		\begin{align*}
			\sup_{t\in[t_{n},t_{n+1}]}\|\rho^{\frac{1}{s+2}}v\|_{L^{s+2}}^2\leq \|\rho^{\frac{1}{s+2}}v(t_n)\|_{L^{s+2}}^2+C	\int_{t_n}^{t_{n+1}}\Big(\int \rho^{\zeta'}dx\Big)^{\frac{2}{s+2}}dt+C(t_{n+1}-t_n).
		\end{align*}
		The second and third terms on the right-hand side are uniformly
        bounded with respect to $n$ by \eqref{10-3} and the uniform
        bound on $t_{n+1}-t_n$, respectively. It remains to estimate the first term. If $n\neq -1$, by Lemma~\ref{Lem 3.1} and the Sobolev embedding theorem (which is applicable since $s\le4$),
		\begin{align*}
			\begin{split}
				\|\rho^{\frac{1}{s+2}}v(t_n)\|_{L^{s+2}}
				&\leq \Big(\frac{3}{2}\overline{\rho_0}\Big)^{\frac{1}{s+2}}
				\norm{v(t_n)}_{H^1}\\
				&\leq \Big(\frac{3}{2}\overline{\rho_0}\Big)^{\frac{1}{s+2}}
				\Big(\Big(\frac{2}{\underline{\rho_0}}\Big)^{\frac12}
				\norm{\rho^{\frac12}v(t_n)}_{L^2}
				+\norm{\nabla v(t_n)}_{L^2}\Big).
			\end{split}
		\end{align*}
		This, together with \eqref{2-2} and
		\eqref{prop 3.1 1}, shows that
		$\|\rho^{\frac{1}{s+2}}v(t_n)\|_{L^{s+2}}$ is uniformly bounded
		with respect to $n$. If $n=-1$, the boundedness of
		$\|\rho^{\frac{1}{s+2}}v(0)\|_{L^{s+2}}$ follows directly from
		the assumption on the initial data. This completes the proof
		of Proposition~\ref{Prop 3.4}.
	\end{proof}
	
	The following technical lemma provides the key step in
	bootstrapping the regularity of both the density and the
	effective velocity starting from the energy estimates.
	\begin{lema}\label{Lem 3.6}
		Assume that $\gamma\in (\frac{15\alpha-7}{3},\frac{111\alpha^2-122\alpha+33}{21\alpha-13})$ and let $q_0=6\alpha-3$. Then there exists an integer
		$n\in\mathbb{N}^+$, depending only on $\alpha$ and $\gamma$,
		such that the sequence $\{q_k\}_{k=0}^{n}$ defined recursively by
		\begin{align*}
			\left\{
			\begin{array}{lc}
				s_{k+1}=\frac{-12\gamma+24\alpha-12+8q_k}{6\gamma-12\alpha+3-q_k},\\
				q_{k+1}=2-2\alpha+\frac{3\alpha-2}{1-s_{k+1}},
			\end{array}
			\right.
		\end{align*}
		for $k=0,\dots,n-1$, satisfies
		\begin{align*}
			6\alpha-3=q_0<\cdots<q_{n-1}
			<2\gamma-4\alpha+\frac{5}{3},
			\quad
			q_n\ge 2\gamma-4\alpha+\frac{5}{3}.
		\end{align*}
	\end{lema}
	\begin{rmk}\label{RMK 3.7}
		Under the assumptions of Lemma~\ref{Lem 3.6}, the sequence
		$\{s_k\}_{k=1}^{n}$ satisfies
		\begin{align*}
			\frac{5\alpha-3}{8\alpha-5}<s_1<\cdots<s_n<1.
		\end{align*}
	\end{rmk}
	\begin{proof}
		The proof is by contradiction. Suppose that the conclusion
		of the lemma fails. 
		
		Define the auxiliary function
		\[
		g(q):=\frac{-12\gamma+24\alpha-12+8q}{6\gamma-12\alpha+3-q},
		\qquad
		q\in\Big[6\alpha-3,\,
		2\gamma-4\alpha+\frac{5}{3}\Big).
		\]
		We claim that $g(q)\in(0,1)$ throughout its domain of definition
		and that $g$ is strictly increasing with respect to $q$.
		Since the monotonicity of $g$ is obvious, it suffices to verify
		that $g(q)\in(0,1)$. Indeed,
		\begin{align*}
			&\quad -12\gamma+24\alpha-12+8q\\
			&\ge -12\gamma+24\alpha-12+8(6\alpha-3)\\
			&=-12\gamma+72\alpha-36\\
			&>-12\times\frac{111\alpha^2-122\alpha+33}{21\alpha-13}+72\alpha-36\\
			&=\frac{12(3\alpha-2)(5\alpha-3)}{21\alpha-13}\\
			&>0,
		\end{align*}
		where the last inequality follows from the fact that
		$\alpha>\frac23$. Moreover,
		\begin{align*}
			\begin{split}
				&\quad 6\gamma-12\alpha+3-q\\
				&>6\gamma-12\alpha+3-(2\gamma-4\alpha+\frac{5}{3})\\
				&=4\gamma-8\alpha+\frac{4}{3}\\
				&>4\times\frac{15\alpha-7}{3}-8\alpha+\frac{4}{3}\\
				&=12\alpha-8\\
				&>0,
			\end{split}
		\end{align*}
		where the last inequality again follows from
		$\alpha>\frac23$. Finally,
		\begin{align*}
			-12\gamma+24\alpha-12+8q
			<6\gamma-12\alpha+3-q
			\iff
			q<2\gamma-4\alpha+\frac53,
		\end{align*}
		which is satisfied by every
		$q\in\left[6\alpha-3,\,2\gamma-4\alpha+\frac53\right)$.
		Therefore, $g(q)\in (0,1)$  throughout the domain of definition of $g$.
		
		Next, define another auxiliary function
		\begin{align*}
		f(q)&:=2-2\alpha+\frac{3\alpha-2}{1-g(q)}
		\\
		&=\frac{(6\gamma-12\alpha+3)\alpha-(2-2\alpha)(-12\gamma+24\alpha-12)+(15\alpha-16)q}
		{18\gamma-36\alpha+15-9q},
			\end{align*}
		for
		\[
		q\in\Big[6\alpha-3,\,
		2\gamma-4\alpha+\frac{5}{3}\Big).
		\]
		We claim that
		\begin{align}\label{11-1}
			f(q)>q
		\end{align}
		throughout the domain of definition of $f$.  Indeed, \eqref{11-1} is equivalent to
		\begin{align*}
			9q^2+(51\alpha-18\gamma-31)q
			+(6\gamma-12\alpha+3)\alpha
			-(2-2\alpha)(-12\gamma+24\alpha-12)>0.
		\end{align*}
		Viewing the left-hand side as a quadratic polynomial in $q$,
		its axis of symmetry is located at
		\[
		q=\frac{18\gamma+31-51\alpha}{18}.
		\]
		Moreover,
		\begin{align*}
			6\alpha-3>\frac{18\gamma+31-51\alpha}{18}
			\iff
			\gamma<\frac{159\alpha-85}{18},
		\end{align*}
		which is automatically satisfied under the assumptions of the
		lemma. Therefore, the quadratic polynomial is increasing on
		the interval $\Big[6\alpha-3,\,
		2\gamma-4\alpha+\frac{5}{3}\Big)$  and hence attains its minimum at $q=6\alpha-3$. It remains to
		verify that
		\begin{align*}
			&9q^2+(51\alpha-18\gamma-31)q
			+(6\gamma-12\alpha+3)\alpha
			-(2-2\alpha)(-12\gamma+24\alpha-12)
			\Big|_{q=6\alpha-3}>0,
		\end{align*}
		which follows  from the assumption
		\[
		\gamma<
		\frac{111\alpha^2-122\alpha+33}{21\alpha-13}.
		\]
		This proves the claim.
		
		Now  define $s_1=g(q_0)$. Since $g(q_0)\in(0,1)$, the quantity
		\[
		q_1=2-2\alpha+\frac{3\alpha-2}{1-g(q_0)}=f(q_0)
		\]
		is well defined. Moreover, \eqref{11-1} implies that
		$q_1>q_0$. Under our contradictory assumption,  we
		necessarily have
		$q_1<2\gamma-4\alpha+\frac53$.
		Repeating the same argument, we define
		$s_2=g(q_1)$ and
		\[
		q_2
		=2-2\alpha+\frac{3\alpha-2}{1-g(q_1)}
		=f(q_1).
		\]
		Again, $q_2>q_1,$ while the contradictory assumption ensures that
		$q_2<2\gamma-4\alpha+\frac53$.	Continuing inductively, we obtain a sequence
		$\{q_k\}_{k=0}^{\infty}$ satisfying
		\begin{align}
			q_{k+1}=f(q_k),\qquad k\in\mathbb N,\label{11-2'}
		\end{align}
		and
		\[
		6\alpha-3=q_0<q_1<\cdots<q_k<\cdots<
		2\gamma-4\alpha+\frac53.
		\]
		Since $\{q_k\}_{k=0}^\infty$ is monotone increasing and bounded above,
		there exists
		$q^*\in\left(6\alpha-3,\,2\gamma-4\alpha+\frac53\right]$
		such that
		\[
		q_k\to q^*,\text{ as }k\to\infty.
		\]
		If
		$q^*\in\left(6\alpha-3,\,2\gamma-4\alpha+\frac53\right)$,
		then passing to the limit in \eqref{11-2'} gives
		$q^*=f(q^*)$, contradicting \eqref{11-1}. It remains to consider the case
		$q^*=2\gamma-4\alpha+\frac53$. In this case,
		\begin{align*}
			&\lim_{q\to q^*}
			\Big((6\gamma-12\alpha+3)\alpha
			-(2-2\alpha)(-12\gamma+24\alpha-12)
			+(15\alpha-16)q\Big)\\
			&=(12\alpha-8)\gamma-24\alpha^2+20\alpha-\frac83\\
			&>\frac{(12\alpha-8)(15\alpha-7)}3
			-24\alpha^2+20\alpha-\frac83\\
			&=36\alpha^2-48\alpha+16>0,
		\end{align*}
		where the strict inequality follows from
		$\gamma>\frac{15\alpha-7}{3}$. Moreover,
		\[
		\lim_{q\to q^*}
		\left(18\gamma-36\alpha+15-9q\right)=0.
		\]
		Therefore,
		\[
		\lim_{q\to q^*}f(q)=+\infty.
		\]
		Since $q_k\to q^*$, we conclude that
		\[
		\lim_{k\to\infty}f(q_k)=+\infty.
		\]
		On the other hand, \eqref{11-2'} implies that
		\[
		f(q_k)=q_{k+1}
		<
		2\gamma-4\alpha+\frac53,
		\qquad
		\forall\,k\in\mathbb N,
		\]
		which is impossible. This contradiction completes the proof.
	\end{proof}
	
	\section{Critical Estimates: $L^3$ Estimates for $v$}
	The following proposition establishes the critical $L^3$ estimate for the effective velocity, which serves as the key step toward the supercritical estimates.
	\begin{prop}\label{Prop 4.1}
		There exists a constant $C>0$, depending only on
		$\gamma,\alpha,\nu,\varepsilon,E_0,
		\overline{\rho_0},\underline{\rho_0}$ and
		$\norm{\rho_0^{1/3}v_0}_{L^3}$, such that
		\begin{align}\label{rho v 3}
			\sup_{0\leq t<\infty}\int \rho|v|^3dx\leq C.
		\end{align}
	\end{prop}
	\begin{proof}
		By \eqref{2-1} and the Sobolev embedding theorem, there exists a
		positive constant $C$, depending only on
		$\gamma,\alpha,\nu,\varepsilon,$ and $E_0$, such that
		\begin{align*}
			\sup_{0\leq t<\infty}\int \rho^{6\alpha-3}\,dx\leq C.
		\end{align*}
		
		We distinguish two cases.
		
		\noindent\textbf{Case I: $\gamma\in\left[1,\frac{15\alpha-7}{3}\right]$.}
		For $(q,s)=(6\alpha-3,1)$, the conditions of
        Proposition~\ref{Prop 3.4} are satisfied. Hence, Proposition~\ref{Prop 3.4}
        gives \eqref{rho v 3}.
		
		\noindent\textbf{Case II: $\gamma\in\left(\frac{15\alpha-7}{3},\frac{111\alpha^2-122\alpha+33}{21\alpha-13}\right)$.}
		Let $\{q_k\}_{k=0}^n$ and $\{s_k\}_{k=1}^n$ be the sequences constructed in Lemma~\ref{Lem 3.6}. Taking $(q,s)=(q_0,s_1)$ in Proposition~\ref{Prop 3.4}, we obtain a positive constant
		$C$, depending only on
		$\gamma,\alpha,\nu,\varepsilon,E_0,
		\overline{\rho_0},\underline{\rho_0}$ and
		$\norm{\rho_0^{1/3}v_0}_{L^3}$, such that
		\begin{align*}
			\sup_{0\leq t<\infty}\int \rho |v|^{s_1+2}\,dx\leq C.
		\end{align*}
		Proposition~\ref{Prop 3.2}(i), together with
		Remark~\ref{RMK 3.7}, then yields
		\begin{align*}
			\sup_{0\leq t<\infty}\int \rho^{q_1}\,dx\leq C.
		\end{align*}
		Iterating the above argument, we obtain
		\begin{align*}
			\sup_{0\leq t<\infty}\int \rho^{q_n}\,dx\leq C.
		\end{align*}
		Since $q_n\ge 2\gamma-4\alpha+\frac53$, Hölder's inequality further implies
		\begin{align*}
			\sup_{0\leq t<\infty}
			\int \rho^{2\gamma-4\alpha+\frac53}\,dx
			\leq C.
		\end{align*}
		It remains to apply Proposition~\ref{Prop 3.4} with
        $(q,s)=\big(2\gamma-4\alpha+\frac53,1\big)$.
        The corresponding assumptions are satisfied, and hence
        \eqref{rho v 3} follows from Proposition~\ref{Prop 3.4}.

		This completes the proof of Proposition~\ref{Prop 4.1}. 
	\end{proof}
	
As a consequence of Proposition~\ref{Prop 4.1} and
Proposition~\ref{Prop 3.2}(ii), we obtain the following corollary.
	\begin{corl}\label{Cor 4.2}
		For every $q\in[1,\infty)$, there exists a constant $C>0$,
		depending only on
		$\gamma,\alpha,\nu,\varepsilon,E_0,
		\underline{\rho_0},\overline{\rho_0},
		\norm{\rho_0^{1/3}v_0}_{L^3}$ and $q$, such that
		\begin{align*}
			\sup_{0\leq t<\infty}\int \rho^q\,dx\leq C.
		\end{align*}
	\end{corl}
	
	\section{Supercritical Estimates I: $L^p$ Estimates for $v$ ($3<p\le6$) and Uniform Density Upper Bounds}
	With the aid of Corollary~\ref{Cor 4.2}, we establish in this section the $L^p$ estimates for the effective velocity with $3<p\le6$, which are crucial for closing the upper bound of the density.
	
	We begin with the following lemma.
	\begin{lema}\label{Lem 5.1}
		For every $s\in (0,n_{3,\alpha,\beta})$, it holds that
		\begin{align}\label{13-1}
			\frac{1}{s+2}\frac{d}{dt}\int \rho|v|^{s+2}dx+\delta_s\int \rho^\alpha|v|^{s}|\nabla v|^2dx\leq C\int \rho^\gamma|v|^{s}|\nabla v|dx,
		\end{align}
		where
		$\delta_s=(2\nu-c)(3\alpha-2)-\frac{(c-\nu+\sqrt{3}(1-\alpha)(2\nu-c))^2}{4\nu}s>0$ and $C>0$ depends only on $s$. Furthermore,
		\begin{align}\label{13-2}
			\frac{d}{dt}\int \rho|v|^{s+2}dx\leq C_{4,s}\int \rho |v|^{s+2}dx+C_{4,s},
		\end{align}
		where $C_{4,s}>0$ depends only on
		$\gamma,\alpha,\nu,\varepsilon,E_0,\underline{\rho_0},
		\overline{\rho_0},
		\norm{\rho_0^{1/3}v_0}_{L^3}$ and $s$.
	\end{lema}
	\begin{proof}
		For any $s\in(0,n_{3,\alpha,\beta})$, multiplying   \eqref{Equ_0'}$_2$ by $|v|^s v$, integrating the resulting equation over $\mathbb{T}^3$, and arguing exactly as in \eqref{3-8}--\eqref{3-11}, we obtain
		\begin{align}\label{11-3}
			\begin{split}
				\frac{1}{s+2}\frac{d}{dt}\int \rho|v|^{s+2}dx+\delta_s\int \rho^\alpha|v|^{s}|\nabla v|^2dx
				\leq C\int \rho^\gamma|v|^{s}|\nabla v|dx,
			\end{split}
		\end{align}
		where $\delta_s>0$ since $s<n_{3,\alpha,\beta}$. This proves \eqref{13-1}.
		
		We next estimate the term on the right-hand side of \eqref{11-3}. By Young's inequality,
		\begin{align*}
			\begin{split}
				C\int \rho^\gamma|v|^{s}|\nabla v|dx
				&\leq \frac{\delta_s}{2}\int \rho^\alpha|v|^s|\nabla v|^2dx
				+C\int \rho^{2\gamma-\alpha}|v|^sdx\\
				&\leq \frac{\delta_s}{2}\int \rho^\alpha|v|^s|\nabla v|^2dx
				+C\int \rho |v|^{s+2}dx
				+C\int \rho^{\frac{(2\gamma-\alpha)(s+2)-s}{2}}dx.
			\end{split}
		\end{align*}
		Substituting the above estimate into \eqref{13-1} and invoking Corollary~\ref{Cor 4.2}, we obtain \eqref{13-2}.
		
		This completes the proof of Lemma~\ref{Lem 5.1}.
	\end{proof}

	The following proposition establishes the $L^3$--$L^6$ estimates for the effective velocity.
	\begin{prop}\label{Prop 5.1}
		For every $s\in(1,4]$, there exists a  constant $C>0$,
		depending only on
		$\gamma,\alpha,\nu,\varepsilon,E_0,
		\underline{\rho_0},\overline{\rho_0},
		s$, and
		$\norm{\rho_0^{1/(s+2)}v_0}_{L^{s+2}}$,
		such that
		\begin{align}\label{11-5}
			\sup_{0\leq t<\infty}\int \rho|v|^{s+2}dx\leq C.
		\end{align}
	\end{prop}
	\begin{proof}
		Fix any $s\in(1,4]$. Since \eqref{3-6.5} implies that $4<n_{3,\alpha,\beta}$, Lemma~\ref{Lem 5.1} is applicable. Applying Gr\"onwall's inequality to \eqref{13-2} on each interval $[t_n,t_{n+1}]$, where $n\in\mathbb{N}\cup\{-1\}$, yields
		\begin{align}\label{11-6}
			\begin{split}
				\sup_{t\in[t_{n},t_{n+1}]}\int \rho |v|^{s+2}dx
				&\leq \int \rho |v|^{s+2}(t_n)dx\,e^{C_{4,s}(t_{n+1}-t_n)}
				+e^{C_{4,s}(t_{n+1}-t_n)}-1\\
				&\leq \int \rho |v|^{s+2}(t_n)dx\,e^{2C_{4,s}T_0}
				+e^{2C_{4,s}T_0}-1.
			\end{split}
		\end{align}
		Next, we estimate the first term on the right-hand side. By Lemma~\ref{Lem 3.1}, the Sobolev embedding theorem, and \eqref{2-2}, for every $n\in\mathbb{N}$,
		\begin{align*}
			\int \rho |v|^{s+2}(t_n)dx
			&\leq \frac{3\overline{\rho_0}}{2}C(\norm{v(t_n)}_{L^2}+\norm{\nabla v(t_n)}_{L^2})^{s+2}\\
			&\leq \frac{3\overline{\rho_0}}{2}C\Big(\big(\frac{2}{\underline{\rho_0}}\big)^{\frac12}
			\norm{\rho^{\frac12}v(t_n)}_{L^2}
			+\norm{\nabla v(t_n)}_{L^2}\Big)^{s+2}\\
			&\leq C_5,
		\end{align*}
		where the positive constant $C_5$ depends only on
		$\gamma,\alpha,\nu,\varepsilon,E_0,
		\overline{\rho_0},\underline{\rho_0}$, and $s$, but is independent of $n$. For $n=-1$, the boundedness of
		$\int \rho |v|^{s+2}(t_n)\,dx$
		follows directly from the assumption on the initial data.
		Therefore, the right-hand side of \eqref{11-6} is bounded uniformly with respect to $n$, which proves \eqref{11-5}.
	\end{proof}

	Next, we exploit the uniform integrability estimate for the effective velocity established in Proposition~\ref{Prop 5.1} to derive a uniform upper bound for the density.  To this end, for any $(\alpha,\gamma,\beta)$ satisfying \eqref{3d gamma} or \eqref{3d gamma'}, we choose
	$\hat{q}_3\in(1,4]$, depending only on $\alpha$ and $\beta$, such that
	\begin{align}
		&\alpha>\frac{\hat{q}_3+1}{\hat{q}_3+2}, \label{3d alpha 1}\\
		&1<\hat{q}_3<n_{3,\alpha,\beta}. \label{3d alpha 2}
	\end{align}
	Such a choice is possible since $\alpha>\frac23$ and \eqref{3-6.5} holds. By Proposition~\ref{Prop 5.1}, we have
	\begin{align}\label{4-3}
		\sup_{0\leq t<\infty}\int \rho|v|^{\hat{q}_3+2}dx\leq C,
	\end{align}
	where the positive constant $C$ depends only on
	$\gamma,\alpha,\nu,\varepsilon,E_0,
	\underline{\rho_0},\overline{\rho_0},
	\hat{q}_3$, and
	$\norm{\rho_0^{1/(\hat{q}_3+2)}v_0}_{L^{\hat{q}_3+2}}$.
	
	The following proposition establishes a uniform upper bound for the density via the modified Nash--Moser iteration developed in \cite{Gu-Huang-Meng-Zhou-2026}, applied on each interval $[t_n,t_{n+1}]$. 
	\begin{prop}\label{Prop 5.3}
		There exists a positive constant $C>0$, depending only on
		$\gamma,\alpha,\nu,\varepsilon,E_0,
		\underline{\rho_0}$, $\overline{\rho_0},$ 
		$\hat{q}_3$, and
		$\norm{\rho_0^{1/(\hat{q}_3+2)}v_0}_{L^{\hat{q}_3+2}}$,
		such that
		\begin{align}\label{3d RT}
			\sup_{0\le t<\infty}\|\rho\|_{L^\infty}\le C.
		\end{align}
	\end{prop}
	\begin{proof}
		Fix an arbitrary $n\in\{-1\}\cup\mathbb{N}$. We aim to establish a uniform upper bound for the density on the interval $[t_n,t_{n+1}]$, independent of $n$.
		
		\noindent\textbf{Step 1: Preparation for the iteration.}
		
		We claim that there exists a universal constant $C_6\ge1$, depending only on the quantities stated in Proposition~\ref{Prop 5.3} but independent of $p$ and $n$, such that, for every $p>0$,
		\begin{align}\label{2-8.4}
			\|\rho^{\frac{p+2}{2}}\|_{L^\infty(t_n,t_{n+1};L^2)}^2
			+\|\rho^{\frac{p+\alpha+1}{2}}\|_{L^2(t_n,t_{n+1};H^1)}^2
			\le
			C_6(p+2)^2
			\|\rho\|_{L^{l}(t_n,t_{n+1};L^{l\frac{\hat{q}_3+2}{\hat{q}_3}})}^l
			+C_6\|\rho^{\frac{p+2}{2}}(t_n)\|_{L^2}^2,
		\end{align}
		where
		$
		l=p+3-\alpha-\frac{2}{\hat{q}_3+2}
		=\frac{p(\hat{q}_3+2)+(3-\alpha)(\hat{q}_3+2)-2}{\hat{q}_3+2}.
		$
		
		To prove this estimate, let $p>0$ be arbitrary. Multiplying $\eqref{Equ_0'}_1$ by $\rho^{p+1}$, integrating over $\mathbb{T}^3$, and applying integration by parts together with Young's inequality, we arrive at
		\begin{align}\label{2-6}
			\begin{split}
				&\quad\frac{1}{p+2}\frac{d}{dt}\int \rho^{p+2}dx
				+c\alpha(p+1)\int \rho^{\alpha+p-1}|\nabla\rho|^2dx\\
				&=\int \rho v\cdot\nabla\rho^{p+1}dx\\
				&\le \frac{c\alpha(p+1)}{2}\int \rho^{\alpha+p-1}|\nabla\rho|^2dx
				+C(p+1)\int \rho^{p-\alpha+3}|v|^2dx\\
				&\le \frac{c\alpha(p+1)}{2}\int \rho^{\alpha+p-1}|\nabla\rho|^2dx
				+C(p+1)
				\left(\int \rho|v|^{\hat{q}_3+2}dx\right)^{\frac{2}{\hat{q}_3+2}}
				\left(\int \rho^{l\frac{\hat{q}_3+2}{\hat{q}_3}}dx\right)^{\frac{\hat{q}_3}{\hat{q}_3+2}}.
			\end{split}
		\end{align}
		Invoking \eqref{4-3}, the above inequality reduces to
		\begin{align}\label{2-8}
			\begin{split}
				&\quad\frac{1}{p+2}\frac{d}{dt}\int \rho^{p+2}dx
				+\frac{c\alpha(p+1)}{2}\int \rho^{\alpha+p-1}|\nabla\rho|^2dx\\
				&\le C(p+1)
				\left(\int \rho^{l\frac{\hat{q}_3+2}{\hat{q}_3}}dx\right)^{\frac{\hat{q}_3}{\hat{q}_3+2}}.
			\end{split}
		\end{align}
		Integrating \eqref{2-8} over the interval $[t_n,t_{n+1}]$ yields
		{\small
			\begin{align}\label{2-8.5}
				\begin{split}
					\|\rho^{\frac{p+2}{2}}\|_{L^\infty(t_n,t_{n+1};L^2)}^2
					+\|\nabla\rho^{\frac{p+\alpha+1}{2}}\|_{L^2(t_n,t_{n+1};L^2)}^2
					\le
					C(p+2)^2
					\|\rho\|_{L^{l}(t_n,t_{n+1};L^{l\frac{\hat{q}_3+2}{\hat{q}_3}})}^l
					+C\|\rho^{\frac{p+2}{2}}(t_n)\|_{L^2}^2.
				\end{split}
			\end{align}
		}
		It remains to estimate
		$\|\rho^{\frac{p+\alpha+1}{2}}\|_{L^2(t_n,t_{n+1};L^2)}^2$.
		By Hölder's inequality,
		\begin{align}\label{2-9}
			\begin{split}
				\|\rho^{\frac{p+\alpha+1}{2}}\|_{L^2(t_n,t_{n+1};L^2)}^2
				&=\int_{t_n}^{t_{n+1}}\int \rho^{p+\alpha+1}dxdt\\
				&\le
				\int_{t_n}^{t_{n+1}}
				\left(\int \rho^{p+2}dx\right)^{\frac{p+\alpha+1}{p+2}}
				\left(\int 1dx\right)^{\frac{1-\alpha}{p+2}}dt\\
				&\le
				(t_{n+1}-t_n)
				\left(
				\sup_{t_n\le t\le t_{n+1}}
				\int \rho^{p+2}dx
				\right)^{\frac{p+\alpha+1}{p+2}}\\
				&\le
				2T_0
				\left(
				\sup_{t_n\le t\le t_{n+1}}
				\int \rho^{p+2}dx
				\right)^{\frac{\alpha-1}{p+2}}
				\sup_{t_n\le t\le t_{n+1}}
				\int \rho^{p+2}dx.
			\end{split}
		\end{align}
		On the other hand, the conservation of mass implied by $\eqref{Equ_0'}_1$ together with Hölder's inequality yields, for every $t\in[t_n,t_{n+1}]$,
		\begin{align*}
			\rho_s
			=\int\rho_0dx
			=\int\rho(x,t)dx
			&\le
			\left(\int\rho^{p+2}(x,t)dx\right)^{\frac1{p+2}}
			\left(\int1dx\right)^{\frac{p+1}{p+2}}\\
			&\le
			\left(\int\rho^{p+2}(x,t)dx\right)^{\frac1{p+2}}\\
			&\le
			\left(
			\sup_{t_n\le t\le t_{n+1}}
			\int\rho^{p+2}dx
			\right)^{\frac1{p+2}},
		\end{align*}
		which, together with the fact that $\alpha\le1$, implies
		\begin{align}\label{2-10}
			\left(
			\sup_{t_n\le t\le t_{n+1}}
			\int\rho^{p+2}dx
			\right)^{\frac{\alpha-1}{p+2}}
			\le
			\rho_s^{\alpha-1}.
		\end{align}
		Substituting \eqref{2-10} into \eqref{2-9}, we obtain
		\begin{align*}
			\|\rho^{\frac{p+\alpha+1}{2}}\|_{L^2(t_n,t_{n+1};L^2)}^2
			\le
			2T_0\rho_s^{\alpha-1}
			\|\rho^{\frac{p+2}{2}}\|_{L^\infty(t_n,t_{n+1};L^2)}^2.
		\end{align*}
		Combining this estimate with \eqref{2-8.5} completes the proof of Step~1.
		
		\noindent\textbf{Step 2: Reverse Hölder inequality.}
		
		We next derive the reverse Hölder inequality, which is the key ingredient in the iteration procedure. To this end, we choose
		$r\in(2,6)$, depending only on $\alpha$ and $\hat{q}_3$, such that
		\begin{align}
			&\frac{\hat{q}_3}{\hat{q}_3+2}-\frac{2}{r}>0,\label{2-5.55}\\
			&\alpha\ge
			\frac{\Big(\frac{\hat{q}_3}{\hat{q}_3+2}-\frac{2}{r}\Big)\hat{q}_3+2\hat{q}_3+2}
			{\Big(\frac{\hat{q}_3}{\hat{q}_3+2}-\frac{2}{r}\Big)(\hat{q}_3+2)+2\hat{q}_3+4}.
			\label{2-5.6}
		\end{align}
		Such a choice is available under the assumptions above. In fact, \eqref{2-5.55} follows  from
		$\hat{q}_3>1$, while \eqref{2-5.6} is guaranteed by
		\eqref{3d alpha 1}. By Hölder's inequality in both space and time, we obtain
		\begin{align}\label{2-11}
			\|\rho\|_{L^o(t_n,t_{n+1};L^{\frac{\hat{q}_3+2}{\hat{q}_3}o})}
			\le
			\|\rho^{\frac{p+2}{2}}\|_{L^\infty(t_n,t_{n+1};L^2)}^\xi
			\|\rho^{\frac{p+\alpha+1}{2}}\|_{L^2(t_n,t_{n+1};L^r)}^\eta,
		\end{align}
		where
		{\small
			\begin{align}\label{2-11.5}
				\begin{split}
					&o=(p+2)\Big(\frac{\hat{q}_3}{\hat{q}_3+2}-\frac{2}{r}\Big)
					+p+\alpha+1,\\
					&\eta=
					\frac{2}
					{(p+2)\Big(\frac{\hat{q}_3}{\hat{q}_3+2}-\frac{2}{r}\Big)
						+p+\alpha+1},
					\qquad
					\xi=
					\frac{2\Big(\frac{\hat{q}_3}{\hat{q}_3+2}-\frac{2}{r}\Big)}
					{(p+2)\Big(\frac{\hat{q}_3}{\hat{q}_3+2}-\frac{2}{r}\Big)
						+p+\alpha+1}.
				\end{split}
			\end{align}
		}
	Since $H^1\hookrightarrow L^r$, combining \eqref{2-11} with \eqref{2-8.4} and the Sobolev embedding inequality yields
	\begin{align}\label{2-12}
		\begin{split}
			\|\rho\|_{L^o(t_n,t_{n+1};L^{\frac{\hat{q}_3+2}{\hat{q}_3}o})}
			&\le
			C_r^\eta
			\|\rho^{\frac{p+2}{2}}\|_{L^\infty(t_n,t_{n+1};L^2)}^\xi
			\|\rho^{\frac{p+\alpha+1}{2}}\|_{L^2(t_n,t_{n+1};H^1)}^\eta\\
			&\le
			C_r^\eta
			\left(
			C_6(p+2)^2
			\|\rho\|_{L^l(t_n,t_{n+1};L^{l\frac{\hat{q}_3+2}{\hat{q}_3}})}^l
			+
			C_6
			\|\rho^{\frac{p+2}{2}}(t_n)\|_{L^2}^2
			\right)^{\frac{\xi+\eta}{2}},
		\end{split}
	\end{align}
	where $C_r\ge 1$ denotes the Sobolev embedding constant. For convenience, define $
		C_7(n)
		=
		\|\rho(t_n)\|_{L^2}
		+
		\|\rho(t_n)\|_{L^\infty}
		+
		1.
		$
		Then
		\begin{align}\label{2-13}
			\|\rho^{\frac{p+2}{2}}(t_n)\|_{L^2}^2
			=
			\|\rho(t_n)\|_{L^{p+2}}^{p+2}
			\le
			(\|\rho(t_n)\|_{L^2}+\|\rho(t_n)\|_{L^\infty})^{p+2}
			\le
			C_7(n)^{p+2}
			\le
			C_7(n)^{C_8l},
		\end{align}
		where $C_8>0$ depends only on $\alpha$ and $\hat{q}_3$, is independent of $p$, and is chosen such that
		\[
		\frac{p+2}{l}\le C_8.
		\]
		Next define
		$
		\Upsilon_n(l)
		=
		\max
		\left\{
		\|\rho\|_{L^l(t_n,t_{n+1};L^{\frac{\hat{q}_3+2}{\hat{q}_3}l})},
		\,
		C_7(n)^{C_8}
		\right\}.
		$
		Then, by \eqref{2-12} and \eqref{2-13},
		\begin{align*}
			\begin{split}
				\Upsilon_n(o)
				&\le
				\max
				\left\{
				C_r^\eta
				\left(
				C_6(p+2)^2
				\|\rho\|_{L^l(t_n,t_{n+1};L^{l\frac{\hat{q}_3+2}{\hat{q}_3}})}^l
				+
				C_6
				\|\rho^{\frac{p+2}{2}}(t_n)\|_{L^2}^2
				\right)^{\frac{\xi+\eta}{2}},
				C_7(n)^{C_8}
				\right\}\\
				&\le
				\max
				\left\{
				C_r^\eta
				\left(
				2C_6(p+2)^2
				\Upsilon_n(l)^l
				\right)^{\frac{\xi+\eta}{2}},
				\Upsilon_n(l)
				\right\}.
			\end{split}
		\end{align*}
		Finally, by the definition of $\xi,\eta,l$, the condition \eqref{2-5.6} implies that
		\[
		\frac{\xi+\eta}{2}l\le1.
		\]
		Consequently,
		\begin{align}\label{2-14}
			\Upsilon_n(o)
			\le
			(2C_r^2C_6)^{\frac{\xi+\eta}{2}}
			(p+2)^{\xi+\eta}
			\Upsilon_n(l).
		\end{align}
		Thus, we obtain the desired reverse Hölder inequality, which will be used to initiate the subsequent iteration.

		\noindent \textbf{Step 3: Initialization and iteration scheme.}
		
		We next initialize the iteration. By virtue of \eqref{2-5.55}, we have
		\begin{align*}
			\frac{o}{l}
			&=\left((p+2)\Big(\frac{\hat{q}_3}{\hat{q}_3+2}-\frac{2}{r}\Big)+p+\alpha+1\right)
			\frac{\hat{q}_3+2}{p(\hat{q}_3+2)+(3-\alpha)(\hat{q}_3+2)-2}\\
			&\to1+\frac{\hat{q}_3}{\hat{q}_3+2}-\frac{2}{r}>1,
			\qquad\text{as }p\to\infty.
		\end{align*}
		Hence, one can choose $p_0\ge2$ sufficiently large such that, for every $p>p_0$,
		\begin{align}
			\frac{o}{l}\ge
			1+\frac12\left(
			\frac{\hat{q}_3}{\hat{q}_3+2}-\frac{2}{r}
			\right)
			=:d>1,
			\label{2-15}
		\end{align}
		and
		\begin{align}
			\frac{o}{l}\le
			1+\frac32\left(
			\frac{\hat{q}_3}{\hat{q}_3+2}-\frac{2}{r}
			\right)
			=:d'.
			\label{2-16}
		\end{align}
		Replacing $p$ by $p_0$ in the definitions of $l$, $o$, $\xi$, and $\eta$, we obtain the corresponding quantities $l_0$, $o_0$, $\xi_0$, and $\eta_0$. We then define the iteration recursively by setting $l_1=o_0$, choosing $p_1$ according to the definition of $l_1$, and letting  $o_1$, $\xi_1$, and $\eta_1$ denote the associated quantities. Repeating this procedure inductively generates sequences
		$\{p_k\}$, $\{l_k\}$, $\{o_k\}$, $\{\xi_k\}$, and $\{\eta_k\}$. Applying \eqref{2-14} at each iteration step yields, for any $k\in\mathbb{N}$, 
		\begin{align}\label{2-16.1}
			\begin{split}
				\Upsilon_n(l_{k+1})
				&\leq
				(2C_r^2C_6)^{\frac{\xi_k+\eta_k}{2}}
				(p_k+2)^{\xi_k+\eta_k}
				\Upsilon_n(l_k)\\
				&\leq
				(2C_r^2C_6)^{
					\frac{\xi_k+\eta_k}{2}
					+\frac{\xi_{k-1}+\eta_{k-1}}{2}}
				(p_k+2)^{\xi_k+\eta_k}
				(p_{k-1}+2)^{\xi_{k-1}+\eta_{k-1}}
				\Upsilon_n(l_{k-1})\\
				&\leq
				(2C_r^2C_6)^{
					\sum_{i=0}^{k}\frac{\xi_i+\eta_i}{2}}
				\prod_{i=0}^{k}
				(p_i+2)^{\xi_i+\eta_i}
				\Upsilon_n(l_0).
			\end{split}
		\end{align}
		
		We now control the factors appearing on the right-hand side. It follows from \eqref{2-15} that
		\begin{align}\label{2-16.109}
			o_i\ge d^{i+1}l_0.
		\end{align}
		Combining this with \eqref{2-11.5}, we infer that
		\begin{align}\label{2-16.11}
			\begin{split}
				\frac{\xi_i+\eta_i}{2}
				=
				\frac{
					\frac{\hat q_3}{\hat q_3+2}
					-\frac2r+1}
				{o_i}
				\le
				\frac{
					\frac{\hat q_3}{\hat q_3+2}
					-\frac2r+1}
				{l_0}
				\frac1{d^{i+1}}.
			\end{split}
		\end{align}
		Consequently,
		\begin{align}\label{2-16.115}
			\sum_{i=0}^{\infty}
			\frac{\xi_i+\eta_i}{2}
			<\infty.
		\end{align}
		On the other hand, \eqref{2-16} implies
		\[
		o_i\le d'^{\,i+1}l_0,
		\]
		which further yields
		\begin{align}\label{2-16.12}
			p_i
			\le
			\frac{
				d'^{\,i+1}l_0
				-2\left(
				\frac{\hat q_3}{\hat q_3+2}
				-\frac2r
				\right)
				-\alpha-1}
			{\frac{\hat q_3}{\hat q_3+2}
				-\frac2r+1}
			\le
			\frac{l_0}
			{\frac{\hat q_3}{\hat q_3+2}
				-\frac2r+1}
			d'^{\,i+1}.
		\end{align}
		By \eqref{2-16.11} and \eqref{2-16.12},
		\begin{align*}
			&\quad
			\sum_{i=0}^{\infty}
			(\xi_i+\eta_i)\log(p_i+2)\\
			&\le
			\sum_{i=0}^{\infty}
			\frac{
				2\left(
				\frac{\hat q_3}{\hat q_3+2}
				-\frac2r+1
				\right)}
			{l_0}
			\frac1{d^{i+1}}
			\log\left(
			\frac{2l_0}
			{\frac{\hat q_3}{\hat q_3+2}
				-\frac2r+1}
			d'^{\,i+1}
			\right)\\
			&\le
			\sum_{i=0}^{\infty}
			\frac{
				2\left(
				\frac{\hat q_3}{\hat q_3+2}
				-\frac2r+1
				\right)}
			{l_0}
			\frac1{d^{i+1}}
			\left(
			\log\left(
			\frac{2l_0}
			{\frac{\hat q_3}{\hat q_3+2}
				-\frac2r+1}
			\right)
			+(i+1)\log d'
			\right)
			<\infty,
		\end{align*}
		and hence
		\begin{align}\label{2-16.13}
			\prod_{i=0}^{\infty}
			(p_i+2)^{\xi_i+\eta_i}
			<\infty.
		\end{align}
        
		\noindent\textbf{Step 4: Uniform upper bound for the density.}
		Recalling \eqref{2-16.1}, we have, for every $k\in\mathbb N$,
		\begin{align}\label{4-7.1}
			\begin{split}
				\Upsilon_n(l_{k+1})
				\le
				(2C_r^2C_6)^{\sum_{i=0}^{\infty}\frac{\xi_i+\eta_i}{2}}
				\prod_{i=0}^{\infty}(p_i+2)^{\xi_i+\eta_i}
				\Upsilon_n(l_0),
			\end{split}
		\end{align}
		where both
		\[
		\sum_{i=0}^{\infty}\frac{\xi_i+\eta_i}{2}
		\quad\text{and}\quad
		\prod_{i=0}^{\infty}(p_i+2)^{\xi_i+\eta_i}
		\]
		are finite and independent of $n$, by \eqref{2-16.115} and \eqref{2-16.13}, respectively.
		
		It remains to verify that $\Upsilon_n(l_0)$ is also bounded independently of $n$. Indeed, Hölder's inequality gives
		\begin{align*}
			\|\rho\|_{L^{l_0}(t_n,t_{n+1};L^{\frac{\hat q_3+2}{\hat q_3}l_0})}
			&\le
			\|\rho\|_{L^\infty(t_n,t_{n+1};L^{\frac{\hat q_3+2}{\hat q_3}l_0})}
			(t_{n+1}-t_n)^{\frac1{l_0}}\\
			&\le
			(2T_0)^{\frac1{l_0}}
			\|\rho\|_{L^\infty(t_n,t_{n+1};L^{\frac{\hat q_3+2}{\hat q_3}l_0})}.
		\end{align*}
		Combining this estimate with Corollary~\ref{Cor 4.2}, we conclude that$
		\|\rho\|_{L^{l_0}(t_n,t_{n+1};L^{\frac{\hat q_3+2}{\hat q_3}l_0})}
		$ is bounded independently of $n$. On the other hand,
		\begin{align*}
			(C_7(n))^{C_8}
			=
			(\|\rho(t_n)\|_{L^2}
			+
			\|\rho(t_n)\|_{L^\infty}
			+
			1)^{C_8},
		\end{align*}
		which, together with \eqref{prop 3.1 0}, shows that $(C_7(n))^{C_8}$ is also bounded independently of $n$. Consequently, $\Upsilon_n(l_0)$ admits a bound independent of $n$.
		
		Therefore, Hölder's inequality and \eqref{4-7.1} imply
		\begin{align}
			\|\rho\|_{L^{l_{k+1}}(\mathbb{T}^3\times[t_n,t_{n+1}])}
			\le
			\Upsilon_n(l_{k+1})
			\le
			(2C_r^2C_6)^{\sum_{i=0}^{\infty}\frac{\xi_i+\eta_i}{2}}
			\prod_{i=0}^{\infty}(p_i+2)^{\xi_i+\eta_i}
			\Upsilon_n(l_0).
		\end{align}
		Finally, letting $k\to\infty$, we infer from \eqref{2-16.109} that $l_{k+1}\to\infty$. Hence,
		\begin{align*}
			\|\rho\|_{L^\infty(\mathbb{T}^3\times[t_n,t_{n+1}])}
			\le
			(2C_r^2C_6)^{\sum_{i=0}^{\infty}\frac{\xi_i+\eta_i}{2}}
			\prod_{i=0}^{\infty}(p_i+2)^{\xi_i+\eta_i}
			\Upsilon_n(l_0).
		\end{align*}
		Since the right-hand side is bounded independently of $n$, the desired uniform upper bound follows immediately.
		
		This completes the proof of Proposition~\ref{Prop 5.3}.
	\end{proof}
	
	\section{Supercritical Estimates II: $L^p$ Estimates for $v$ ($6\le p<n_{3,\alpha,\beta}+2$) and Uniform Density Lower Bounds}
	In this section, we further improve the integrability of the effective velocity to exponents exceeding $6$. This improvement is essential for establishing a uniform lower bound for the density.
	
	We begin with the following lemma.
	\begin{lema}\label{Lem5.5}
		Assume that, for some $k\in\mathbb{N}$ satisfying
		\[
		6\times3^k-2<n_{3,\alpha,\beta},
		\]
		the following estimates hold:
		\begin{align}
			&\sup_{0\leq t<\infty}\int \rho|v|^{6\times3^{k}}dx\leq C,\label{13-4}\\
			&\int_{\tau_1}^{\tau_2}\int \rho^\alpha|v|^{6\times3^{k-1}-2}|\nabla v|^2\,dxdt
			=o(\tau_2-\tau_1),
			\qquad\text{as }\tau_2-\tau_1\to\infty. \label{13-5}
		\end{align}
		Then it holds that
		\begin{align}\label{13-8}
			\int_{\tau_1}^{\tau_2}\int \rho^\alpha|v|^{6\times3^{k}-2}|\nabla v|^2\,dxdt
			=o(\tau_2-\tau_1),
			\qquad\text{as }\tau_2-\tau_1\to\infty.
		\end{align}
	\end{lema}
	\begin{proof}
		\noindent\textbf{Step 1.}
		Set
		\[
		\iota_0=6\times 3^{k-1}-2,\qquad
		\iota_1=\frac{6\times 3^k+\iota_0}{2}.
		\]
		Since $\iota_1<6\times3^k-2<n_{3,\alpha,\beta}$, we may apply \eqref{13-1} with $s=\iota_1$ to obtain
		\begin{align}\label{13-3}
			\frac{1}{\iota_1+2}\frac{d}{dt}\int \rho|v|^{\iota_1+2}dx+\delta_{\iota_1}\int \rho^\alpha|v|^{\iota_1}|\nabla v|^2dx\leq C\int \rho^\gamma|v|^{\iota_1}|\nabla v|dx.
		\end{align}
		Next, applying Hölder's inequality together with \eqref{13-4} and \eqref{3d RT}, we obtain
		\begin{align*}
			\begin{split}
				C\int\rho^\gamma|v|^{\iota_1}|\nabla v|dx
				&\leq C\Big(\int \rho^\alpha|v|^{\iota_0}|\nabla v|^2dx\Big)^{\frac{1}{2}}
				\Big(\int \rho^{2\gamma-\alpha}|v|^{6\times 3^k}dx\Big)^{\frac{1}{2}}\\
				&\leq C\Big(\int \rho^\alpha|v|^{\iota_0}|\nabla v|^2dx\Big)^{\frac{1}{2}}
				\norm{\rho}_{L^\infty}^{\frac{2\gamma-\alpha-1}{2}}\Big(\int \rho|v|^{6\times 3^k}dx\Big)^{\frac{1}{2}}\\
				&\leq C\Big(\int \rho^\alpha|v|^{\iota_0}|\nabla v|^2dx\Big)^{\frac{1}{2}}.
			\end{split}
		\end{align*}
		
		Therefore, integrating \eqref{13-3} over $(\tau_1,\tau_2)$ for any $0\le \tau_1<\tau_2<\infty$, and using \eqref{13-4}, \eqref{13-5}, Hölder's inequality, and Young's inequality, we arrive at
		\begin{align*}
			\begin{split}
				&\quad\delta_{\iota_1}\int_{\tau_1}^{\tau_2}\int \rho^\alpha|v|^{\iota_1}|\nabla v|^2dxdt\\
				&\leq \frac{1}{\iota_1+2}\int \rho|v|^{\iota_1+2}(\tau_1)dx
				-\frac{1}{\iota_1+2}\int \rho|v|^{\iota_1+2}(\tau_2)dx\\
				&\qquad
				+C\int_{\tau_1}^{\tau_2}\Big(\int \rho^\alpha|v|^{\iota_0}|\nabla v|^2dx\Big)^{\frac{1}{2}}dt\\
				&\leq \frac{1}{\iota_1+2}\int \rho(|v|^{6\times 3^k}+1)(\tau_1)dx+C\Big(\int_{\tau_1}^{\tau_2}\int \rho^\alpha|v|^{\iota_0}|\nabla v|^2dxdt\Big)^{\frac{1}{2}} \sqrt{\tau_2-\tau_1}\\
				&\leq C+o(\sqrt{\tau_2-\tau_1})\sqrt{\tau_2-\tau_1},
			\end{split}
		\end{align*}
		where we have used the fact that $\iota_1+2\leq 6\times 3^k$.  Consequently,
		\begin{align}
			\int_{\tau_1}^{\tau_2}\int \rho^\alpha|v|^{\iota_1}|\nabla v|^2dxdt=o(\tau_2-\tau_1),\text{ as }\tau_2-\tau_1\to\infty.
		\end{align}
		
		\noindent\textbf{Step 2.}
		Define the sequence $\{\iota_i\}_{i\ge1}$ recursively by
		\[
		\iota_i=\frac{6\times3^k+\iota_{i-1}}{2},\qquad i\in\mathbb{N}^+.
		\]
		The argument in Step 1 can be iterated. More precisely, whenever
		\begin{align}\label{13-5.55}
			\int_{\tau_1}^{\tau_2}\int \rho^\alpha|v|^{\iota_i}|\nabla v|^2dxdt=o(\tau_2-\tau_1),\text{ as }\tau_2-\tau_1\to\infty,
		\end{align}
		and
		\[
		\iota_{i+1}\leq 6\times3^k-2,\quad (\text{which automatically implies $\iota_{i+1}<n_{3,\alpha,\beta}$})
		\]
		it follows that
		\begin{align}\label{13-3.56}
			\int_{\tau_1}^{\tau_2}\int \rho^\alpha|v|^{\iota_{i+1}}|\nabla v|^2dxdt=o(\tau_2-\tau_1),\text{ as }\tau_2-\tau_1\to\infty.
		\end{align}
		
		Indeed, replacing $s$ by $\iota_{i+1}$ in \eqref{13-1} yields
		\begin{align}\label{13-3.55}
			\frac{1}{\iota_{i+1}+2}\frac{d}{dt}\int \rho|v|^{\iota_{i+1}+2}dx+\delta_{\iota_{i+1}}\int \rho^\alpha|v|^{\iota_{i+1}}|\nabla v|^2dx\leq C\int \rho^\gamma|v|^{\iota_{i+1}}|\nabla v|dx.
		\end{align}
		Applying Hölder's inequality together with \eqref{13-4} and \eqref{3d RT}, we obtain
		\begin{align*}
			\begin{split}
				C\int\rho^\gamma|v|^{\iota_{i+1}}|\nabla v|dx
				&\leq C\Big(\int \rho^\alpha|v|^{\iota_i}|\nabla v|^2dx\Big)^{\frac12}
				\Big(\int \rho^{2\gamma-\alpha}|v|^{6\times3^k}dx\Big)^{\frac12}\\
				&\leq C\Big(\int \rho^\alpha|v|^{\iota_i}|\nabla v|^2dx\Big)^{\frac12}.
			\end{split}
		\end{align*}
		Therefore, integrating \eqref{13-3.55} over $(\tau_1,\tau_2)$, and using
		\eqref{13-4}, \eqref{13-5.55}, Hölder's inequality, and Young's inequality, we arrive at
		\begin{align*}
			\begin{split}
				&\quad\delta_{\iota_{i+1}}\int_{\tau_1}^{\tau_2}\int \rho^\alpha|v|^{\iota_{i+1}}|\nabla v|^2dxdt\\
				&\leq \frac{1}{\iota_{i+1}+2}\int \rho|v|^{\iota_{i+1}+2}(\tau_1)dx
				-\frac{1}{\iota_{i+1}+2}\int \rho|v|^{\iota_{i+1}+2}(\tau_2)dx\\
				&\qquad
				+C\int_{\tau_1}^{\tau_2}\Big(\int \rho^\alpha|v|^{\iota_i}|\nabla v|^2dx\Big)^{\frac12}dt\\
				&\leq C+o(\sqrt{\tau_2-\tau_1})\sqrt{\tau_2-\tau_1},
			\end{split}
		\end{align*}
		where we have used the fact that $\iota_{i+1}+2\leq6\times3^k$.
		This proves \eqref{13-3.56}.
		
		\noindent\textbf{Step 3.}
		Since $\{\iota_i\}$ is a strictly increasing sequence converging to $6\times3^k$, there exists an integer $i_*\in\mathbb{N}^+$ such that
		\begin{align*}
			\iota_{i_*}\leq 6\times3^k-2,\qquad
			\iota_{i_*+1}>6\times3^k-2.
		\end{align*}
		By repeatedly applying the argument developed in Steps~1 and 2, we obtain
		\begin{align}\label{13-7}
			\int_{\tau_1}^{\tau_2}\int \rho^\alpha|v|^{\iota_{i_*}}|\nabla v|^2dxdt=o(\tau_2-\tau_1),\text{ as }\tau_2-\tau_1\to\infty.
		\end{align}
		
		Since $6\times3^k-2<n_{3,\alpha,\beta}$, we may substitute $s=6\times3^k-2$ into \eqref{13-1} to obtain
		\begin{align}\label{13-6}
			\frac{1}{6\times3^k}\frac{d}{dt}\int \rho|v|^{6\times3^k}dx+\delta_{6\times3^k-2}\int \rho^\alpha|v|^{6\times3^k-2}|\nabla v|^2dx\leq C\int \rho^\gamma|v|^{6\times3^k-2}|\nabla v|dx.
		\end{align}
		Applying Hölder's inequality together with \eqref{13-4} and \eqref{3d RT}, we obtain
		\begin{align*}
			\begin{split}
				C\int\rho^\gamma|v|^{6\times3^k-2}|\nabla v|dx
				&\leq C\Big(\int \rho^\alpha|v|^{\iota_{i_*}}|\nabla v|^2dx\Big)^{\frac12}
				\Big(\int \rho^{2\gamma-\alpha}|v|^{12\times3^k-4-\iota_{i_*}}dx\Big)^{\frac12}\\
				&\leq C\Big(\int \rho^\alpha|v|^{\iota_{i_*}}|\nabla v|^2dx\Big)^{\frac12},
			\end{split}
		\end{align*}
		where we have used the fact that
		\begin{align*}
			12\times3^k-4-\iota_{i_*}<6\times3^k
			\Longleftrightarrow
			\iota_{i_*}>6\times3^k-4
			\Longleftrightarrow
			\iota_{i_*+1}>6\times3^k-2.
		\end{align*}
		
		Therefore, integrating \eqref{13-6} over $(\tau_1,\tau_2)$ for any $0\leq \tau_1<\tau_2<\infty$, and using \eqref{13-4}, \eqref{13-7}, Hölder's inequality, and Young's inequality, we obtain
		\begin{align*}
			\begin{split}
				&\quad\delta_{6\times3^k-2}\int_{\tau_1}^{\tau_2}\int \rho^\alpha|v|^{6\times3^k-2}|\nabla v|^2dxdt\\
				&\leq \frac{1}{6\times3^k}\int \rho|v|^{6\times3^k}(\tau_1)dx
				-\frac{1}{6\times3^k}\int \rho|v|^{6\times3^k}(\tau_2)dx\\
				&\qquad
				+C\int_{\tau_1}^{\tau_2}\Big(\int \rho^\alpha|v|^{\iota_{i_*}}|\nabla v|^2dx\Big)^{\frac12}dt\\
				&\leq C+o(\sqrt{\tau_2-\tau_1})\sqrt{\tau_2-\tau_1},
			\end{split}
		\end{align*}
		which implies that \eqref{13-8} holds. 
		
		This completes the proof of Lemma~\ref{Lem5.5}.
	\end{proof}

\begin{lema}\label{Lem 5.6}
Assume that, for some $k\in\mathbb{N}$ satisfying
\begin{align}\label{63k+1}
6\times3^{k+1}-2<n_{3,\alpha,\beta},
\end{align}
the following estimates hold:
\begin{align}
&\sup_{0\leq t<\infty}\int \rho|v|^{6\times3^{k}}dx\le C,
\label{13-9}\\
&\int_{\tau_1}^{\tau_2}\int \rho^\alpha|v|^{6\times3^{k}-2}|\nabla v|^2dxdt
=o(\tau_2-\tau_1),
\qquad\text{as }\tau_2-\tau_1\to\infty.
\label{13-10}
\end{align}
Then there exists a  constant $C>0$ (possibly different) such that
\begin{align}
\sup_{0\leq t<\infty}\int \rho|v|^{6\times3^{k+1}}dx
\le C.
\label{13-11}
\end{align}
\end{lema}
\begin{proof}
		Fix $\delta^{(1)}>0$ to be sufficiently small. By \eqref{3d RT}, \eqref{nabla rho 4}, and \eqref{13-10}, we have
		\begin{align*}
			\int_{\tau_1}^{\tau_2}\int \Big(|\nabla\rho|^4+\rho^\alpha|v|^{6\times 3^{k}-2}|\nabla v|^2\Big)dxdt
			=o(\tau_2-\tau_1),\qquad\text{as }\tau_2-\tau_1\to\infty.
		\end{align*}
		Therefore, there exists a sufficiently large constant $T_0^{(1)}>0$, depending on $\delta^{(1)}$, such that, for every $n\in\mathbb N$,
		\begin{align*}
			\frac{1}{T_0^{(1)}}\int_{nT_0^{(1)}}^{(n+1)T_0^{(1)}}\int
			\Big(|\nabla\rho|^4+\rho^\alpha|v|^{6\times 3^{k}-2}|\nabla v|^2\Big)dxdt
			\le\delta^{(1)}.
		\end{align*}
		By the mean value theorem for integrals, there exists a point
		$t_n^{(1)}\in[nT_0^{(1)},(n+1)T_0^{(1)}]$ such that
		\begin{align*}
			\int\Big(|\nabla\rho|^4+\rho^\alpha|v|^{6\times 3^{k}-2}|\nabla v|^2\Big)(t_n^{(1)})dx
			\le\delta^{(1)}.
		\end{align*}
		Arguing as in the proof of Lemma~\ref{Lem 3.1}, we now choose $\delta^{(1)}>0$ sufficiently small so that
		\begin{align}
			\frac{1}{2}\underline{\rho_0}\leq \rho(x,t_n^{(1)})\leq \frac{3}{2}\overline{\rho_0},
			\quad \forall x\in\mathbb{T}^3,
			\label{13-13}\\
			\int |v|^{6\times3^{k}-2}|\nabla v|^2(t_n^{(1)})dx
			\leq
			\Big(\frac{2}{\underline{\rho_0}}\Big)^\alpha.
			\label{13-14}
		\end{align}
		
		Since \eqref{63k+1} holds, we may substitute
		$s=6\times3^{k+1}-2$ into \eqref{13-2} to obtain
		\begin{align}
			\begin{split}
				\frac{d}{dt}\int \rho|v|^{6\times3^{k+1}}dx
				\le
				C_{4,6\times3^{k+1}-2}
				\int \rho|v|^{6\times3^{k+1}}dx
				+
				C_{4,6\times3^{k+1}-2}.
			\end{split}
		\end{align}
		Set $t_{-1}^{(1)}=0$. Applying Grönwall's inequality on each interval
		$[t_n^{(1)},t_{n+1}^{(1)}]$, $n\in\mathbb N\cup\{-1\}$, yields
		\begin{align}\label{13-15}
			\begin{split}
				&\quad
				\sup_{t\in[t_n^{(1)},t_{n+1}^{(1)}]}
				\int \rho|v|^{6\times3^{k+1}}dx\\
				&\le
				\int \rho|v|^{6\times3^{k+1}}(t_n^{(1)})dx
				e^{C_{4,6\times3^{k+1}-2}(t_{n+1}^{(1)}-t_n^{(1)})}
				+
				e^{C_{4,6\times3^{k+1}-2}(t_{n+1}^{(1)}-t_n^{(1)})}
				-
				1\\
				&\le
				\int \rho|v|^{6\times3^{k+1}}(t_n^{(1)})dx
				e^{2C_{4,6\times3^{k+1}-2}T_0^{(1)}}
				+
				e^{2C_{4,6\times3^{k+1}-2}T_0^{(1)}}
				-
				1.
			\end{split}
		\end{align}
		
		It remains to show that the first term on the right-hand side of
		\eqref{13-15} is bounded independently of $n$. By \eqref{13-13} and the Sobolev embedding theorem,
		\begin{align*}
			\int \rho |v|^{6\times3^{k+1}}(t_n^{(1)})dx
			&\le
			\frac{3}{2}\overline{\rho_0}
			\|v(t_n^{(1)})\|_{L^{6\times3^{k+1}}}^{6\times3^{k+1}}\\
			&\le
			\frac{3}{2}\overline{\rho_0}
			\||v|^{3^{k+1}}(t_n^{(1)})\|_{L^6}^{6}\\
			&\le
			\frac{3}{2}\overline{\rho_0}
			C\||v|^{3^{k+1}}(t_n^{(1)})\|_{H^1}^{6}.
		\end{align*}
		The right-hand side is uniformly bounded with respect to $n$ in view of
		\eqref{13-13}, \eqref{13-14}, and \eqref{13-9}. Consequently, the
		right-hand side of \eqref{13-15} is bounded independently of $n$, which
		establishes \eqref{13-11}. This completes the proof of
		Lemma~\ref{Lem 5.6}.
	\end{proof}
	
	With the aid of Lemmas~\ref{Lem5.5} and \ref{Lem 5.6}, we establish the $L^p$ integrability of the effective velocity for all $6\le p<n_{3,\alpha,\beta}+2.$
	\begin{prop}\label{Prop 5.4}
		For any $s\in(4,n_{3,\alpha,\beta})$, there exists a constant $C>0$, depending only on
		$\gamma,\alpha,\nu,\varepsilon,E_0,\underline{\rho_0},\overline{\rho_0},
		\hat{q}_3,s,
		\norm{\rho_0^{1/(\hat{q}_3+2)}v_0}_{L^{\hat{q}_3+2}}$, and
		$\norm{\rho_0^{1/(s+2)}v_0}_{L^{s+2}}$, such that
		\begin{align}\label{13-16}
			\sup_{0\le t<\infty}\int \rho|v|^{s+2}\,dx\le C.
		\end{align}
	\end{prop}
	
	\begin{proof} 
		We first consider the case $\alpha<1$. Let $k_0\in\mathbb{N}$ be such that
		\begin{align*}
			6\times 3^{k_0}-2<n_{3,\alpha,\beta}\leq 6\times 3^{k_0+1}-2.
		\end{align*}
		
		\noindent\textbf{Step I.}
		We claim that \eqref{13-16} holds with $s=6\times3^{k_0}-2$.
		
		If $k_0=0$, then the claim follows directly from Proposition~\ref{Prop 5.1}.
		
		If $k_0\ge1$, we prove the claim by induction. Assume that, for some
		$k\le k_0-1$ with $k\in\mathbb{N}$,
		\begin{align}
			\int_{\tau_1}^{\tau_2}\int \rho^\alpha|v|^{6\times 3^{k-1}-2}|\nabla v|^2dxdt&= o(\tau_2-\tau_1),\text{ as }\tau_2-\tau_1\to\infty,\\
			\sup_{0\leq t<\infty}\int \rho |v|^{6\times 3^k}dx&\leq C.\label{14-5}
		\end{align}
		Then, by Lemmas~\ref{Lem5.5} and \ref{Lem 5.6}, it follows that
		\begin{align}
			\int_{\tau_1}^{\tau_2}\int \rho^\alpha|v|^{6\times 3^{k}-2}|\nabla v|^2dxdt&= o(\tau_2-\tau_1),\text{ as }\tau_2-\tau_1\to\infty,\label{14-6}\\
			\sup_{0\leq t<\infty}\int \rho |v|^{6\times 3^{k+1}}dx&\leq C. \label{14-7}
		\end{align}
		It remains to verify the induction basis. Proposition~\ref{Prop 5.1} together with \eqref{2-2} yields
		\begin{align}
			\int_{\tau_1}^{\tau_2}\int \rho^\alpha|\nabla v|^2dxdt&= o(\tau_2-\tau_1),\text{ as }\tau_2-\tau_1\to\infty, \\
			\sup_{0\leq t<\infty}\int \rho |v|^{6}dx&\leq C_9,
		\end{align}
		where $C_9>0$ depends only on
		$\gamma,\alpha,\nu,\varepsilon,E_0,
		\underline{\rho_0},\overline{\rho_0}$,
		and $\norm{\rho_0^{1/6}v_0}_{L^{6}}$.
		Therefore, the induction can be carried out successively for
		$k=0,1,\ldots,k_0-1$, which proves that \eqref{13-16} holds with
		$s=6\times3^{k_0}-2$. This completes the proof of the claim.
		
		\noindent\textbf{Step II.}
		We next claim that \eqref{13-16} holds for every
		$s\in(6\times3^{k_0}-2,n_{3,\alpha,\beta})$.
		
		By Step~I, we have already established
		\begin{align}
			\int_{\tau_1}^{\tau_2}\int \rho^\alpha|v|^{6\times 3^{k_0-1}-2}|\nabla v|^2dxdt&= o(\tau_2-\tau_1),\text{ as }\tau_2-\tau_1\to\infty,\\
			\sup_{0\leq t<\infty}\int \rho |v|^{6\times 3^{k_0}}dx&\leq C. \label{14-9}
		\end{align}
		Applying Lemma~\ref{Lem5.5}, we further obtain
		\begin{align}\label{14-8}
			\int_{\tau_1}^{\tau_2}\int \rho^\alpha|v|^{6\times 3^{k_0}-2}|\nabla v|^2dxdt
			=o(\tau_2-\tau_1),\text{ as }\tau_2-\tau_1\to\infty.
		\end{align}
		Next, we fix a sufficiently small constant $\delta^{(2)}>0$. By
		\eqref{3d RT}, \eqref{nabla rho 4}, and \eqref{14-8}, we have
		\begin{align*}
			\int_{\tau_1}^{\tau_2}\int
			\Big(|\nabla\rho|^4+\rho^\alpha|v|^{6\times3^{k_0}-2}|\nabla v|^2\Big)
			dxdt
			=o(\tau_2-\tau_1),
			\quad\text{as }\tau_2-\tau_1\to\infty.
		\end{align*}
		Consequently, there exists $T_0^{(2)}>0$, sufficiently large depending on $\delta^{(2)}$, such that for every $n\in\mathbb{N}$,
		\begin{align*}
			\frac{1}{T_0^{(2)}}
			\int_{nT_0^{(2)}}^{(n+1)T_0^{(2)}}
			\int
			\Big(|\nabla\rho|^4+\rho^\alpha|v|^{6\times3^{k_0}-2}|\nabla v|^2\Big)
			dxdt
			\le\delta^{(2)}.
		\end{align*}
		By the mean value theorem, there exists
		$t_n^{(2)}\in[nT_0^{(2)},(n+1)T_0^{(2)}]$ such that
		\begin{align*}
			\int
			\Big(|\nabla\rho|^4+\rho^\alpha|v|^{6\times3^{k_0}-2}|\nabla v|^2\Big)
			(t_n^{(2)})dx
			\le\delta^{(2)}.
		\end{align*}
		Arguing as in the proof of Lemma~\ref{Lem 3.1}, we choose $\delta^{(2)}>0$ sufficiently small so that
		\begin{align}
			\frac{1}{2}\underline{\rho_0}
			\le
			\rho(x,t_n^{(2)})
			\le
			\frac{3}{2}\overline{\rho_0},
			\qquad
			\forall x\in\mathbb{T}^3,
			\label{14-13}\\
			\int |v|^{6\times3^{k_0}-2}
			|\nabla v|^2(t_n^{(2)})dx
			\le
			\Big(\frac{2}{\underline{\rho_0}}\Big)^\alpha.
			\label{14-14}
		\end{align}
		
		Since $s<n_{3,\alpha,\beta}$, it follows from \eqref{13-2} that
		\begin{align*}
			\frac{d}{dt}\int \rho|v|^{s+2}dx
			\le
			C_{4,s}\int \rho|v|^{s+2}dx+C_{4,s}.
		\end{align*}
		Set $t_{-1}^{(2)}=0$. Applying Gr\"onwall's inequality on each interval
		$[t_n^{(2)},t_{n+1}^{(2)}]$, $n\in\mathbb{N}\cup\{-1\}$, yields
		\begin{align}\label{14-15}
			\begin{split}
				&\quad
				\sup_{t\in[t_n^{(2)},t_{n+1}^{(2)}]}
				\int \rho|v|^{s+2}dx\\
				&\le
				\int \rho|v|^{s+2}(t_n^{(2)})dx
				e^{C_{4,s}(t_{n+1}^{(2)}-t_n^{(2)})}
				+
				e^{C_{4,s}(t_{n+1}^{(2)}-t_n^{(2)})}
				-
				1\\
				&\le
				\int \rho|v|^{s+2}(t_n^{(2)})dx
				e^{2C_{4,s}T_0^{(2)}}
				+
				e^{2C_{4,s}T_0^{(2)}}
				-
				1.
			\end{split}
		\end{align}
		
		It remains to estimate the first term on the right-hand side of
		\eqref{14-15}. Since
		$s<n_{3,\alpha,\beta}\le6\times3^{k_0+1}-2$, it follows from
		\eqref{14-13}, Hölder's inequality, and the Sobolev embedding that
		\begin{align*}
			\int \rho|v|^{s+2}(t_n^{(2)})dx
			&\le
			\frac{3}{2}\overline{\rho_0}
			\|v(t_n^{(2)})\|_{L^{6\times3^{k_0+1}}}^{\,s+2}\\
			&\le
			\frac{3}{2}\overline{\rho_0}
			\||v|^{3^{k_0+1}}(t_n^{(2)})\|_{L^6}^{\frac{s+2}{3^{k_0+1}}}\\
			&\le
			\frac{3}{2}\overline{\rho_0}
			C\||v|^{3^{k_0+1}}(t_n^{(2)})\|_{H^1}^{\frac{s+2}{3^{k_0+1}}}.
		\end{align*}
		The right-hand side is uniformly bounded with respect to $n$ by
		\eqref{14-13}, \eqref{14-14}, and \eqref{14-9}. Consequently,
		\eqref{14-15} yields the desired estimate, thereby completing Step~II.
		
		Finally, when $\alpha=1$, $n_{3,1,0}=\infty$. Therefore, the same induction argument as in Step~I can be carried out successively for
		$k_0=0,1,2,\ldots$, which completes the proof of the proposition.
	\end{proof}
	
	We are now in a position to establish a uniform positive lower bound for the density in the case $\alpha<1$. To this end, we introduce the following notation. For any $(\alpha,\beta,\gamma)$ satisfying \eqref{3d gamma}, we choose
	$\check{q}_3$, depending only on $\alpha$ and $\beta$, such that
	\begin{align}\label{check q3}
		\frac{2\alpha}{1-\alpha}<\check{q}_3<n_{3,\alpha,\beta}.
	\end{align}
	The existence of such a $\check{q}_3$ is guaranteed by the assumption
	$\beta\in[0,\beta_3^{+}(\alpha))$, which implies that
	$\frac{2\alpha}{1-\alpha}<n_{3,\alpha,\beta}$.
	Then Proposition~\ref{Prop 5.4} immediately yields
	\begin{align}\label{14-0}
		\sup_{0\leq t<\infty}\int \rho|v|^{\check{q}_3+2}dx\leq C,
	\end{align}
	where the constant $C>0$ depends only on
	$\gamma,\alpha,\nu,\varepsilon,E_0,
	\underline{\rho_0},\overline{\rho_0},
	\hat{q}_3,\check{q}_3,
	\|\rho_0^{1/(\hat{q}_3+2)}v_0\|_{L^{\hat{q}_3+2}}$,
	and
	$\|\rho_0^{1/(\check{q}_3+2)}v_0\|_{L^{\check{q}_3+2}}$.
	
	With the above uniform integrability estimate for the effective velocity at hand, we are ready to establish a uniform positive lower bound for the density.
	\begin{prop}\label{Prop 7.2}
		Assume that \eqref{3d gamma} holds. Then there exists a positive constant $C>0$, depending only on
		$\gamma,\alpha,\nu,\varepsilon,E_0,
		\underline{\rho_0},\overline{\rho_0},
		\hat{q}_3,\check{q}_3,
		\|\rho_0^{1/(\hat{q}_3+2)}v_0\|_{L^{\hat{q}_3+2}}$,
		and
		$\|\rho_0^{1/(\check{q}_3+2)}v_0\|_{L^{\check{q}_3+2}}$,
		such that
		\begin{align}\label{14-16}
			\sup_{0\leq t<\infty}\norm{\rho^{-1}}_{L^\infty}\leq C.
		\end{align}
	\end{prop}
	\begin{proof}
		We carry out, on each interval $[t_n,t_{n+1}]$, $n\in\mathbb{N}\cup\{-1\}$, the same iteration procedure as that in the proof of Proposition~5.4 of \cite{Gu-Huang-Meng-Zhou-2026} to derive a positive lower bound for the density on $[t_n,t_{n+1}]$. By Lemma~\ref{Lem 3.1}, the lengths of the intervals $[t_n,t_{n+1}]$ are uniformly bounded from above, and the density at each initial time $t_n$ admits uniform positive upper and lower bounds. Moreover, the effective velocity estimate \eqref{14-0} is uniform in time. Consequently, the lower bound for the density obtained through the iteration on each interval is independent of $n$, which yields the desired uniform positive lower bound for the density. For brevity, we omit the details, since the proof follows that of Proposition~5.4 in \cite{Gu-Huang-Meng-Zhou-2026} step by step.
	\end{proof}
	
	\section{Higher-order Estimates ($\alpha<1$)}
	In this section, we derive higher-order estimates for the density and the effective velocity under assumption \eqref{3d gamma}, for which the effective velocity formulation reads as
	\begin{equation}
		\label{0}
		\left\{
		\begin{array}{l}
			\rho_t+\div(\rho v)-c\Delta\rho^\alpha=0,\\
			\rho v_t+\rho u\cdot\nabla v+\nabla P=\divg(\rho^\alpha\mathbb{F}[v]),\\
			(\rho, v)(x,0)=(\rho_0(x),v_0(x)),
		\end{array}
		\right.
	\end{equation}
	where
	\[
	\mathbb{F}[v]
	=\nu\nabla v
	+(\nu-c)(\nabla v)^\top
	+(\alpha-1)(2\nu-c)\divg v\,\mathbb{I}.
	\]
	As an application of these estimates, we further establish the large-time behavior of both the density and the effective velocity.
	
	The following proposition is a consequence of the energy estimates, the uniform upper and lower bounds for the density, and the higher integrability of the effective velocity.
	\begin{prop}\label{Prop 8.0}
		Assume that \eqref{3d gamma} holds. Then there exists a constant $M>0$, depending only on
		$\gamma,\alpha,\nu,\varepsilon,E_0,
		\underline{\rho_0},\overline{\rho_0},
		\hat{q}_3,\check{q}_3,
		\|\rho_0^{1/(\hat{q}_3+2)}v_0\|_{L^{\hat{q}_3+2}}$,
		and
		$\|\rho_0^{1/(\check{q}_3+2)}v_0\|_{L^{\check{q}_3+2}}$,
		such that
		\begin{align}\label{14-1}
			\sup_{0\leq t<\infty}\big(\norm{\rho}_{L^\infty}+\norm{\rho^{-1}}_{L^\infty}\big)\leq M.
		\end{align}
		Moreover, there exist a constant $C>0$ and an exponent $q>3$, depending only on
		$\gamma,\alpha,\nu,\varepsilon,E_0,M$, and the initial data, such that
		\begin{align}\label{hhh00}
			\sup_{0\leq t<\infty}\big(\norm{v}_{L^{q}}+\norm{\nabla\rho}_{L^2}\big)
			+\int_0^\infty
			\big(
			\norm{\nabla v}_{L^2}^2
			+\norm{\nabla\rho}_{L^2}^2
			+\norm{\nabla\rho}_{L^4}^4
			+\norm{\nabla^2\rho}_{L^2}^2
			\big)\,dt
			\leq C.
		\end{align}
	\end{prop}
	\begin{proof}
		Estimate \eqref{14-1} follows directly from Propositions \ref{Prop 5.3} and \ref{Prop 7.2}.
		
		Let $q=\check{q}_3+2$. By the choice of $\check{q}_3$ in \eqref{check q3}, we have $\check{q}_3>1$, and hence $q>3$. Combining \eqref{14-1} with \eqref{14-0}, we obtain
		\begin{align*}
			\sup_{0\le t<\infty}\norm{v}_{L^{q}}\le C.
		\end{align*}
		
		Furthermore, it follows from \eqref{14-1}, \eqref{2-1}, and \eqref{2-2} that
		\begin{align*}
			\sup_{0\leq t<\infty}\norm{\nabla\rho}_{L^2}
			+\int_0^\infty
			\big(
			\norm{\nabla v}_{L^2}^2
			+\norm{\nabla\rho}_{L^2}^2
			\big)dt
			\leq C.
		\end{align*}
		
		To estimate the second-order derivatives of the density, we deduce from \eqref{14-1} and \eqref{nabla rho 4} that
		\begin{align*}
        \begin{split}
			\int_0^\infty\int |\nabla^2 \rho|^2dxdt
			&= \int_0^\infty \int |\nabla^2(\rho^{\frac{3\alpha-2}{2}\cdot\frac{2}{3\alpha-2}})|^2dxdt\\
			&\leq C\int_0^\infty\int |\nabla(\rho^{\frac{4-3\alpha}{2}}\nabla\rho^{\frac{3\alpha-2}{2}})|^2dxdt\\
			&\leq C\int_0^\infty\int \rho^{4-3\alpha}|\nabla^2\rho^{\frac{3\alpha-2}{2}}|^2dxdt
			+C\int_0^\infty\int \rho^{4-3\alpha}|\nabla\rho^{\frac{3\alpha-2}{4}}|^4dxdt\\
			&\leq C\int_0^\infty\int |\nabla^2\rho^{\frac{3\alpha-2}{2}}|^2dxdt
			+C\int_0^\infty\int |\nabla\rho^{\frac{3\alpha-2}{4}}|^4dxdt\\
			&\leq C.
         \end{split}
		\end{align*}
		Similarly, the $L^4$ space-time estimate for the first-order derivatives of the density follows from \eqref{14-1} and \eqref{nabla rho 4}.
		
		This completes the proof of Proposition \ref{Prop 8.0}.
	\end{proof}

	The following proposition establishes the decay of the $L^2$-norm of $\nabla\rho$ as $t\to \infty$.
	\begin{prop}\label{Prop 8.1.5}
		Assume that \eqref{3d gamma} holds. Then
		\begin{align}\label{16-4}
			\lim_{t\to\infty}\norm{\nabla\rho(t)}_{L^2}=0.
		\end{align}
	\end{prop}
	\begin{proof}
		The density equation $\eqref{0}_1$ can be rewritten as
		\begin{align*}
			\rho_t+\div(\rho v)-c\alpha\rho^{\alpha-1}\Delta\rho-c\alpha(\alpha-1)\rho^{\alpha-2}|\nabla\rho|^2=0.
		\end{align*}
		Testing the above equation with $\Delta\rho$, and applying Young's inequality, the Sobolev embedding theorem, together with \eqref{14-1} and \eqref{hhh00}, we obtain
		\begin{align*}
			\begin{split}
				&\quad\frac{1}{2}\frac{d}{dt}\int |\nabla\rho|^2dx+\frac{c\alpha}{2}\int \rho^{\alpha-1}|\Delta\rho|^2 dx\\
				&\leq C\int|v|^2 |\nabla\rho|^2dx+C\int \rho^2|\nabla v|^2dx+C\int |\nabla\rho|^4dx\\
				&\leq C\norm{v}_{L^{q}}^2\norm{\nabla\rho}_{H^1}^2+C\norm{\nabla v}_{L^2}^2+C\norm{\nabla \rho}_{L^4}^4\\
				&\leq C\norm{\nabla\rho}_{H^1}^2+C\norm{\nabla \rho}_{L^4}^4+C\norm{\nabla v}_{L^2}^2.
			\end{split}
		\end{align*}
		For convenience, we introduce
		\begin{align*}
			g_0(t):=\frac{1}{2}\|\nabla\rho(t)\|_{L^2}^2.
		\end{align*}
		It follows that, for any $t>2$ and any $s\in[t-1,t]$,
		\begin{align*}
			g_0(t)
			&\leq g_0(s)+\int_s^t g_0'(\zeta)\,d\zeta\\
			&\leq g_0(s)+C\int_{t-1}^t\big(\norm{\nabla\rho}_{H^1}^2+\norm{\nabla \rho}_{L^4}^4+\norm{\nabla v}_{L^2}^2\big)\,d\zeta.
		\end{align*}
		Averaging the above inequality with respect to $s$ over the interval $[t-1,t]$, we obtain
		\begin{align*}
			g_0(t)
			\leq
			\int_{t-1}^{t} g_0(\zeta)\,d\zeta
			+
			C\int_{t-1}^t\big(\norm{\nabla\rho}_{H^1}^2+\norm{\nabla \rho}_{L^4}^4+\norm{\nabla v}_{L^2}^2\big)\,d\zeta.
		\end{align*}
		Passing to the limit as $t\to\infty$ and using \eqref{hhh00}, we conclude that
		\begin{align*}
			\lim_{t\to\infty}g_0(t)=0.
		\end{align*}
		The desired conclusion \eqref{16-4} then follows immediately from the definition of $g_0$. This completes the proof.
	\end{proof}

	\subsection{First-level higher-order estimates}
	
	We begin by recalling the following lemma from \cite[Lemma 6.1]{Gu-Huang-Meng-Zhou-2026}, which plays a crucial role in closing the first-level higher-order estimates.
	
	\begin{lema}\label{Lemma 8.1}
	Suppose that
	\begin{align*}
		\frac{9\sqrt{3}-4\sqrt{2}}{9\sqrt{3}-2\sqrt{2}}
		< \alpha < 1.
	\end{align*}
	Then, there exist a small constant $\eta_\alpha > 0$ 
	depending only on $\alpha$, and a large 
	constant $C > 0$ depending on $\alpha, \nu$, and  
	$\varepsilon$, such that
		\begin{align}\label{ggg0}
			\begin{split}
				\frac{1}{4c}\frac{d}{dt}\int |\nabla\rho|^4 dx
				+ \frac{\eta_\alpha}{2}\int \rho^{\alpha-1}
				|\nabla\rho|^2|\nabla^2\rho|^2 dx
				\leq C\int \rho^{1-\alpha}|\nabla\rho|^2|\nabla(\rho v)|^2 dx.
			\end{split}
		\end{align}
	\end{lema}
	
	We next recall another lemma from \cite[Lemma 6.2]{Gu-Huang-Meng-Zhou-2026}, which reveals the elliptic structure inherent in the effective velocity system and will be repeatedly used in the subsequent analysis.
	
	\begin{lema}\label{Lema 8.2}
		Assume that $\alpha,\nu>0$ and $0<\varepsilon\le\nu$. Let $w\in H^3$ and $F\in H^1$
		satisfy
		\begin{align}\label{10-6}
			\nu\Delta w+\big(\nu-c+(\alpha-1)(2\nu-c)\big) \nabla\divg w=F,
		\end{align}
		where $c=\nu+\sqrt{\nu^2-\varepsilon^2}$. Then there exists a
		constant $C>0$ depending only on $\alpha,\nu,\varepsilon$
		such that
		\begin{align}
			\|\nabla^2 w\|_{L^2}&\leq C\|F\|_{L^2},\label{10-4}\\
			\|\nabla^3 w\|_{L^2}&\leq C\|\nabla F\|_{L^2}. \label{10-5}
		\end{align}
	\end{lema}
	
	The following proposition establishes the first level of higher-order estimates.
	\begin{prop}\label{Prop 8.2}
		Assume that \eqref{3d gamma} holds. Then there exists a constant $C>0$, depending only on
		$\gamma,\alpha,\nu,\varepsilon,E_0,M$, and the initial data, such that
		\begin{align}\label{hhh0}
			\begin{split}
				\sup_{0\leq t<\infty}\big(\norm{\nabla^2\rho}_{L^2}&+\norm{\nabla v}_{L^2}+\norm{\rho_t}_{L^2}\big)\\
				&+\int_0^\infty \big(\norm{\nabla^3 \rho}_{L^2}^2+\norm{\nabla\rho_t}_{L^2}^2+\norm{v_t}_{L^2}^2+\norm{\nabla^2 v}_{L^2}^2\big)dt\leq C.
			\end{split}
		\end{align}
	\end{prop}
	\begin{proof}
		First, we rewrite the density equation $\eqref{0}_1$ in the form
		\begin{align}\label{hhh0.5}
			\rho_t+\div(\rho v)-c\alpha\rho^{\alpha-1}\Delta\rho-c\alpha(\alpha-1)\rho^{\alpha-2}|\nabla\rho|^2=0.
		\end{align}
		Testing the above equation against $\Delta\rho_t$, integrating over $\mathbb{T}^3$, and performing integration by parts, we infer from Young's inequality together with \eqref{14-1} that
		\begin{align}\label{hhh1}
			\begin{split}
				&\quad\frac{c\alpha}{2}\frac{d}{dt}\int \rho^{\alpha-1}|\Delta\rho|^2dx+\int |\nabla\rho_t|^2dx\\
				&\leq C\int |\nabla\rho_t||\nabla^2(\rho v)|dx+C\int |\nabla\rho_t|(|\nabla\rho||\nabla^2\rho|+|\nabla\rho|^3)dx+\frac{c\alpha}{2}\int (\rho^{\alpha-1})_t|\Delta\rho|^2 dx\\
				&\leq \frac{1}{2}\int |\nabla\rho_t|^2dx+C\int (|\nabla^2(\rho v)|^2+ |\nabla\rho|^2|\nabla^2 \rho|^2+|\nabla\rho|^6)dx+\frac{c\alpha}{2}\int (\rho^{\alpha-1})_t|\Delta\rho|^2 dx.
			\end{split}
		\end{align}
		Moreover, applying equation $\eqref{0}_1$ once more, one can choose a sufficiently small constant $\varepsilon_0>0$ such that
		\begin{align}\label{hhh2}
			\varepsilon_0\int |\nabla\Delta\rho^\alpha|^2dx\leq \frac{1}{4}\int |\nabla\rho_t|^2dx+C\int |\nabla^2(\rho v)|^2dx.
		\end{align}
		Combining \eqref{hhh1} with \eqref{hhh2}, we arrive at
		\begin{align}\label{hhh3}
			\begin{split}
				&\quad\frac{c\alpha}{2}\frac{d}{dt}\int \rho^{\alpha-1}|\Delta\rho|^2dx+\frac{1}{4}\int |\nabla\rho_t|^2dx+\varepsilon_0\int |\nabla\Delta\rho^\alpha|^2dx\\
				&\leq A_1\int |\nabla^2 v|^2dx+C\int |\nabla\rho|^2|\nabla v|^2dx+C\int|\nabla^2\rho|^2|v|^2dx\\
				&\quad+C\int |\nabla\rho|^2|\nabla^2 \rho|^2dx+C\int|\nabla\rho|^6dx+\frac{c\alpha}{2}\int (\rho^{\alpha-1})_t|\Delta\rho|^2 dx,
			\end{split}
		\end{align}
		where $A_1$ denotes the coefficient of the first term on the right-hand side, depending only on the quantities specified in Proposition~\ref{Prop 8.2}.
		
		Next, we multiply $\eqref{0}_2$ by $v_t$, integrate the resulting identity over $\mathbb{T}^3$, and apply integration by parts, Young's inequality, together with \eqref{14-1}, to obtain
		\begin{align}\label{hhh4}
			\begin{split}
				&\quad\frac{1}{2}\frac{d}{dt}\int\rho^\alpha\mathbb{F}[v]:\nabla vdx+\int \rho |v_t|^2dx\\
				&\leq C\int \rho|u||\nabla v||v_t|dx+C\int|\nabla\rho||v_t|dx+\int\frac{(\rho^\alpha)_t}{2}\mathbb{F}[v]:\nabla vdx\\
				&\leq \frac{1}{2}\int \rho |v_t|^2dx+C\int\left( |v|^2|\nabla v|^2+|\nabla\rho|^2|\nabla v|^2+|\nabla\rho|^2\right)dx+\int\frac{(\rho^\alpha)_t}{2}\mathbb{F}[v]:\nabla vdx,
			\end{split}
		\end{align}
		where we have used the identity
		\begin{align*}
			\mathbb{F}[v_t]:\nabla v=\mathbb{F}[v]:\nabla v_t.
		\end{align*}
		Furthermore, equation $\eqref{0}_2$ can be viewed as the following elliptic system:
		\begin{align}\label{ell sys}
			\begin{split}
				&\quad\nu\Delta v+\left((\nu-c)+(\alpha-1)(2\nu-c)\right)\nabla\divg v\\
				&=\frac{1}{\rho^\alpha}\left(\rho v_t+\rho (v-c\alpha\rho^{\alpha-2}\nabla\rho)\cdot\nabla v+\nabla P-\nabla\rho^\alpha\cdot\mathbb{F}[v]\right).
			\end{split}
		\end{align}
		Applying Lemma~\ref{Lema 8.2} together with \eqref{14-1}, we can choose a sufficiently small constant $\varepsilon_1>0$ such that
		\begin{align}\label{hhh4.5}
			\varepsilon_1\int|\nabla^2 v|^2dx\leq \frac{1}{4}\int \rho |v_t|^2dx+C\int  \left(|v|^2|\nabla v|^2+|\nabla\rho|^2|\nabla v|^2+|\nabla\rho|^2\right)dx.
		\end{align}
		Combining \eqref{hhh4} and \eqref{hhh4.5}, we deduce
		\begin{align}\label{hhh5}
			\begin{split}
				&\quad\frac{1}{2}\frac{d}{dt}\int\rho^\alpha\mathbb{F}[v]:\nabla vdx+\frac{1}{4}\int \rho |v_t|^2dx+\varepsilon_1\int|\nabla^2 v|^2dx\\
				&\leq C\int |v|^2|\nabla v|^2dx+C\int|\nabla\rho|^2|\nabla v|^2dx+C\int|\nabla\rho|^2dx+\int\frac{(\rho^\alpha)_t}{2}\mathbb{F}[v]:\nabla vdx.
			\end{split}
		\end{align}
		
		Finally, multiplying \eqref{hhh5} by $2A_1/\varepsilon_1$ and adding the resulting inequality to \eqref{hhh3}, we conclude that
		\begin{align}\label{I_0}
			\begin{split}
				&\quad\frac{c\alpha}{2}\frac{d}{dt}\int \rho^{\alpha-1}|\Delta\rho|^2dx+\frac{A_1}{\varepsilon_1}\frac{d}{dt}\int\rho^\alpha\mathbb{F}[v]:\nabla vdx\\
				&\quad+\frac{1}{4}\int |\nabla\rho_t|^2dx+\varepsilon_0\int |\nabla\Delta\rho^\alpha|^2dx+\frac{A_1}{2\varepsilon_1}\int \rho |v_t|^2dx+A_1\int|\nabla^2 v|^2dx\\
				&\leq C\int |\nabla\rho|^2|\nabla v|^2dx+C\int|\nabla^2\rho|^2|v|^2dx+C\int |v|^2|\nabla v|^2dx+C\int|\nabla\rho|^2dx\\
				&\quad+C\int |\nabla\rho|^2|\nabla^2 \rho|^2dx+C\int|\nabla\rho|^6dx\\
				&\quad+\frac{c\alpha}{2}\int (\rho^{\alpha-1})_t|\Delta\rho|^2 dx+\int\frac{(\rho^\alpha)_t}{2}\mathbb{F}[v]:\nabla vdx=:\sum_{i=1}^{8}I_i.
			\end{split}
		\end{align}
		It remains to estimate the terms $I_1,\ldots,I_8$ separately.
		
		For the term $I_1$, integrating by parts gives
		\begin{align*}
			\begin{split}
				I_1
				&=C\sum_{i,j=1}^3\int |\nabla\rho|^2(\partial_i v_j)^2dx\\
				&\leq C\int |v||\nabla v||\nabla\rho||\nabla^2\rho|dx
				+C\int |v||\nabla^2v||\nabla\rho|^2dx\\
				&=:I_1^1+I_1^2.
			\end{split}
		\end{align*}
		We begin with the estimate of $I_1^1$. By Hölder's inequality, Young's inequality, the Sobolev embedding theorem, \eqref{14-1}, \eqref{hhh00}, and the fact that $q>3$, we obtain
		\begin{align*}
			\begin{split}
				I_1^1
				&\leq C\|v\|_{L^{q}}
				\|\nabla v\|_{L^{\frac{6q}{2q-3}}}
				\||\nabla\rho||\nabla^2\rho|\|_{L^{\frac{6q}{4q-3}}}\\
				&\leq
				C\|v\|_{L^{q}}
				\|\nabla v\|_{L^{\frac{6q}{2q-3}}}
				\||\nabla\rho||\nabla^2\rho|\|_{L^2}^{\frac{2q+3}{3q}}
				\||\nabla\rho||\nabla^2\rho|\|_{L^1}^{\frac{q-3}{3q}}\\
				&\leq
				C\|\nabla v\|_{H^1}
				\||\nabla\rho||\nabla^2\rho|\|_{L^2}^{\frac{2q+3}{3q}}
				\||\nabla\rho||\nabla^2\rho|\|_{L^1}^{\frac{q-3}{3q}}\\
				&\leq
				\frac{A_1}{32}\|\nabla v\|_{H^1}^2
				+\delta\int |\nabla\rho|^2|\nabla^2\rho|^2dx
				+C_\delta
				\left(
				\int |\nabla\rho||\nabla^2\rho|dx
				\right)^2\\
				&\leq
				\frac{A_1}{32}\|\nabla v\|_{H^1}^2
				+\delta\int |\nabla\rho|^2|\nabla^2\rho|^2dx
				+C_\delta\int |\nabla^2\rho|^2dx,
			\end{split}
		\end{align*}
		where $\delta>0$ is a sufficiently small constant to be fixed later.
		To control the second term on the right-hand side of the above inequality, we first observe from \eqref{14-1} and Young's inequality that
		\begin{align}\label{hhh7}
			\begin{split}
				\delta\int |\nabla\rho|^2|\nabla^2\rho|^2dx
				&\leq
				\delta C
				\int
				\rho^{2-2\alpha}
				|\nabla\rho^\alpha|^2
				|\nabla(\rho^{1-\alpha}\nabla\rho^\alpha)|^2dx\\
				&\leq
				\delta C
				\int
				|\nabla\rho^\alpha|^2
				|\nabla^2\rho^\alpha|^2dx
				+\delta C
				\int
				|\nabla\rho^\alpha|^6dx.
			\end{split}
		\end{align}
		We next estimate the first term on the right-hand side of \eqref{hhh7}. Integrating by parts twice and applying \eqref{14-1} together with Young's inequality, we infer
		\begin{align*}
			\begin{split}
				\int |\nabla\rho^\alpha|^2|\nabla^2\rho^\alpha|^2dx
				&=
				\sum_{i=1}^3
				\int
				(\partial_i\rho^\alpha)^2
				|\nabla^2\rho^\alpha|^2dx\\
				&\leq
				C
				\int
				\rho^\alpha
				|\nabla\rho^\alpha|
				|\nabla^2\rho^\alpha|
				|\nabla^3\rho^\alpha|dx
				+
				C
				\int
				\rho^\alpha
				|\nabla^2\rho^\alpha|^3dx\\
				&\leq
				C
				\int
				|\nabla\rho^\alpha|
				|\nabla^2\rho^\alpha|
				|\nabla^3\rho^\alpha|dx
				+
				C
				\sum_{i,j=1}^3
				\int
				(\partial_{ij}\rho^\alpha)^2
				|\nabla^2\rho^\alpha|dx\\
				&\leq
				C
				\int
				|\nabla\rho^\alpha|
				|\nabla^2\rho^\alpha|
				|\nabla^3\rho^\alpha|dx\\
				&\leq
				\frac12
				\int
				|\nabla\rho^\alpha|^2
				|\nabla^2\rho^\alpha|^2dx
				+
				C
				\int
				|\nabla^3\rho^\alpha|^2dx,
			\end{split}
		\end{align*}
		which immediately yields
		\begin{align}\label{14-2}
			\int
			|\nabla\rho^\alpha|^2
			|\nabla^2\rho^\alpha|^2dx
			\leq
			C
			\int
			|\nabla^3\rho^\alpha|^2dx.
		\end{align}
		It remains to estimate the second term in \eqref{hhh7}. Repeating the same argument based on integration by parts, together with \eqref{14-1}, Young's inequality, and \eqref{14-2}, we derive
		\begin{align*}
			\int |\nabla\rho^\alpha|^6dx
			&=
			\sum_{i=1}^3
			\int
			(\partial_i\rho^\alpha)^2
			|\nabla\rho^\alpha|^4dx\\
			&\leq
			C
			\int
			\rho^\alpha
			|\nabla\rho^\alpha|^4
			|\nabla^2\rho^\alpha|dx\\
			&\leq
			C
			\int
			|\nabla\rho^\alpha|^4
			|\nabla^2\rho^\alpha|dx\\
			&=
			C
			\sum_{i=1}^3
			\int
			(\partial_i\rho^\alpha)^2
			|\nabla\rho^\alpha|^2
			|\nabla^2\rho^\alpha|dx\\
			&\leq
			C
			\int
			|\nabla\rho^\alpha|^2
			|\nabla^2\rho^\alpha|^2dx
			+
			C
			\int
			|\nabla\rho^\alpha|^3
			|\nabla^3\rho^\alpha|dx\\
			&\leq
			\frac12
			\int
			|\nabla\rho^\alpha|^6dx
			+
			C
			\int
			|\nabla^3\rho^\alpha|^2dx,
		\end{align*}
		and hence
		\begin{align}\label{14-3}
			\int
			|\nabla\rho^\alpha|^6dx
			\leq
			C
			\int
			|\nabla^3\rho^\alpha|^2dx.
		\end{align}
		Combining \eqref{14-2}, \eqref{14-3}, and \eqref{hhh7}, we arrive at
		\begin{align}\label{14-4}
			\delta
			\int
			|\nabla\rho|^2
			|\nabla^2\rho|^2dx
			\leq
			\delta A_2
			\int
			|\nabla^3\rho^\alpha|^2dx,
		\end{align}
		where $A_2>0$ depends only on the quantities appearing in Proposition~\ref{Prop 8.2}. Choosing $\delta>0$ sufficiently small such that
		\[
		\delta A_2\leq \frac{\varepsilon_0}{32},
		\]
		and combining the above estimate with \eqref{14-1}, we conclude that
		\begin{align}\label{I_11'}
			\begin{split}
				I_1^1&=C\int |v||\nabla v||\nabla\rho||\nabla^2\rho|dx\\
				&\leq\frac{A_1}{32}\|\nabla v\|_{H^1}^2+\frac{\varepsilon_0}{32}\|\nabla^3\rho^\alpha\|_{L^2}^2+C\int |\nabla^2\rho|^2dx.
			\end{split}
		\end{align}
		For $I_1^2$, Hölder's inequality and the interpolation inequality lead to
		\begin{align}\label{I_1^2'0}
			\begin{split}
				I_1^2&\leq C\|v\|_{L^q}\|\nabla^2 v\|_{L^2}\|\nabla\rho^\alpha\|_{L^{12}}^{\frac{6q+12}{5q}}\|\nabla\rho^\alpha\|_{L^2}^{\frac{4q-12}{5q}}.
			\end{split}
		\end{align}
		To estimate the $L^{12}$-norm of $\nabla\rho^\alpha$, we perform integration by parts and apply Hölder's inequality together with \eqref{14-1}, obtaining
		\begin{align*}
			\begin{split}
				\int|\nabla\rho^\alpha|^{12}dx
				&=C\sum_{i=1}^3\int (\partial_i \rho^\alpha)^2|\nabla\rho^\alpha|^{10}dx\\
				&\leq C\int \rho^\alpha|\nabla\rho^\alpha|^{10}|\nabla^2 \rho^\alpha|dx \\
				&\leq C\int |\nabla\rho^\alpha|^{10}|\nabla^2\rho^\alpha|dx\\
				&\leq C\|\nabla\rho^\alpha\|_{L^{12}}^{10}\|\nabla^2\rho^\alpha\|_{L^6}.
			\end{split}
		\end{align*}
		Consequently, the Sobolev embedding theorem yields
		\begin{align}\label{14-10}
			\|\nabla\rho^{\alpha}\|_{L^{12}}
			\leq
			C\|\nabla^2\rho^\alpha\|_{L^6}^{\frac{1}{2}}
			\leq
			C\|\nabla^2 \rho^\alpha\|_{H^1}^{\frac{1}{2}}.
		\end{align}
		Inserting \eqref{14-10} into \eqref{I_1^2'0}, and then using \eqref{14-1}, \eqref{hhh00}, together with Young's inequality, we arrive at
		\begin{align}\label{I_1^2'}
			\begin{split}
				I_1^2&=C\int |v||\nabla^2 v||\nabla\rho|^2dx\\
				&\leq C\|v\|_{L^q}\|\nabla^2 v\|_{L^2}\|\nabla^2\rho^\alpha\|_{H^1}^{\frac{3q+6}{5q}}\|\nabla\rho^\alpha\|_{L^2}^{\frac{4q-12}{5q}}\\
				&\leq \frac{A_1}{32}\|\nabla^2 v\|_{L^2}^2+\frac{\varepsilon_0}{32}\|\nabla^2 \rho^\alpha\|_{H^1}^2+C\|\nabla\rho^\alpha\|_{L^2}^4\\
				&\leq \frac{A_1}{32}\|\nabla^2 v\|_{L^2}^2+\frac{\varepsilon_0}{32}\|\nabla^2 \rho^\alpha\|_{H^1}^2+C\|\nabla\rho\|_{L^2}^2.
			\end{split}
		\end{align}
		
		For the term $I_2$, Hölder's inequality together with \eqref{14-1} yields
		\begin{align}\label{I_222'}
			\begin{split}
				I_2&=C\int |v|^2|\nabla^2\rho|^2dx\\
				&\leq C\|v\|_{L^q}^2\|\nabla^2 \rho\|_{L^{\frac{2q}{q-2}}}^2\\
				&\leq C\|v\|_{L^q}^2\|\nabla(\rho^{1-\alpha}\nabla\rho^\alpha)\|_{L^{\frac{2q}{q-2}}}^2\\
				&\leq C\|v\|_{L^q}^2\|\nabla^2 \rho^\alpha\|_{L^{\frac{2q}{q-2}}}^2+C\|v\|_{L^q}^2\|\nabla\rho^\alpha\|_{L^{\frac{4q}{q-2}}}^4.
			\end{split}
		\end{align}
		The first term on the right-hand side can be estimated by means of the Gagliardo--Nirenberg inequality, the Sobolev embedding theorem, Young's inequality, together with \eqref{hhh00} and \eqref{14-1}, as follows:
		\begin{align}\label{42-1'}
			\begin{split}
				&\quad C\|v\|_{L^q}^2\|\nabla^2 \rho^\alpha\|_{L^{\frac{2q}{q-2}}}^2\\
				&\leq C  \|\nabla^2 \rho^\alpha\|_{L^2}^{\frac{2q-6}{q}}\|\nabla^2 \rho^\alpha\|_{H^1}^{\frac{6}{q}}\\
				&\leq \frac{\varepsilon_0}{32}\|\nabla^2 \rho^\alpha\|_{H^1}^2+C\|\nabla^2 \rho^\alpha\|_{L^2}^2\\
				&\leq \frac{\varepsilon_0}{32}\|\nabla^2 \rho^\alpha\|_{H^1}^2+C\|\nabla^2 \rho\|_{L^2}^2+C\|\nabla\rho\|_{L^4}^4.
			\end{split}
		\end{align}
		As for the second term, combining the Gagliardo--Nirenberg inequality with \eqref{14-10} gives
		\begin{align}\label{42-2'}
			\begin{split}
				&\quad C\|v\|_{L^q}^2\|\nabla\rho^\alpha\|_{L^{\frac{4q}{q-2}}}^4\\
				&\leq C  \|\nabla\rho^\alpha\|_{L^2}^{\frac{8q-24}{5q}}\|\nabla\rho^\alpha\|_{L^{12}}^{\frac{12q+24}{5q}}\\
				&\leq C  \|\nabla\rho^\alpha\|_{L^2}^{\frac{8q-24}{5q}}\|\nabla^2\rho^\alpha\|_{H^1}^{\frac{6q+12}{5q}}\\
				&\leq \frac{\varepsilon_0}{32}\|\nabla^2 \rho^\alpha\|_{H^1}^2+C\|\nabla\rho^\alpha\|_{L^2}^4\\
				&\le \frac{\varepsilon_0}{32}\|\nabla^2 \rho^\alpha\|_{H^1}^2+C\|\nabla\rho\|_{L^4}^4.
			\end{split}
		\end{align}
		Substituting \eqref{42-1'} and \eqref{42-2'} into \eqref{I_222'}, we arrive at
		\begin{align}\label{I_2'}
			I_2
			\leq
			\frac{\varepsilon_0}{16}\|\nabla^2\rho^\alpha\|_{H^1}^2
			+C\|\nabla^2 \rho\|_{L^2}^2
			+C\|\nabla\rho\|_{L^4}^4.
		\end{align}

		For  the term $I_3$, combining H\"older's inequality, the Gagliardo--Nirenberg inequality, Young's inequality, and \eqref{hhh00}, we obtain
		\begin{align}\label{I_3'}
			\begin{split}
				I_3&= C\int |v|^2|\nabla v|^2dx\\
				&\leq C\|v\|_{L^q}^2\|\nabla v\|_{L^{\frac{2q}{q-2}}}^2\\
				&\leq C\|v\|_{L^q}^2\|\nabla v\|_{L^2}^{\frac{2q-6}{q}}\|\nabla v\|_{H^1}^{\frac{6}{q}}\\
				&\leq \frac{A_1}{32}\|\nabla^2 v\|_{L^2}^2+C\|\nabla v\|_{L^2}^2.
			\end{split}
		\end{align}
		
		For  the term $I_4$, 
		\begin{align}\label{I_4'}
			\begin{split}
				I_4=C\int |\nabla\rho|^2dx.
			\end{split}
		\end{align}
		
		For the combined term $I_5+I_6$, integration by parts, together with \eqref{14-1} and Young's inequality, gives
		\begin{align}\label{11-222}
			\begin{split}
				\int |\nabla\rho|^6dx
				&\leq \sum_{i=1}^3\int |\nabla\rho|^4(\partial_i\rho)^2dx\\
				&\leq C\int \rho |\nabla\rho|^4|\nabla^2\rho|dx\\
				&\leq \frac{1}{2}\int |\nabla\rho|^6dx
				+C\int |\nabla\rho|^2|\nabla^2\rho|^2dx,
			\end{split}
		\end{align}
		which implies that
		\begin{align}\label{I_5+I_6'}
			\begin{split}
				I_5+I_6
				\leq C\int |\nabla\rho|^2|\nabla^2\rho|^2dx.
			\end{split}
		\end{align}
		
		For $I_7$, it follows from $\eqref{0}_1$, integration by parts, \eqref{14-1}, and Young's inequality that
		\begin{align}\label{I_7'00}
			\begin{split}
				I_7
				&= -\frac{c\alpha(\alpha-1)}{2}\sum_{i=1}^3\int\rho^{\alpha-2}\partial_i(\rho v_i)|\Delta\rho|^2dx+\frac{c^2\alpha(\alpha-1)}{2}\sum_{i=1}^3\int\rho^{\alpha-2}\Delta\rho^\alpha\Delta\rho \partial_{ii}\rho dx\\
				&= \frac{c\alpha(\alpha-1)}{2}\sum_{i=1}^3\int\rho v_i\partial_i(\rho^{\alpha-2}|\Delta\rho|^2)dx-\frac{c^2\alpha(\alpha-1)}{2}\sum_{i=1}^3\int\partial_i(\rho^{\alpha-2}\Delta\rho^\alpha\Delta\rho)\partial_{i}\rho dx\\
				&\leq C\int \rho|v||\Delta\rho||\nabla^3 \rho|dx+C\int \rho|v||\nabla\rho||\Delta\rho|^2dx\\
				&\quad+C\int|\nabla\rho|^2|\Delta\rho^\alpha||\nabla^2\rho|dx+C\int |\nabla\Delta\rho^\alpha||\nabla^2\rho||\nabla\rho|dx+C\int |\Delta\rho^\alpha||\nabla^3\rho||\nabla\rho|dx.
			\end{split}
		\end{align}
		Notice that \eqref{14-1} implies
		\begin{align}\label{33-1}
			\begin{split}
				|\nabla^3\rho|&\leq C\sum_{i,j,k}|\partial_{ijk}\rho|\leq C\sum_{i,j,k} |\partial_{ij}\big(\rho^{1-\alpha}\partial_k\rho^\alpha\big)|\\
				&\leq C\sum_{i,j,k}|\partial_{ij}\rho^{1-\alpha}||\partial_k\rho^\alpha|+C\sum_{i,j,k}|\partial_i\rho^{1-\alpha}||\partial_{jk}\rho^\alpha|+C\sum_{i,j,k}\rho^{1-\alpha}|\partial_{ijk}\rho^\alpha|\\
				&\leq C\sum_{i,j,k}\rho^{1-2\alpha}|\partial_{ij}\rho^{\alpha}||\partial_k\rho^\alpha|+C\sum_{i,j,k}\rho^{1-3\alpha}|\partial_{i}\rho^{\alpha}||\partial_j\rho^\alpha||\partial_k\rho^\alpha|\\
				&\quad+C\sum_{i,j,k}\rho^{1-\alpha}|\partial_{ijk}\rho^\alpha|\\
				&\leq C\big(|\nabla^3\rho^\alpha|+|\nabla^2\rho^\alpha||\nabla\rho^\alpha|+|\nabla\rho^\alpha|^3 \big).
			\end{split}
		\end{align}
		Consequently, Young's inequality yields
		\begin{align}\label{40-2}
			\begin{split}
				&\quad C\int \rho|v||\Delta\rho||\nabla^3 \rho|dx+C\int |\Delta\rho^\alpha||\nabla^3\rho||\nabla\rho|dx\\
				&\le C\int (|v||\Delta\rho|+|\Delta\rho||\nabla\rho|+|\nabla\rho|^3)(|\nabla^3\rho^\alpha|+|\nabla^2\rho^\alpha||\nabla\rho^\alpha|+|\nabla\rho^\alpha|^3 )dx\\
				&\le \frac{\varepsilon_0}{64}\|\nabla^3\rho^\alpha\|_{L^2}^2+C\int(|v|^2|\Delta\rho|^2+|\nabla^2\rho|^2|\nabla\rho|^2+|\nabla\rho|^6+|\nabla^2\rho^\alpha|^2|\nabla\rho^\alpha|^2+|\nabla\rho^\alpha|^6) dx\\
				&\le \frac{\varepsilon_0}{64}\|\nabla^3\rho^\alpha\|_{L^2}^2+C\int(|v|^2|\Delta\rho|^2+|\nabla\rho|^2|\nabla^2\rho|^2+|\nabla\rho|^6)dx,
			\end{split}
		\end{align}
		where the last inequality follows from
		\begin{align*}
			|\nabla^2\rho^\alpha|\le C|\nabla^2\rho|+C|\nabla\rho|^2,\qquad
			|\nabla\rho^\alpha|\le C|\nabla\rho|.
		\end{align*}
		The remaining terms can be estimated directly by Young's inequality:
		\begin{align}\label{40-3}
			\begin{split}
				C\int \rho|v||\nabla\rho||\Delta\rho|^2dx
				&\le C\int|v|^2|\Delta\rho|^2+C\int|\nabla\rho|^2|\nabla^2\rho|^2dx,\\
				C\int|\nabla\rho|^2|\Delta\rho^\alpha||\nabla^2\rho| dx
				&\le C\int |\nabla\rho|^2(|\nabla^2\rho|+|\nabla\rho|^2)|\nabla^2\rho|dx\\
				&\le C\int |\nabla\rho|^2|\nabla^2\rho|^2dx+C\int |\nabla\rho|^4|\nabla^2\rho|dx\\
				&\le C\int |\nabla\rho|^2|\nabla^2\rho|^2dx+C\int |\nabla\rho|^6dx,\\
				C\int |\nabla\Delta\rho^\alpha||\nabla^2\rho||\nabla\rho|dx
				&\le \frac{\varepsilon_0}{64}\|\nabla^3\rho^\alpha\|_{L^2}^2+C\int |\nabla\rho|^2|\nabla^2\rho|^2dx.
			\end{split}
		\end{align}
		Substituting \eqref{40-2} and \eqref{40-3} into \eqref{I_7'00}, and then applying \eqref{11-222} together with \eqref{I_2'}, we arrive at
		\begin{align}\label{I_7'}
			\begin{split}
				I_7
				&\leq \frac{\varepsilon_0}{32}\|\nabla^3\rho^\alpha\|_{L^2}^2+C\int |v|^2|\Delta\rho|^2dx+ C\int \big(|\nabla\rho|^2|\nabla^2\rho|^2+|\nabla\rho|^6\big)dx\\
				&\leq \frac{\varepsilon_0}{32}\|\nabla^3\rho^\alpha\|_{L^2}^2+C\int |v|^2|\Delta\rho|^2dx+ C\int |\nabla\rho|^2|\nabla^2\rho|^2dx\\
				&\leq \frac{3\varepsilon_0}{32}\|\nabla^2\rho^\alpha\|_{H^1}^2+C\int |\nabla\rho|^2|\nabla^2\rho|^2dx+C\|\nabla^2 \rho\|_{L^2}^2+C\|\nabla\rho\|_{L^4}^4.
			\end{split}
		\end{align}
		
		For $I_8$, using integration by parts,  Young's inequality, and \eqref{14-1}, we have
		\begin{align}\label{I_8,}
			\begin{split}
				I_8&=\frac{\alpha}{2}\int\rho^{\alpha-1}\rho_t\mathbb{F}[v]:\nabla vdx\\
				&=\frac{c\alpha}{2}\sum_{i, j=1}^3\int \rho^{\alpha-1}\Delta\rho^\alpha(\mathbb{F}[v])_{ij}\partial_i v_jdx-\frac{\alpha}{2}\sum_{i=1}^3\int \rho^{\alpha-1}\partial_i(\rho v_i)\mathbb{F}[v]:\nabla vdx\\
				&=- \frac{c\alpha}{2}\sum_{i, j=1}^3\int \partial_i(\rho^{\alpha-1}\Delta\rho^\alpha(\mathbb{F}[v])_{ij})v_j dx+\frac{\alpha}{2}\sum_{i=1}^3\int\rho v_i\partial_i(\rho^{\alpha-1}\mathbb{F}[v]:\nabla v)dx\\
				&\leq C\int |\nabla\rho||\Delta\rho^\alpha||\nabla v||v|dx+C\int |\nabla\Delta\rho^\alpha||\nabla v||v|dx+C\int |\Delta\rho^\alpha||\nabla^2 v||v|dx\\
				&\quad+C\int |\nabla\rho||v||\nabla v|^2dx+C\int |v||\nabla v||\nabla^2 v|dx\\
				&\leq \frac{\varepsilon_0}{32}\|\nabla^3 \rho^\alpha\|_{L^2}^2+\frac{A_1}{32}\|\nabla^2 v\|_{L^2}^2+C\int |\nabla\rho||\Delta\rho||v||\nabla v|dx\\
				&\quad+C\int |v|^2|\nabla v|^2dx+C\int |v|^2|\Delta\rho|^2dx+C\int |v|^2|\nabla\rho|^4dx+C\int |\nabla\rho|^2|\nabla v|^2dx,
			\end{split}
		\end{align}
		where in the last inequality we have used
		\begin{align*}
			\begin{split}
				\int |\nabla\rho||\Delta\rho^\alpha||\nabla v||v|dx&\leq C\int (|\nabla\rho|^2+|\Delta\rho|)|\nabla\rho||\nabla v||v|dx\\
				&\leq C\int |v|^2|\nabla\rho|^4dx+C\int |\nabla\rho|^2|\nabla v|^2dx+C\int |\nabla\rho||\Delta\rho||v||\nabla v|dx,\\
				\int |\Delta\rho^\alpha|^2|v|^2dx&\leq C\int (|\nabla\rho|^4+|\Delta\rho|^2)|v|^2dx,\\
				\int |\nabla\rho||v||\nabla v|^2&\leq C\int |v|^2|\nabla v|^2dx+C\int |\nabla\rho|^2|\nabla v|^2dx. 
			\end{split}
		\end{align*}
		We now estimate the terms on the right-hand side of \eqref{I_8,}. By applying
		\eqref{I_11'} to the third term,
		\eqref{I_3'} to the fourth,
		\eqref{I_2'} to the fifth,
		\eqref{42-2'} to the sixth,
		and combining \eqref{I_11'} with \eqref{I_1^2'} for the seventh term,
		we obtain
		\begin{align}\label{I_8'}
			\begin{split}
				I_8\leq \frac{7\varepsilon_0}{32}\|\nabla^3 \rho^\alpha\|_{L^2}^2+\frac{5A_1}{32}\|\nabla^2 v\|_{L^2}^2+C\|\nabla^2 \rho\|_{L^2}^2+C\|\nabla \rho\|_{L^4}^4+C\|\nabla\rho\|_{L^2}^2+C\|\nabla v\|_{L^2}^2.
			\end{split}
		\end{align}
		
		Collecting the estimates \eqref{I_11'}, \eqref{I_1^2'}, \eqref{I_2'}, \eqref{I_3'}, \eqref{I_4'}, \eqref{I_5+I_6'}, \eqref{I_7'}, and \eqref{I_8'}, and inserting them into \eqref{I_0}, together with an application of \eqref{14-1}, we arrive at
		\begin{align}\label{I_0''}
			\begin{split}
				&\quad\frac{c\alpha}{2}\frac{d}{dt}\int \rho^{\alpha-1}|\Delta\rho|^2dx+\frac{A_1}{\varepsilon_1}\frac{d}{dt}\int\rho^\alpha\mathbb{F}[v]:\nabla vdx\\
				&+\frac{1}{4}\int |\nabla\rho_t|^2dx+\frac{9\varepsilon_0}{16}\int |\nabla\Delta\rho^\alpha|^2dx+\frac{A_1}{2\varepsilon_1}\int \rho |v_t|^2dx+\frac{3A_1}{4}\int|\nabla^2 v|^2dx\\
				&\leq C\int|\nabla\rho|^2|\nabla^2 \rho|^2dx+C\|\nabla^2 \rho\|_{L^2}^2+C\|\nabla \rho\|_{L^4}^4+C\|\nabla\rho\|_{L^2}^2+C\|\nabla v\|_{L^2}^2\\
				&\leq A_4\int\rho^{\alpha-1}|\nabla^2\rho|^2|\nabla \rho|^2dx+C\|\nabla^2 \rho\|_{L^2}^2+C\|\nabla \rho\|_{L^4}^4+C\|\nabla\rho\|_{L^2}^2+C\|\nabla v\|_{L^2}^2,
			\end{split}
		\end{align}
		where $A_4$ denotes the coefficient of the first term on the right-hand side of the above inequality, depending only on the quantities specified in Proposition~\ref{Prop 8.2}.
		
		To absorb the first term on the right-hand side of \eqref{I_0''}, we multiply \eqref{ggg0} by $4A_4/\eta_\alpha$, add the resulting inequality to \eqref{I_0''}, and invoke \eqref{14-1} to obtain
		\begin{align}\label{I_0'''}
			\begin{split}
				&\quad\frac{c\alpha}{2}\frac{d}{dt}\int \rho^{\alpha-1}|\Delta\rho|^2dx+\frac{A_1}{\varepsilon_1}\frac{d}{dt}\int\rho^\alpha\mathbb{F}[v]:\nabla vdx+\frac{A_4}{\eta_{\alpha}c}\frac{d}{dt}\int |\nabla\rho|^4 dx \\
				&+\frac{1}{4}\int |\nabla\rho_t|^2dx+\frac{9\varepsilon_0}{16}\int |\nabla\Delta\rho^\alpha|^2dx+\frac{A_1}{2\varepsilon_1}\int \rho |v_t|^2dx+\frac{3A_1}{4}\int|\nabla^2 v|^2dx\\
				&\quad+A_4\int\rho^{\alpha-1}|\nabla^2\rho|^2|\nabla \rho|^2dx\\
				&\leq C\int |v|^2|\nabla\rho|^4dx+C\int |\nabla\rho|^2|\nabla v|^2dx+C\|\nabla^2 \rho\|_{L^2}^2+C\|\nabla \rho\|_{L^4}^4+C\|\nabla\rho\|_{L^2}^2+C\|\nabla v\|_{L^2}^2.
			\end{split}
		\end{align}
		To estimate the first two terms on the right-hand side of \eqref{I_0'''}, we invoke \eqref{42-2'} for the first one and combine \eqref{I_11'} with \eqref{I_1^2'} for the second, which gives
		\begin{align}\label{I_0''''}
			\begin{split}
				&\quad\frac{c\alpha}{2}\frac{d}{dt}\int \rho^{\alpha-1}|\Delta\rho|^2dx+\frac{A_1}{\varepsilon_1}\frac{d}{dt}\int\rho^\alpha\mathbb{F}[v]:\nabla vdx+\frac{A_4}{\eta_{\alpha}c}\frac{d}{dt}\int |\nabla\rho|^4 dx \\
				&+\frac{1}{4}\int |\nabla\rho_t|^2dx+\frac{15\varepsilon_0}{32}\int |\nabla\Delta\rho^\alpha|^2dx+\frac{A_1}{2\varepsilon_1}\int \rho |v_t|^2dx+\frac{11A_1}{16}\int|\nabla^2 v|^2dx\\
				&\quad+A_4\int\rho^{\alpha-1}|\nabla^2\rho|^2|\nabla \rho|^2dx\\
				&\leq C\|\nabla^2 \rho\|_{L^2}^2+C\|\nabla \rho\|_{L^4}^4+C\|\nabla\rho\|_{L^2}^2+C\|\nabla v\|_{L^2}^2.
			\end{split}
		\end{align}
		Since $c\in[\nu,2\nu)$, $\alpha\in(2/3,1)$, and $|\divg v|\leq \sqrt{3}|\nabla v|$, we infer that
		\begin{align}\label{15-4}
			\begin{split}
				\mathbb{F}[v]:\nabla v
				&=\nu|\nabla v|^2+(\nu-c)\nabla v:(\nabla v)^\top+(\alpha-1)(2\nu-c)(\divg v)^2\\
				&\ge \nu |\nabla v|^2-(c-\nu)|\nabla v|^2-3(1-\alpha)(2\nu-c)|\nabla v|^2\\
				&\ge (3\alpha-2)(2\nu-c)|\nabla v|^2.
			\end{split}
		\end{align}
		Consequently, integrating \eqref{I_0''''} over the time and invoking \eqref{hhh00} together with \eqref{14-1}, we arrive at
		\begin{align}\label{10-7}
			\begin{split}
				&\sup_{0\leq t<\infty}\big(\norm{\nabla^2\rho}_{L^2}+\norm{\nabla v}_{L^2}\big)\\
				+\int_0^\infty \big(\norm{\nabla^3 \rho^\alpha}_{L^2}^2&+\norm{\nabla\rho_t}_{L^2}^2+\norm{v_t}_{L^2}^2+\norm{\nabla^2 v}_{L^2}^2+\norm{|\nabla\rho||\nabla^2\rho|}_{L^2}^2\big)\,dt\leq C.
			\end{split}
		\end{align}
		Combining the above estimate with \eqref{33-1} and \eqref{11-222}, we conclude that
		\begin{align*}
			\int_0^\infty\int |\nabla^3 \rho|^2dxdt
			&\leq C\int_0^\infty\int \big(|\nabla^3\rho^\alpha|^2+|\nabla^2\rho|^2|\nabla\rho|^2+|\nabla\rho|^6\big)dxdt\\
			&\leq C\int_0^\infty\int \big(|\nabla^3\rho^\alpha|^2+|\nabla^2\rho|^2|\nabla\rho|^2\big)dxdt\\
			&\leq C.
		\end{align*}
		Furthermore, it follows from $\eqref{0}_1$, H\"older's inequality, and the Sobolev embedding that
		\begin{align*}
			\|\rho_t\|_{L^2}&\leq C\|\nabla(\rho v)\|_{L^2}+C\|\Delta\rho^{\alpha}\|_{L^2}\\
			&\leq C\|\nabla\rho\|_{L^{\frac{2q}{q-2}}}\|v\|_{L^q}+C\|\rho\|_{L^\infty}\|\nabla v\|_{L^2}+C\|\nabla^2\rho\|_{L^2}+C\|\nabla\rho\|_{L^4}^2\\
			&\leq C\|\nabla\rho\|_{H^1}\|v\|_{L^q}+C\|\rho\|_{L^\infty}\|\nabla v\|_{L^2}+C\|\nabla^2\rho\|_{L^2}+C\|\nabla\rho\|_{H^1}^2,
		\end{align*}
		which, together with \eqref{14-1}, \eqref{10-7}, and \eqref{hhh00}, yields
		\begin{align}\label{10-9}
			\sup_{0\le t<\infty}\|\rho_t\|_{L^2}\leq C. 
		\end{align}
		
		This completes the proof of Proposition~\ref{Prop 8.2}.
	\end{proof}

	The following proposition establishes the temporal decay of $L^2$-norms of $\nabla^2\rho$ and $\nabla v$.
	\begin{prop}\label{Prop 8.3}
		Assume that \eqref{3d gamma} holds. Then
		\begin{align}\label{16-3}
			\lim_{t\to\infty}\big(\norm{\nabla^2 \rho(t)}_{L^2}+\norm{\nabla v(t)}_{L^2}\big)=0.
		\end{align}
	\end{prop}
	\begin{proof}
		It follows from \eqref{I_0''''} that
		\begin{align*}
			\begin{split}
				&\quad\frac{c\alpha}{2}\frac{d}{dt}\int \rho^{\alpha-1}|\Delta\rho|^2 dx+\frac{A_1}{\varepsilon_1}\frac{d}{dt}\int \rho^\alpha\mathbb{F}[v]:\nabla vdx+\frac{A_4}{\eta_{\alpha}c}\frac{d}{dt}\int |\nabla\rho|^4 dx\\
				&\leq C\norm{\nabla^2 \rho}_{L^2}^2+C\norm{\nabla \rho}_{L^4}^4+C\norm{\nabla \rho}_{L^2}^2+C\norm{\nabla v}_{L^2}^2.
			\end{split}
		\end{align*}
		For convenience, we introduce the quantity
		\begin{align*}
			g_1(t):=\int \Big(\frac{c\alpha}{2} \rho^{\alpha-1}|\Delta\rho|^2 +\frac{A_1}{\varepsilon_1} \rho^\alpha\mathbb{F}[v]:\nabla v+\frac{A_4}{\eta_{\alpha}c}|\nabla\rho|^4\Big)(t) dx.
		\end{align*}
		Then, for any $t>2$ and any $s\in[t-1,t]$, it follows that
		\begin{align*}
			g_1(t)&\leq g_1(s)+\int_s^t g'_1(\zeta)\,d\zeta\\
			&\leq g_1(s)+C\int_{t-1}^t\Big(\norm{\nabla^2 \rho}_{L^2}^2+\norm{\nabla \rho}_{L^4}^4+\norm{\nabla \rho}_{L^2}^2+\norm{\nabla v}_{L^2}^2\Big)\,d\zeta.
		\end{align*}
		Averaging the above inequality with respect to $s$ over the interval $[t-1,t]$ gives
		\begin{align*}
			g_1(t)\leq \int_{t-1}^{t} g_1(\zeta)\,d\zeta
			+C\int_{t-1}^t\Big(\norm{\nabla^2 \rho}_{L^2}^2+\norm{\nabla \rho}_{L^4}^4+\norm{\nabla \rho}_{L^2}^2+\norm{\nabla v}_{L^2}^2\Big)\,d\zeta.
		\end{align*}
		Letting $t\to\infty$ and using \eqref{14-1} and \eqref{hhh00}, we conclude that
		\begin{align*}
			g_1(t)\to 0,\qquad\text{as }t\to\infty.
		\end{align*}
		The desired conclusion \eqref{16-3} then follows immediately from \eqref{15-4} and \eqref{14-1}. This completes the proof of Proposition~\ref{Prop 8.3}.
	\end{proof}

	\subsection{Second-level higher-order estimates}
	The following proposition provides a uniform-in-time estimate for the higher-order derivatives of the effective velocity.
	\begin{prop}\label{Prop 8.5}
		Assume that \eqref{3d gamma} holds. Then there exists a constant $C>0$, depending only on
		$\gamma,\alpha,\nu,\varepsilon,E_0,M$, and the initial data, such that
		\begin{align}\label{hhh8}
			\sup_{0\leq t<\infty}\big(\norm{v_t}_{L^2}+\norm{\nabla^2 v}_{L^2}\big)+\int_0^\infty \big(\norm{\nabla v_t}_{L^2}^2+\norm{\nabla^3 v}_{L^2}^2\big)dt\leq C.
		\end{align}
	\end{prop}
	\begin{proof}
		Since $\rho>0$ on $\mathbb{T}^3\times[0,\infty)$, the system \eqref{0} can be rewritten in the form
		\begin{equation}
			\label{0'}
			\left\{
			\begin{array}{l}
				\rho_t+\div(\rho v)-c\Delta\rho^\alpha=0,\\
				v_t-\rho^{\alpha-1}\divg\mathbb{F}[v]=-u \cdot\nabla v-\gamma\rho^{\gamma-2}\nabla\rho+\alpha\rho^{\alpha-2}\nabla\rho\cdot\mathbb{F}[v].
			\end{array}
			\right.
		\end{equation}
		Differentiating \eqref{0'}$_2$ with respect to time gives
		\begin{align*}
			&\quad v_{tt}-(\rho^{\alpha-1})_t\divg \mathbb{F}[v]-\rho^{\alpha-1}\divg \mathbb{F}[v_t]\\
			&=-u_t\cdot\nabla v-u\cdot\nabla v_t-\gamma(\rho^{\gamma-2})_t\nabla\rho -\gamma\rho^{\gamma-2}\nabla\rho_t\\
			&\quad+\alpha(\rho^{\alpha-2})_t\nabla\rho\cdot\mathbb{F}[v]+\alpha\rho^{\alpha-2}\nabla\rho_t\cdot\mathbb{F}[v]+\alpha\rho^{\alpha-2}\nabla\rho\cdot\mathbb{F}[v_t].
		\end{align*}
		Testing the above equation with $v_t$, integrating over $\mathbb{T}^3$, and then applying integration by parts, Young's inequality, \eqref{14-1}, and the elliptic estimate stated in Lemma~\ref{Lema 8.2}, we infer that there exist sufficiently small positive constants $\varepsilon_3$ and $\varepsilon_4$, depending only on the quantities specified in Proposition~\ref{Prop 8.5}, such that (see the derivation of equation~(205) in Proposition~6.3 of \cite{Gu-Huang-Meng-Zhou-2026} for the details)
		\begin{align}\label{J_0}
			\begin{split}
				&\quad\frac{1}{2}\frac{d}{dt}\int |v_t|^2dx
				+\frac{\varepsilon_3}{2}\int |\nabla v_t|^2dx
				+\varepsilon_4\int |\nabla^3 v|^2dx\\
				&\leq C\int |\nabla\rho||\nabla v_t||v_t|dx
				+C\int |\rho_t||\nabla^2 v||v_t|dx
				+C\int |\nabla v||v_t|^2dx\\
				&\quad+C\int |\rho_t||\nabla\rho||\nabla v||v_t|dx
				+C\int |\rho_t||\nabla v||\nabla v_t|dx
				+C\int |v||\nabla v_t||v_t|dx\\
				&\quad+C\int |\rho_t||\nabla\rho||v_t|dx
				+C\int |\nabla\rho_t||v_t|dx\\
				&\quad+C\int |\nabla\rho|^2|v_t|^2dx
				+C\int |\nabla\rho|^2|v|^2|\nabla v|^2dx
				+C\int |\nabla v|^4dx\\
				&\quad+C\int |v|^2|\nabla^2 v|^2dx
				+C\int |\nabla\rho|^4|\nabla v|^2dx
				+C\int |\nabla^2\rho|^2|\nabla v|^2dx\\
				&\quad+C\int |\nabla\rho|^2|\nabla^2 v|^2dx
				+C\int |\nabla\rho|^4dx
				+C\int |\nabla^2\rho|^2dx\\
				&=: \sum_{j=1}^{17}J_j.
			\end{split}
		\end{align}
		
		We now turn to estimating each term. For $J_1$, Hölder's inequality, Young's inequality, the Sobolev embedding theorem, along with \eqref{14-1} and \eqref{hhh0}, yields
		\begin{align}\label{J_1}
			\begin{split}
				J_1&\leq \delta\|\nabla v_t\|_{L^2}^2+C_\delta\|\nabla\rho\|_{L^\infty}^2\|v_t\|_{L^2}^2\\
				&\leq \delta\|\nabla v_t\|_{L^2}^2+C_\delta\|\nabla\rho\|_{H^2}^2\|v_t\|_{L^2}^2\\
				&\leq \delta\|\nabla v_t\|_{L^2}^2+C_\delta(1+\|\nabla^3\rho\|_{L^2}^2)\|v_t\|_{L^2}^2,
			\end{split}
		\end{align}
		where $\delta>0$ is a sufficiently small constant to be fixed later.
		
		For $J_2$, an application of Hölder's inequality, the Gagliardo--Nirenberg inequality, Young's inequality, and \eqref{hhh0} leads to
		\begin{align}\label{J_2}
			\begin{split}
				J_2&\leq C\|\rho_t\|_{L^6}\|\nabla^2 v\|_{L^3}\|v_t\|_{L^2}\\
				&\leq C\|\rho_t\|_{H^1}\|\nabla^2 v\|_{L^2}^{\frac{1}{2}}\|\nabla^2 v\|_{H^1}^{\frac{1}{2}}\|v_t\|_{L^2}\\
				&\leq \delta\|\nabla^3 v\|_{L^2}^2+C_\delta\|\nabla^2 v\|_{L^2}^2+C_\delta\|\rho_t\|_{H^1}^2\|v_t\|_{L^2}^2\\
				&\leq  \delta\|\nabla^3 v\|_{L^2}^2+C_\delta\|\nabla^2 v\|_{L^2}^2+C_\delta(1+\|\nabla\rho_t\|_{L^2}^2)\|v_t\|_{L^2}^2.
			\end{split}
		\end{align}
		
		For $J_3$, the combination of Hölder's inequality, the Gagliardo--Nirenberg inequality, \eqref{hhh0}, and Young's inequality yields
		\begin{align}\label{J_3}
			\begin{split}
				J_3&\leq C\|\nabla v\|_{L^2}\|v_t\|_{L^4}^2\\
				&\leq C\|v_t\|_{L^2}^{\frac{1}{2}}\|v_t\|_{H^1}^{\frac{3}{2}}\\
				&\leq \delta\|v_t\|_{H^1}^2+C_\delta\|v_t\|_{L^2}^2.
			\end{split}
		\end{align}
		
		For $J_4$, combining H\"older's inequality, the Sobolev embedding theorem, \eqref{14-1}, \eqref{hhh0}, and Young's inequality, we derive
		\begin{align}\label{J_4}
			\begin{split}
				J_4&\leq C\|\rho_t\|_{L^2}\|\nabla\rho\|_{L^\infty}\|\nabla v\|_{L^\infty}\|v_t\|_{L^2}\\
				&\leq C\|\nabla\rho\|_{H^2}\|\nabla v\|_{H^2}\|v_t\|_{L^2}\\
				&\leq \delta\|\nabla v\|_{H^2}^2+C_\delta\|\nabla \rho\|_{H^2}^2\|v_t\|_{L^2}^2\\
				&\leq \delta\|\nabla v\|_{H^2}^2+C_\delta(1+\|\nabla^3\rho\|_{L^2}^2)\|v_t\|_{L^2}^2. 
			\end{split}
		\end{align}
		
		For $J_5$, by H\"older's inequality, the Gagliardo--Nirenberg inequality, Young's inequality, and \eqref{hhh0}, we obtain
		\begin{align}\label{J_5}
			\begin{split}
				J_5&\leq C\|\rho_t\|_{L^2}\|\nabla v\|_{L^\infty}\|\nabla v_t\|_{L^2}\\
				&\leq C\|\nabla v\|_{L^2}^{\frac{1}{4}}\|\nabla v\|_{H^2}^{\frac{3}{4}}\|\nabla v_t\|_{L^2}\\
				&\leq \delta\|\nabla v_t\|_{L^2}^2+\delta\|\nabla v\|_{H^2}^2+C_\delta\|\nabla v\|_{L^2}^2.
			\end{split}
		\end{align}
		
		For $J_6$, by H\"older's inequality, the Gagliardo--Nirenberg inequality, \eqref{hhh0}, \eqref{hhh00}, and Young's inequality, we obtain
		\begin{align}\label{J_6}
			\begin{split}
				J_6&\leq C\|v\|_{L^6}\|\nabla v_t\|_{L^2}\|v_t\|_{L^{3}}\\
				&\leq C\|\nabla v_t\|_{L^2}\|v_t\|_{L^{2}}^{\frac{1}{2}}\|v_t\|_{H^1}^{\frac{1}{2}}\\
				&\leq \delta\|v_t\|_{H^1}^2+C_\delta\|v_t\|_{L^2}^2.
			\end{split}
		\end{align}
		
		For $J_7-J_9$, employing H\"older's inequality, Young's inequality, the Sobolev embedding, \eqref{14-1}, and \eqref{hhh0}, we obtain 
		\begin{align}\label{J_7-9}
			\begin{split}
				\sum_{j=7}^{9}J_j&\leq C\|\rho_t\|_{L^2}\|\nabla\rho\|_{L^\infty}\|v_t\|_{L^2}+C\|\nabla\rho_t\|_{L^2}\|v_t\|_{L^2}+C\|\nabla\rho\|_{L^\infty}^2\|v_t\|_{L^2}^2\\
				&\leq C\|\nabla\rho\|_{H^2}^2+C\|v_t\|_{L^2}^2+C\|\nabla\rho_t\|_{L^2}^2+C\|\nabla\rho\|_{H^2}^2\|v_t\|_{L^2}^2\\
				&\leq C\|\nabla\rho\|_{H^2}^2+C\|v_t\|_{L^2}^2+C\|\nabla\rho_t\|_{L^2}^2+C(1+\|\nabla^3\rho\|_{L^2}^2)\|v_t\|_{L^2}^2.
			\end{split}
		\end{align}
		
		For the remaining terms $J_{10}$--$J_{17}$, by H\"older's inequality, the Sobolev embedding theorem, Young's inequality, together with \eqref{hhh0} and \eqref{hhh00}, we derive the following estimates:
		\begin{align}\label{J_10-17}
			\begin{split}
				\sum_{j=10}^{17}J_j&\leq C\norm{\nabla\rho}_{L^6}^2\norm{v}_{L^6}^2\norm{\nabla v}_{L^6}^2+C\norm{\nabla v}_{L^4}^4+C\norm{v}_{L^6}^2\norm{\nabla^2 v}_{L^3}^2+C\norm{\nabla\rho}_{L^6}^4\norm{\nabla v}_{L^6}^2\\
				&\quad +C\norm{\nabla^2\rho}_{L^2}^2\norm{\nabla v}_{L^\infty}^2+C\norm{\nabla \rho}_{L^6}^2\norm{\nabla^2 v}_{L^3}^2+C\norm{\nabla\rho}_{L^4}^4+C\norm{\nabla^2 \rho}_{L^2}^2\\
				&\leq C\norm{\nabla\rho}_{H^1}^2\norm{v}_{H^1}^2\norm{\nabla v}_{H^1}^2+C\norm{\nabla v}_{L^2}^2\norm{\nabla v}_{L^\infty}^2+C\norm{v}_{H^1}^2\norm{\nabla^2 v}_{L^3}^2+C\norm{\nabla\rho}_{H^1}^4\norm{\nabla v}_{H^1}^2\\
				&\quad +C\norm{\nabla^2\rho}_{L^2}^2\norm{\nabla v}_{L^\infty}^2+C\norm{\nabla \rho}_{L^6}^2\norm{\nabla^2 v}_{L^3}^2+C\norm{\nabla\rho}_{L^4}^4+C\norm{\nabla^2 \rho}_{L^2}^2\\
				&\leq C\norm{\nabla v}_{H^1}^2+C\norm{\nabla v}_{L^\infty}^2+C\norm{\nabla^2 v}_{L^3}^2+C\norm{\nabla\rho}_{L^4}^4+C\norm{\nabla^2 \rho}_{L^2}^2\\
				&\leq \delta\norm{\nabla v}_{H^2}^2+C_\delta\norm{\nabla v}_{H^1}^2+C\norm{\nabla\rho}_{L^4}^4+C\norm{\nabla^2 \rho}_{L^2}^2.
			\end{split}
		\end{align}
		
		Substituting the estimates \eqref{J_1}--\eqref{J_10-17} into \eqref{J_0}, and choosing $\delta>0$ sufficiently small, we arrive at
		\begin{align}\label{J_0'}
			\begin{split}
				&\quad\frac{1}{2}\frac{d}{dt}\int v_t^2dx+\frac{ \varepsilon_3}{4}\int |\nabla v_t|^2dx+\frac{\varepsilon_4}{2}\int |\nabla^3 v|^2dx\\
				&\leq C\big(\norm{\nabla^3\rho}_{L^2}^2+\norm{\nabla\rho_t}_{L^2}^2\big)\norm{v_t}_{L^2}^2\\
				&\quad+C\big(\norm{\nabla v}_{H^1}^2+\norm{v_t}_{L^2}^2+\norm{\nabla\rho}_{H^2}^2+\norm{\nabla\rho_t}_{L^2}^2+\norm{\nabla\rho}_{L^4}^4\big).
			\end{split}
		\end{align}
		Applying Gr\"onwall's inequality together with \eqref{hhh00} and \eqref{hhh0}, we obtain
		\begin{align}\label{17-1}
			\sup_{0\leq t<\infty}\norm{v_t}_{L^2}+\int_0^\infty \big(\norm{\nabla v_t}_{L^2}^2+\norm{\nabla^3 v}_{L^2}^2\big)dt\leq C.
		\end{align}
		
		It remains to estimate $\norm{\nabla^2 v}_{L^2}$. Since $v$ satisfies the elliptic system \eqref{ell sys}, it follows from \eqref{14-1}, H\"older's inequality, \eqref{hhh00}, and \eqref{hhh0} that
		\begin{align}\label{100-2}
			\begin{split}
				\norm{\nabla^2 v}_{L^2}
				&\leq C\norm{v_t}_{L^2}+C\norm{v\cdot\nabla v}_{L^2}+C\norm{\nabla\rho\cdot\nabla v}_{L^2}+C\norm{\nabla\rho}_{L^2}\\
				&\leq C\norm{v_t}_{L^2}+C\norm{v}_{L^6}\norm{\nabla v}_{L^3}+C\norm{\nabla\rho}_{L^6}\norm{\nabla v}_{L^3}+C\norm{\nabla\rho}_{L^2}\\
                &\leq C\norm{v_t}_{L^2}+C\norm{\nabla v}_{L^3}+C\norm{\nabla\rho}_{L^2}\\
				&\leq C\norm{v_t}_{L^2}+C\norm{\nabla v}_{L^2}^{\frac{1}{2}}\norm{\nabla v}_{H^1}^{\frac{1}{2}}+C\norm{\nabla\rho}_{L^2}\\
				&\leq \frac{1}{2}\norm{\nabla^2 v}_{L^2}+C\norm{v_t}_{L^2}+C\norm{\nabla v}_{L^2}+C\norm{\nabla\rho}_{L^2}.
			\end{split}
		\end{align}
		Therefore, combining the above estimate with \eqref{17-1}, \eqref{hhh0}, and \eqref{hhh00}, we deduce that
		\begin{align*}
			\sup_{0\leq t<\infty}\norm{\nabla^2 v}_{L^2}\leq C.
		\end{align*}
		
		This completes the proof of Proposition~\ref{Prop 8.5}.
	\end{proof}
	
	\begin{prop}\label{Prop 8.4}
		Assume that \eqref{3d gamma} holds. Then
		\begin{align}\label{100-1}
			\lim_{t\to\infty}\norm{\nabla^2 v(t)}_{L^2}=0.
		\end{align}
	\end{prop}
	\begin{proof}
		We first establish the following claim:
		\begin{align}\label{17-2}
			\lim\limits_{t\to\infty}\norm{v_t(t)}_{L^2}=0.
		\end{align}
		Indeed, it follows from \eqref{J_0'} and \eqref{hhh8} that
		\begin{align*}
			\begin{split}
				\frac{1}{2}\frac{d}{dt}\int v_t^2dx
				\leq C\Big(\norm{\nabla v}_{H^1}^2+\norm{v_t}_{L^2}^2+\norm{\nabla\rho}_{H^2}^2+\norm{\nabla\rho_t}_{L^2}^2+\norm{\nabla\rho}_{L^4}^4\Big).
			\end{split}
		\end{align*}
		For convenience, we introduce the quantity
		\[
		g_2(t):=\frac{1}{2}\norm{v_t(t)}_{L^2}^2.
		\]
		Then, for any $t>2$ and any $s\in[t-1,t]$, we have
		\begin{align*}
			g_2(t)
			&\leq g_2(s)+\int_s^t g'_2(\zeta)\,d\zeta\\
			&\leq g_2(s)
			+C\int_{t-1}^t\Big(\norm{\nabla v}_{H^1}^2+\norm{v_t}_{L^2}^2+\norm{\nabla\rho}_{H^2}^2+\norm{\nabla\rho_t}_{L^2}^2+\norm{\nabla\rho}_{L^4}^4\Big)\,d\zeta.
		\end{align*}
		Averaging the above inequality with respect to $s$ over the interval $[t-1,t]$, we obtain
		\begin{align*}
			g_2(t)
			\leq \int_{t-1}^{t} g_2(\zeta)\,d\zeta
			+C\int_{t-1}^t\Big(\norm{\nabla v}_{H^1}^2+\norm{v_t}_{L^2}^2+\norm{\nabla\rho}_{H^2}^2+\norm{\nabla\rho_t}_{L^2}^2+\norm{\nabla\rho}_{L^4}^4\Big)\,d\zeta.
		\end{align*}
		Letting $t\to\infty$ and using \eqref{hhh0} together with \eqref{hhh00}, we conclude that
		\begin{align*}
			g_2(t)\to 0,\qquad\text{as }t\to\infty,
		\end{align*}
		which proves the claim.
		
		Finally, combining the decay estimates \eqref{17-2}, \eqref{16-3}, and \eqref{16-4} with the elliptic estimate \eqref{100-2}, we immediately obtain \eqref{100-1}. This completes the proof.
	\end{proof}

	The following proposition provides a uniform-in-time estimate for the higher-order derivatives of the density.
	\begin{prop}\label{Prop 8.7}
		Assume that \eqref{3d gamma} holds. Then there exists a constant $C>0$, depending only on
		$\gamma,\alpha,\nu,\varepsilon,E_0,M$, and the initial data, such that
		\begin{align}\label{hhh9}
			\sup_{0\leq t<\infty}\big(\norm{\nabla^3 \rho}_{L^2}+\norm{\nabla\rho_t}_{L^2}\big)+\int_0^\infty \big(\norm{\nabla^4 \rho}_{L^2}^2+\norm{\nabla^2 \rho_t}_{L^2}^2+\norm{\rho_{tt}}_{L^2}^2\big)dt\leq C.
		\end{align}
	\end{prop}
	\begin{proof}
		Differentiating $\eqref{0'}_1$ with respect to $t$ gives
		\begin{align}\label{17-3}
			\rho_{tt}+(\divg(\rho v))_t-c\alpha(\alpha-1)(\rho^{\alpha-2}|\nabla\rho|^2)_t-c\alpha(\rho^{\alpha-1}\Delta\rho)_t=0.
		\end{align}
		Testing the above equation with $\rho_{tt}$, integrating over $\mathbb{T}^3$, and applying Young's inequality together with \eqref{14-1}, while taking into account the identity
		\begin{align*}
			\Delta\rho_t
			=\frac{1}{c\alpha\rho^{\alpha-1}}
			\big(\rho_{tt}+(\divg (\rho v))_t
			-c\alpha(\alpha-1)(\rho^{\alpha-2}|\nabla\rho|^2)_t
			-c\alpha(\rho^{\alpha-1})_t\Delta\rho\big),
		\end{align*}
		we conclude that there exists a sufficiently small constant $\varepsilon_5>0$, depending only on the quantities specified in Proposition~\ref{Prop 8.7}, such that (see the derivation of equation~(221) in Proposition~6.4 of \cite{Gu-Huang-Meng-Zhou-2026} for the detailed computations)
		\begin{align}\label{K_0}
			\begin{split}
				&\quad\frac{c\alpha}{2}\frac{d}{dt}\int \rho^{\alpha-1}|\nabla\rho_t|^2dx+\frac{1}{4}\int \rho_{tt}^2dx+\varepsilon_5\int |\nabla^2 \rho_t|^2dx\\
				&\leq C\int |\rho_t||\nabla\rho_t|^2dx
				+C\int |\rho_t|^2|\Delta\rho|^2dx
				+C\int |\rho_t|^2|\nabla\rho|^4dx
				+C\int |\nabla\rho|^2|\nabla\rho_t|^2dx\\
				&\quad
				+C\int |\nabla\rho_t|^2|v|^2dx
				+C\int|\nabla\rho|^2|v_t|^2dx
				+C\int |\rho_t|^2|\nabla v|^2dx
				+C\int |\nabla v_t|^2dx\\
				&=:\sum_{k=1}^{8}K_k.
			\end{split}
		\end{align}
		
		It remains to estimate each term. For $K_1$, by H\"older's inequality, the Gagliardo--Nirenberg inequality, \eqref{hhh0}, and Young's inequality, we have
		\begin{align}\label{K_1}
			\begin{split}
				K_1&\leq C\|\rho_t\|_{L^2}\|\nabla\rho_t\|_{L^4}^2\\
				&\leq C\|\nabla\rho_t\|_{L^2}^{\frac{1}{2}}\|\nabla\rho_t\|_{H^1}^{\frac{3}{2}}\\
				&\leq \delta\|\nabla^2\rho_t\|_{L^2}^2+C_\delta\|\nabla\rho_t\|_{L^2}^2,
			\end{split}
		\end{align}
		where $\delta>0$ is a sufficiently small constant to be specified later.
		
		For $K_2-K_4$, by H\"older's inequality, the Sobolev embedding, \eqref{14-1}, \eqref{hhh00} and \eqref{hhh0}, 
		\begin{align}\label{K_2}
			\begin{split}
				\sum_{k=2}^{4}K_k&\leq C\|\rho_t\|_{L^6}^2\|\Delta\rho\|_{L^6}^2+C\|\rho_t\|_{L^6}^2\|\nabla\rho\|_{L^6}^4+C\|\nabla\rho\|_{L^\infty}^2\|\nabla\rho_t\|_{L^2}^2\\
				&\leq C\|\rho_t\|_{H^1}^2\|\nabla^2 \rho\|_{H^1}^2+C\|\rho_t\|_{H^1}^2\|\nabla\rho\|_{H^1}^4+C\|\nabla\rho\|_{H^2}^2\|\nabla\rho_t\|_{L^2}^2\\
				&\leq C(1+\norm{\nabla\rho_t}_{L^2}^2)\norm{\nabla^2\rho}_{H^1}^2+C(1+\norm{\nabla\rho_t}_{L^2}^2)\norm{\nabla\rho}_{H^1}^2+C\|\nabla\rho\|_{H^2}^2\|\nabla\rho_t\|_{L^2}^2\\
				&\le C(1+\norm{\nabla\rho_t}_{L^2}^2)\norm{\nabla\rho}_{H^2}^2. 
			\end{split}
		\end{align}
		
		For $K_5-K_8$, we apply H\"older's inequality, the Sobolev embedding, \eqref{14-1}, \eqref{hhh00}, \eqref{hhh0}, and \eqref{hhh8} to obtain the following estimates:
		\begin{align}\label{K_5}
			\begin{split}
				\sum_{k=5}^{8}K_k&\leq C\|\nabla\rho_t\|_{L^2}^2\|v\|_{L^\infty}^2+C\|\nabla\rho\|_{L^\infty}^2\|v_t\|_{L^2}^2+C\|\rho_t\|_{L^6}^2\|\nabla v\|_{L^6}^2+C\|\nabla v_t\|_{L^2}^2\\
				&\leq C\|\nabla\rho_t\|_{L^2}^2\|v\|_{H^2}^2+C\|\nabla\rho\|_{H^2}^2\|v_t\|_{L^2}^2+C\|\rho_t\|_{H^1}^2\|\nabla v\|_{H^1}^2+C\|\nabla v_t\|_{L^2}^2\\
				&\leq C(\norm{\nabla\rho_t}_{L^2}^2+\norm{\nabla\rho}_{H^2}^2+\norm{\nabla v}_{H^1}^2+\norm{\nabla v_t}_{L^2}^2).
			\end{split}
		\end{align}
		
		Substituting the estimates \eqref{K_1}--\eqref{K_5} into \eqref{K_0}, and choosing $\delta>0$ sufficiently small, we arrive at
		\begin{align}\label{100-4}
			\begin{split}
				&\quad\frac{c\alpha}{2}\frac{d}{dt}\int \rho^{\alpha-1}|\nabla\rho_t|^2dx+\frac{1}{4}\int \rho_{tt}^2dx+\frac{\varepsilon_5}{2}\int |\nabla^2 \rho_t|^2dx\\
				&\le C\norm{\nabla\rho}_{H^2}^2\norm{\nabla\rho_t}_{L^2}^2
				+C\big(\norm{\nabla\rho_t}_{L^2}^2+\norm{\nabla\rho}_{H^2}^2+\norm{\nabla v}_{H^1}^2+\norm{\nabla v_t}_{L^2}^2\big).
			\end{split}
		\end{align}
		Applying Gr\"onwall's inequality in conjunction with \eqref{14-1}, \eqref{hhh00}, \eqref{hhh0}, and \eqref{hhh8}, we deduce that
		\begin{align}\label{17-8}
			\sup_{0\leq t<\infty}\norm{\nabla\rho_t}_{L^2}
			+\int_0^\infty\big(\norm{\rho_{tt}}_{L^2}^2+\norm{\nabla^2\rho_t}_{L^2}^2\big)dt
			\leq C.
		\end{align}
		
		We next derive a uniform-in-time estimate for $\|\nabla^3\rho\|_{L^2}$. Since $\rho>0$ on $\mathbb{T}^3\times[0,\infty)$, equation \eqref{hhh0.5} can be rewritten as
		\begin{align}\label{17-8.5}
			\Delta\rho=\frac{1}{c\alpha\rho^{\alpha-1}}\big(\rho_t+\divg(\rho v)-c\alpha(\alpha-1)\rho^{\alpha-2}|\nabla\rho|^2\big).
		\end{align}
		Applying the standard $L^p$ estimates for elliptic equations, together with H\"older's inequality, the Sobolev embedding,  \eqref{14-1}, \eqref{hhh0}, and \eqref{hhh8}, we obtain
		\begin{align}\label{100-7}
			\begin{split}
				\int |\nabla^3\rho|^2dx&\leq C\int |\nabla\rho_t|^2dx+C\int |\nabla\rho|^2|\rho_t|^2dx+C\int |\nabla\rho|^2|\nabla(\rho v)|^2dx\\
				&\quad +C\int |\nabla^2(\rho v)|^2dx+C\int |\nabla\rho|^6dx+C\int |\nabla\rho|^2|\nabla^2\rho|^2dx\\
				&\leq C\|\nabla\rho_t\|_{L^2}^2+C\|\nabla\rho\|_{L^\infty}^2\|\rho_t\|_{L^2}^2+C\|\nabla\rho\|_{L^4}^4\|v\|_{L^\infty}^2+C\|\nabla\rho\|_{L^6}^2\|\nabla v\|_{L^6}^2\\
				&\quad +C\|\nabla^2 v\|_{L^2}^2+C\|\nabla^2 \rho\|_{L^2}^2\|v\|_{L^\infty}^2+C\|\nabla\rho\|_{L^6}^6+C\|\nabla\rho\|_{L^\infty}^2\|\nabla^2 \rho\|_{L^2}^2\\
				&\leq C\|\nabla\rho_t\|_{L^2}^2+C\|\nabla\rho\|_{L^\infty}^2\|\rho_t\|_{L^2}^2+C\|\nabla\rho\|_{H^1}^4\|v\|_{H^2}^2+C\|\nabla\rho\|_{H^1}^2\|\nabla v\|_{H^1}^2\\
				&\quad +C\|\nabla^2 v\|_{L^2}^2+C\|\nabla^2 \rho\|_{L^2}^2\|v\|_{H^2}^2+C\|\nabla\rho\|_{H^1}^6+C\|\nabla\rho\|_{L^\infty}^2\|\nabla^2 \rho\|_{L^2}^2\\
				&\leq C\|\nabla\rho\|_{L^\infty}^2+C\|\nabla\rho_t\|_{L^2}^2+C\|\nabla\rho\|_{H^1}^2+C\|\nabla^2 v\|_{L^2}^2\\
				&\leq \frac{1}{2}\|\nabla^3\rho\|_{L^2}^2+C\|\nabla\rho_t\|_{L^2}^2+C\|\nabla\rho\|_{H^1}^2+C\|\nabla^2 v\|_{L^2}^2,
			\end{split}
		\end{align}
		where the last inequality follows from the Gagliardo--Nirenberg inequality and  Young's inequality. Consequently, combining the above estimate with \eqref{17-8}, \eqref{hhh00}, \eqref{hhh0}, and \eqref{hhh8}, we conclude that
		\begin{align}\label{17-9}
			\sup_{0\leq t<\infty}\|\nabla^3 \rho\|_{L^2}\leq C.
		\end{align}

		Finally, we derive an $L^2(0,\infty;L^2)$ estimate for $\nabla^4\rho$. Applying the standard elliptic estimates to \eqref{17-8.5}, together with H\"older's inequality, gives
		\begin{align}\label{17-10}
			\begin{split}
				\int |\nabla^4 \rho|^2dx&\leq C\int (|\nabla^2\rho|^2+|\nabla\rho|^4)|\rho_t|^2dx+C\int |\nabla\rho|^2|\nabla\rho_t|^2dx+C\int |\nabla^2\rho_t|^2dx\\
				&\quad +C\int (|\nabla^2\rho|^2+|\nabla\rho|^4)|\nabla(\rho v)|^2dx+C\int |\nabla\rho|^2|\nabla^2(\rho v)|^2dx+C\int |\nabla^3(\rho v)|^2dx\\
				&\quad+C\int (|\nabla^2\rho|^2+|\nabla\rho|^4)|\nabla\rho|^4dx+C\int (|\nabla^2\rho|^4+|\nabla\rho|^2|\nabla^3\rho|^2)dx\\
				&\leq C\big(\|\nabla^2\rho\|_{L^6}^2+\|\nabla\rho\|_{L^\infty}^4\big)\big(\|\rho_t\|_{L^6}^2+\|\nabla (\rho v)\|_{L^6}^2+\|\nabla\rho\|_{L^\infty}^4\big)\\
				&\quad +C\|\nabla\rho\|_{L^\infty}^2\big(\|\nabla\rho_t\|_{L^2}^2+\|\nabla^2(\rho v)\|_{L^2}^2\big)\\
				&\quad +C\big(\|\nabla^2\rho_t\|_{L^2}^2+\|\nabla^3(\rho v)\|_{L^2}^2+\|\nabla^2\rho\|_{L^4}^4+\|\nabla\rho\|_{L^\infty}^2\|\nabla^3\rho\|_{L^2}^2\big).
			\end{split}
		\end{align}
		The following estimates are consequences of \eqref{14-1}, \eqref{hhh00}, \eqref{hhh0}, \eqref{hhh8}, \eqref{17-8}, \eqref{17-9}, and the Sobolev embedding:
		\begin{align*}
			\|\nabla^2\rho\|_{L^6}^2+\|\nabla\rho\|_{L^\infty}^4
			&\leq C\|\nabla\rho\|_{H^2}^2,\\
			\|\rho_t\|_{L^6}^2+\|\nabla (\rho v)\|_{L^6}^2+\|\nabla\rho\|_{L^\infty}^4
			&\leq C\big(\|\rho_t\|_{H^1}^2+\|\rho\|_{H^2}^2\|v\|_{H^2}^2+\|\nabla\rho\|_{H^2}^2\big)\leq C,\\
			\|\nabla\rho\|_{L^\infty}^2\big(\|\nabla\rho_t\|_{L^2}^2+\|\nabla^2(\rho v)\|_{L^2}^2\big)
			&\leq C\|\nabla\rho\|_{H^2}^2\big(\|\nabla\rho_t\|_{L^2}^2+\|\rho\|_{H^2}^2\|v\|_{H^2}^2\big)\leq C\|\nabla\rho\|_{H^2}^2,\\
			\|\nabla^3(\rho v)\|_{L^2}^2
			&\leq C\|\rho\|_{L^\infty}^2\|\nabla^3 v\|_{L^2}^2+C\|\nabla^2\rho\|_{L^6}^2\|\nabla v\|_{L^6}^2\\
			&\quad +C\|\nabla\rho\|_{L^6}^2\|\nabla^2 v\|_{L^6}^2+C\|\nabla^3\rho\|_{L^2}^2\|v\|_{L^\infty}^2\\
			&\le C\big(\|\nabla v\|_{H^2}^2+\|\nabla^3\rho\|_{L^2}^2\big),\\
			\|\nabla^2\rho\|_{L^4}^4
			&\le C\|\nabla^2\rho\|_{H^1}^4
			\le C\|\nabla^2\rho\|_{H^1}^2,\\
			\|\nabla\rho\|_{L^\infty}^2\|\nabla^3\rho\|_{L^2}^2
			&\le C\|\nabla\rho\|_{H^2}^2.
		\end{align*}
		Substituting the above estimates into \eqref{17-10}, we obtain
		\begin{align*}
			\|\nabla^4\rho\|_{L^2}^2
			\leq C\big(\|\nabla\rho\|_{H^2}^2+\|\nabla^2\rho_t\|_{L^2}^2+\|\nabla v\|_{H^2}^2\big).
		\end{align*}
		Integrating the above inequality over $(0,\infty)$ and invoking \eqref{hhh00}, \eqref{hhh0}, \eqref{hhh8}, and \eqref{17-8}, we conclude that
		\begin{align}\label{17-11}
			\int_0^\infty\|\nabla^4\rho\|_{L^2}^2\,dt\leq C.
		\end{align}
		
		This completes the proof of Proposition~\ref{Prop 8.7}.
	\end{proof}
	\begin{prop}\label{Prop 8.8}
		Assume that \eqref{3d gamma} holds. Then
		\begin{align}\label{100-3}
			\lim_{t\to\infty}\norm{\nabla^3 \rho(t)}_{L^2}=0.
		\end{align}
	\end{prop}
	\begin{proof}
		We first establish the following decay property:
		\begin{align}\label{100-6}
			\lim\limits_{t\to\infty}\norm{\nabla\rho_t(t)}_{L^2}=0.
		\end{align}
		Indeed, it follows from \eqref{100-4} and \eqref{hhh9} that
		\begin{align}\label{100-5}
			\begin{split}
				\frac{c\alpha}{2}\frac{d}{dt}\int \rho^{\alpha-1}|\nabla\rho_t|^2dx
				\le C\big(\norm{\nabla\rho_t}_{L^2}^2+\norm{\nabla\rho}_{H^2}^2+\norm{\nabla v}_{H^1}^2+\norm{\nabla v_t}_{L^2}^2\big).
			\end{split}
		\end{align}
		For convenience, we define
		\begin{align*}
			g_3(t):=\frac{c\alpha}{2}\int \rho^{\alpha-1}|\nabla\rho_t|^2(t)\,dx.
		\end{align*}
		Then, for any $t>2$ and any $s\in[t-1,t]$, we have
		\begin{align*}
			g_3(t)
			&\leq g_3(s)+\int_s^t g_3'(\zeta)\,d\zeta\\
			&\leq g_3(s)
			+C\int_{t-1}^t\big(\norm{\nabla\rho_t}_{L^2}^2+\norm{\nabla\rho}_{H^2}^2+\norm{\nabla v}_{H^1}^2+\norm{\nabla v_t}_{L^2}^2\big)\,d\zeta.
		\end{align*}
		Averaging the above inequality with respect to $s$ over the interval $[t-1,t]$, we obtain
		\begin{align*}
			g_3(t)
			\leq \int_{t-1}^{t} g_3(\zeta)\,d\zeta
			+C\int_{t-1}^t\big(\norm{\nabla\rho_t}_{L^2}^2+\norm{\nabla\rho}_{H^2}^2+\norm{\nabla v}_{H^1}^2+\norm{\nabla v_t}_{L^2}^2\big)\,d\zeta.
		\end{align*}
		Letting $t\to\infty$ and invoking \eqref{14-1}, \eqref{hhh00}, \eqref{hhh0}, and \eqref{hhh8}, we conclude that
		\begin{align*}
			g_3(t)\to 0,\qquad\text{as }t\to\infty,
		\end{align*}
		which, together with the uniform bound of $\rho$, proves \eqref{100-6}.
		
		Finally, combining the decay estimates \eqref{100-6}, \eqref{16-4}, \eqref{16-3}, and \eqref{100-1} with the elliptic estimate \eqref{100-7}, we arrive at \eqref{100-3}. This completes the proof of Proposition~\ref{Prop 8.8}.
	\end{proof}
	
	\section{Proofs of the Main Theorems}
	
	In this section, we complete the proofs of Theorems \ref{Thm 1.1} and \ref{Thm 1.1'}.
	
	\subsection{Proof of Theorem \ref{Thm 1.1}}
	
	Under the assumptions of Theorem \ref{Thm 1.1} in the three-dimensional case, Propositions \ref{Prop 5.3} and \ref{Prop 7.2} establish the uniform upper and lower bounds for the density. Furthermore, Propositions \ref{Prop 8.1.5}, \ref{Prop 8.3}, \ref{Prop 8.4}, and \ref{Prop 8.8} yield
	\[
	\lim_{t\to\infty}
	\Big(
	\|\rho(t)-\rho_s\|_{H^3}
	+\|\nabla v(t)\|_{H^1}
	\Big)
	=0.
	\]
	Since $u=v-c\alpha\rho^{\alpha-2}\nabla\rho,$ we further obtain
	\[
	\lim_{t\to\infty}\|\nabla u(t)\|_{H^1}=0.
	\]
	Therefore, all the conclusions of Theorem \ref{Thm 1.1} follow, and the proof is complete.
	
	\subsection{Proof of Theorem \ref{Thm 1.1'}}
	
	Under the assumptions of Theorem \ref{Thm 1.1'} in the three-dimensional case, Proposition \ref{Prop 5.3} establishes the uniform upper bound for the density. Therefore, all the assertions of Theorem \ref{Thm 1.1'} follow immediately, which completes the proof.
	
	\section*{Acknowledgments}
	X Huang is partially supported by Chinese Academy of Sciences Project for Young Scientists in Basic Research (Grant No. YSBR-031), National Natural Science Foundation of China (Grant Nos. 12494542, 11688101) and National Key R\&D Program of China (Grant No. 2021YFA1000801). 
	
	\vspace{1cm}
	\noindent\textbf{Data availability statement.} Data sharing is not applicable to this article.
	
	\vspace{0.3cm}
	\noindent\textbf{Conflict of interest.} The authors declare that they have no conflict of interest.

\end{document}